\documentclass[11pt]{amsart}
\RequirePackage{pdf14}
\RequirePackage{fix-cm}
\usepackage[left=2.5cm,right=2.5cm,top=3.5cm]{geometry}
 \usepackage[utf8]{inputenc}

\usepackage{graphicx}
\usepackage{hyperref}
\usepackage{ytableau}
\usepackage{enumerate}
\usepackage{amsfonts,amsmath,amssymb,amsthm}
\usepackage{tikz,tikz-cd}
\usepackage[all]{xy}
\usetikzlibrary{positioning} 
\usepackage{verbatim}
\usepackage{array,setspace,mathrsfs,yfonts,dsfont,bbm,colonequals,amscd,euscript}
\usepackage{relsize,suffix,mathtools,cancel,bbm}

\makeatletter
\newcommand\mathcircled[1]{%
	\mathpalette\@mathcircled{#1}%
}
\newcommand\@mathcircled[2]{%
	\tikz[baseline=(math.base)] \node[draw,circle,inner sep=3pt] (math) {$\m@th#1#2$};%
}
\makeatother

\numberwithin{equation}{subsection}

\theoremstyle{plain}
\newtheorem{thm}{Theorem}[section]
\newtheorem{lem}[thm]{Lemma}

\newtheorem{prop}[thm]{Proposition}
\newtheorem{cor}[thm]{Corollary}
\newtheorem{warn}[thm]{Warning}

\newtheorem*{thm*}{Theorem}

\theoremstyle{definition}
\newtheorem{dfn}[thm]{Definition}

\theoremstyle{remark}
\newtheorem{rem}[thm]{Remark}

\def\IA{\mathbb{A}}

\def\IC{\mathbb{C}}

\def\IR{\mathbb{R}}

\def\IG{\mathbb{G}}
\def\IO{\mathbb{O}}

\def\AA{{\mathcal A}}

\def\DD{{\mathcal D}}

\def\II{{\mathcal I}}

\def\NN{{\mathcal N}}
\def\OO{{\mathcal O}}
\def\PP{{\mathcal P}}

\def\WW{{\mathcal W}}

\def\dim{{\rm dim}}

\def\pgl{\mathfrak{pgl}}

\def\gf{\mathfrak{g}}

\def\uf{\mathfrak{u}}

\def \nf{\mathfrak{n}}
\def \hf{\mathfrak{h}}
\def \slf{\mathfrak{sl}}
\def \glf{\mathfrak{gl}}
\def \Gf{\mathfrak{G}}
\def\tf{\mathfrak{t}}

\newcommand{\bb}{\mathbb}
\newcommand{\mc}{\mathcal}

\newcommand{\mf}{\mathfrak}

\def \eg{{\it e.g.}}
\def \ie{{\it i.e.}}

\def \id{\mathbbm{1}}
\def \ik{\mathbb{\IC}}

\newcommand{\fibre}[1]{\underset{#1}{\times}}
\newcommand{\tensor}[1]{\underset{#1}{\otimes}}

\def \Gr{\mathrm{Gr}}
\newcommand{\Grr}{\textbf{Gr}}
\newcommand{\ul}{{\underline{\lambda}}}
\newcommand{\ut}{{\underline{t}}}

\newcommand{\bd}{{\textup{BD}}}
\newcommand{\Gru}{\underline{\textup{Gr}}}
\newcommand{\Mu}{\underline{\mathcal{M}}}
\newcommand{\Wu}{\underline{\mathcal{W}}}

\DeclareMathOperator{\Spec}{Spec}
\DeclareMathOperator{\Bun}{Bun}

\DeclareMathOperator{\Hom}{Hom}

\newcommand{\ch}{{\textup{ch}}}

\newcommand{\Lie}{{\textup{Lie}}}
\newcommand{\one}{{\textbf{1}}}

\newcommand{\pt}{{\textup{pt}}}

\newcommand{\im}{{\textup{im}}}

\newcommand{\colim}{{\textup{colim}}}

\newcommand{\into}{{\ \hookrightarrow\ }}
\newcommand{\onto}{{\ \twoheadrightarrow\ }}
\newcommand{\Spf}{{\textup{Spf}\ }}
\newcommand{\triv}{{\textup{triv}}}
\newcommand{\ev}{{\text{ev}}}

\newcommand{\Ad}{{{\textup{Ad}}}}
\newcommand{\ad}{{{\textup{ad}}}}

\newcommand{\Ph}{{{[\![\hbar]\!]}}}

\newcommand{\Lt}{{{(\!(t)\!)}}}
\newcommand{\Pt}{{{[\![t]\!]}}}
\newcommand{\Lti}{{{(\!(t^{-1})\!)}}}
\newcommand{\Pti}{{{[\![t^{-1}]\!]}}}

\newcommand{\red}[1]{/\!\!/\!\!/_{#1}}

\newcommand{\open}[1]{\overset{\circ}{#1}}

\newcommand{\reg}{{{\textup{reg}}}}

\newcommand{\HC}{{{\textup{HC}}}}
\newcommand{\Sym}{{{\textup{Sym}}}}
\newcommand{\GL}{{{\textup{GL}}}}
\newcommand{\PGL}{{{\textup{PGL}}}}
\newcommand{\SL}{{{\textup{SL}}}}
\newcommand{\g}{{{\mathfrak{g}}}}
\newcommand{\h}{{{\mathfrak{h}}}}
\newcommand{\gl}{{{\mathfrak{gl}}}}

\title[Inverse Hamiltonian reduction in type A and generalized slices]{Inverse Hamiltonian reduction in type A and generalized slices in the affine Grassmannian}
\author{Dylan Butson and Sujay Nair}
\date{}

\begin{document} 
\maketitle

\begin{abstract}
We give a geometric proof of inverse Hamiltonian reduction for all finite W-algebras in type $A$, a certain embedding of the finite W-algebra corresponding to an arbitrary nilpotent in $\mathfrak{gl}_N$ into that corresponding to a larger nilpotent with respect to the closure order on orbits, tensored with an auxiliary algebra of differential operators. We first prove a classical analogue for equivariant Slodowy slices using multiplication maps on generalized slices in the affine Grassmannian, then deduce the result for equivariant finite W-algebras by Fedosov quantization. This implies the statement for finite W-algebras, as well as Kostant--Whittaker reductions of arbitrary algebras in the category of Harish Chandra bimodules, including quantizations of Moore--Tachikawa varieties.
\end{abstract}

\tableofcontents

\section{Introduction}

The finite W-algebra $\WW_f(\g)$ associated to a nilpotent element $f$ in a semisimple Lie algebra $\g$ is an associative algebra quantizing the transverse slice $S_f\subset \g$ to the nilpotent orbit $\bb O_f=G\cdot f \subset \g$. In type $A_{N-1}$, such orbits are parameterized by partitions of $N$, or equivalently tuples of positive integers $\mu=(\mu_1 \geq ... \geq \mu_{n-1} \geq 0)$ such that the total sum $|\mu|=\sum_{k=1}^{n-1} \mu_k = N$; we call the positive integer $n-1$ the length of such a partition, and identify the set of partitions of length at most $n-1$ with the set of dominant, integral coweights $\Lambda^+(\PGL_n)$ for the projective linear group $\PGL_n$.

In the case $\mu=[1^N]=(1 , ... , 1)$ for which $n-1=N$, the corresponding nilpotent $f_\mu=0 \in \gl_N$, so that the slice $S_{0}=\g$ and the finite W-algebra $\WW_0(\g)=U(\g)$ is simply the universal enveloping algebra of $\g$, which naturally controls the representation theory of $\g$. The algebras $\WW_{f_\mu}(\g)$ are a natural family of generalizations thereof, which analogously control the categories of so-called Whittaker modules for $\g$, with respect to a unipotent group $N_\mu$ and character $\chi_\mu=\langle f_\mu , \cdot \rangle_\g \in \mf n_\mu^\vee$ determined by $f_\mu$ (up to some relatively insignificant auxiliary choices).

The primary goal of this paper is to provide a geometric paradigm for proving a family of results we call \emph{inverse Hamiltonian reduction}, which can be understood most elementarily as the existence of certain highly structured embeddings of algebras. In particular, we establish the following theorem, which is essentially the statement of quantum inverse Hamiltonian reduction for finite W-algebras:

\begin{thm}\label{IRQFintrothm} Let $\mu\in \Lambda^+(\PGL_n)$ and $\alpha$ a positive coroot of $\pgl_n$ for which $\mu+\alpha \in \Lambda^+(\PGL_n)$. Then there exists an embedding of associative algebras with quantum Hamiltonian $\IG_\alpha$-action
\begin{equation}\label{IRQFintroeqn}
	 \WW_{\hbar,f_\mu}(\gl_N) \hookrightarrow D^{\rm loc}_{\hbar,\alpha}\otimes \WW_{\hbar,f_{\mu+\alpha}}(\gl_N) \ , 
\end{equation}
where $\IG_\alpha\cong \IG_a^{\times \ell(\alpha)}$ denotes a rank $\ell(\alpha)$ additive group scheme acting on $ \WW_{\hbar,f_\mu}(\gl_N)$, $D^{\rm loc}_{\hbar,\alpha}$ denotes a localization of the algebra of differential operators on $\IG_\alpha$, and $\IG_\alpha$ acts trivially on $ \WW_{\hbar,f_{\mu+\alpha}}(\gl_N)$.
\end{thm}

We now recall a semi-classical analogue of Theorem \ref{IRQFintrothm}, and explain the geometric interpretation and significance of these statements in the context of (quantum) Hamiltonian reduction:

\begin{thm}\label{IRCFintrothm} There exists an isomorphism of Poisson algebras with Hamiltonian $\IG_\alpha$-action
\[ \mc O(S_{f_\mu})[E_\alpha^{-1}] \cong  \mc O(T^*\IG_{\alpha})[\tilde E_\alpha^{-1}] \otimes \mc O(S_{f_{\mu+\alpha}}) \ , \]
where $\mc O(S_{f_\mu})[E_\alpha^{-1}] $ denotes a localization of $\mc O(S_{f_\mu})$, $\mc O(T^*\IG_{\alpha})[\tilde E_\alpha^{-1}]$ a localization of $\mc O(T^*\IG_{\alpha})$, and $\IG_\alpha$ acts on $\mc O(S_{f_\mu})$ and $\mc O(T^*\IG_{\alpha})$ preserving the localization ideals and trivially on $\mc O(S_{f_{\mu+\alpha}})$.
\end{thm}

In particular, composing with the natural map to the localization, we obtain
\[ \mc O(S_{f_\mu}) \hookrightarrow  \mc O(T^*\IG_{\alpha})[\tilde E_\alpha^{-1}] \otimes \mc O(S_{f_{\mu+\alpha}}) \]
a strongly $\IG_\alpha$-equivariant embedding of Poisson algebras analogous to Equation \ref{IRQFintroeqn}. Geometrically, the localization $\mc O(S_{f_\mu})[E_\alpha^{-1}]$ corresponds to a $\IG_\alpha$-invariant basic affine open $U_{\mu,\alpha}\subset S_{f_\mu}$, on which the Hamiltonian action of $\IG_\alpha$ is appropriately trivialized. This property is abstracted by the following notion of inverse Hamiltonian reduction: Let $H$ be an abelian group with $\mf h=\Lie(H)$, $\chi \in \mf h^*$, and $Y$ a Poisson variety with Hamiltonian $H$-action and moment map $\Phi:Y \to \h^*$.
\begin{dfn} A  (classical, finite) inverse Hamiltonian reduction of $Y\red{\chi}H$ is a pair $(U,\varphi)$ where
	\begin{itemize}
		\item $U\subset \h^*$ is an open, Poisson, subvariety containing $\chi$, and
		\item $\varphi: H\times U \times Y\red{\chi}H \xrightarrow{\cong} Y\times_{\Phi,\h^*} U$ a strongly $H$-equivariant isomorphism of Poisson varieties,
	\end{itemize}
	with respect to the canonical symplectic structure on $H\times U \subset H\times \h^*=T^*H$, as well as the natural Hamiltonian $H$-action on $H\times U$ and the trivial $H$ action on $Y\red{\chi}H$.
\end{dfn}

Note that this gives an embedding $Y\red{\chi}H\to Y$ which can be thought of as inverting the Hamiltonian reduction procedure. More precisely, an inverse Hamiltonian reduction can be understood as a relative variant of the Gan--Ginzburg isomorphism, Lemma 2.1 of \cite{GG}; it requires that the moment map level set $\Phi^{-1}(\chi)\cong H \times Y\red{\chi}H $ is a trivial H-torsor over the Hamiltonian reduction, as in \emph{loc. cit.}, and moreover that such isomorphisms can be constructed compatibly as we vary $\chi \in U$. In particular, this implies that the inclusion of the open set $Y\times_{\Phi,\h^*} U \to Y$ induces an isomorphism
\[ (Y\times_{\Phi,\h^*} U )\red{\chi}H  \xrightarrow{\cong }  Y\red{\chi}H \ . \]

In the example of Theorem \ref{IRCFintrothm}, we take $H=\IG_\alpha$ and $\h=\Gf_\alpha:=\Lie(\IG_\alpha)$, and this reduces to the statement that there is an open subset $U\subset \Gf_\alpha^*$ such that the action map $\IG_\alpha \times U \to (\overline{\bb O}_{f_{\mu+\alpha}} \cap S_{f_\mu})$ is an open immersion over which the family of transverse slices $S_{f}$ for $f\in \im( \IG_\alpha \times U) \subset (\overline{\bb O}_{f_{\mu+\alpha}} \cap S_{f_\mu}) $ is trivialized, and thus we obtain the desired open immersion
\begin{equation}\label{IRimintroeq}
	 \IG_\alpha \times U \times S_{f_{\mu+\alpha}} \to S_{f_\mu}  \ .
\end{equation}

Another perspective on this result is that it provides a Darboux chart on $S_{f_\mu}$ relative to the space $S_{f_{\mu+\alpha}}$, in the sense that choosing a Darboux chart on the latter induces such on the former, noting $\IG_\alpha\times U$ is isomorphic to an open subset of $T^*\bb A^{\ell(\alpha)}$ with the canonical symplectic form. This is the classical, finite-type analogue of the fact that our main theorem in the affine setting, described in Theorem \ref{IRQAintrothm} below, can be understood as a free field realization of $ \WW^\kappa_{f_\mu}(\gl_N)$ relative to $\WW_{f_{\mu+\alpha}}(\gl_N)$.

In fact, the proof of Theorem \ref{IRQFintrothm} follows from an analogous inverse quantum Hamiltonian reduction principle, as we now explain: the algebra $\WW_{f_\mu}(\gl_N)=\Gamma(S_{f_\mu},\mc A_{f_\mu})$ is the global sections of a sheaf of algebras $\mc A_{f_\mu}$ quantizing the structure sheaf $\mc O_{S_{f_\mu}}$, and we prove that its restriction to the image of the open immersion of Equation \ref{IRimintroeq} is given by
\[ \mc A_{f_\mu} |_{\IG_\alpha \times U \times S_{f_{\mu+\alpha}} } \xrightarrow{\cong } \mc D^{\rm loc}_\alpha\boxtimes \mc A_{f_{\mu+\alpha}}  \]
where $\mc D^{\rm loc}_\alpha$ is the sheaf of differential operators on $\IG_\alpha$, extended to a sheaf on $T^*\IG_\alpha$ and restricted to $\IG_{\alpha}\times U$. Thus, the restriction map on sections induces the desired embedding
\[  	 \WW_{\hbar,f_\mu}(\gl_N) =\Gamma(S_{f_\mu},\mc A_{f_\mu})   \xrightarrow{\text{res}} \Gamma(\IG_\alpha \times U \times S_{f_{\mu+\alpha}}, \mc D^{\rm loc}_\alpha\boxtimes \mc A_{f_{\mu+\alpha}}) = D_{\hbar,\alpha}^{\rm loc}\otimes \WW_{\hbar,f_{\mu+\alpha}}(\gl_N) \ . \]

Similarly, we generalize this approach to establish an abstract inverse quantum Hamiltonian reduction theorem in finite type $A$ for Kostant--Whittaker reductions of arbitrary associative algebra objects in the category of Harish Chandra modules:

\begin{thm}\label{IRHCintrothm}
Let $A$ be an associative algebra in $\rm HC_\hbar$. There exists an embedding of associative algebras with quantum Hamiltonian $\IG_\alpha$-action
\[
	A\red{\chi_\mu} N_{\chi_\mu} \hookrightarrow A\red{\chi_{\mu+\alpha}}N_{\chi_{\mu+\alpha}} \otimes D^{\rm loc}_{\hbar,\alpha}  \ . 
\]
\end{thm}

In particular, we deduce the following corollary for the natural quantizations of coordinate rings of non-principal Moore--Tachikawa varieties in type A:

\begin{cor}  There is an embedding of associative algebras with quantum Hamiltonian $\IG_\alpha$-action
	\[  \ik_\hbar[ W_{\mu^1,...,\mu^{b-1}, \mu} (\SL_N)] \into  \ik_\hbar[ W_{\mu^1,...,\mu^{b-1}, \mu+\alpha} (\SL_N)]  \otimes D^{\rm loc}_{\hbar,\alpha}   \ ,  \]
where $ W_{\mu^1,...,\mu^{b-1}, \mu^b} (\SL_N)$ denotes the non-principal Moore--Tachikawa variety of Definition \ref{MTdefn}.
\end{cor}

Finally, we mention the generalization of the preceding results to the setting of affine W-algebras, which is proved in the companion paper \cite{BuN2} and was our primary motivation for this work:

\begin{thm}\label{IRQAintrothm} For generic $\kappa$, there exists an embedding of vertex algebras
\[ 	 \WW^\kappa_{f_\mu}(\gl_N) \hookrightarrow D^\ch_{\alpha}\otimes \WW^\kappa_{f_{\mu+\alpha}}(\gl_N) \ ,   \]
compatible with the chiral Hamiltonian $\IG_\alpha$-action, where ${D}^\ch_{\alpha}$ denotes a localization of the vertex algebra $D^\ch(\IG_\alpha)$ of chiral differential operators on $\IG_\alpha$, and $\IG_\alpha$ acts trivially on $\WW^\kappa_{f_{\mu+\alpha}}(\gl_N) $.
\end{thm}

Our proof of this theorem follows the same pattern: there exists a sheaf of vertex algebras $\mc A^\kappa_{f_\mu}$ on $S_{f_\mu}$ satisfying $\Gamma(S_{f_\mu}, \mc A_{f_\mu}^\kappa )=\WW_{f_\mu}^\kappa(\gl_N)$ and we establish an isomorphism
\[ \mc A_{f_\mu}^\kappa |_{\IG_\alpha \times U \times S_{f_{\mu+\alpha}} } \xrightarrow{\cong } \mc D_\alpha^\ch\boxtimes \mc A_{f_{\mu+\alpha}}^\kappa  \ , \]
where $\mc D_\alpha^\ch$ denotes the sheaf of chiral differential operators on $\IG_\alpha$, analogously extended to $T^*\IG_\alpha$ and restricted to $\IG_\alpha\times U$, so that the restriction map on sections of $\mc A_{f_\mu}^\kappa $ induces the desired embedding.

\subsection{Relation with previous results}

We now discuss the relation of these results to the existing literature. Although the contemporary context and motivation for inverse Hamiltonian reduction comes from the setting of vertex algebras, as we explain below, we begin by discussing some previous results in the finite-type setting.

Though perhaps less obviously than the other results we discuss below, Theorems \ref{IRCFintrothm} and \ref{IRQFintrothm} can be understood as generalizations of the main results of \cite{Kamnitzer:2022ham}, and moreover our proof of the former follows from first proving a direct generalization of \emph{loc. cit}. For any $n\in \bb N$, dominant coweight $\lambda \in \Lambda^+_{\PGL_n}$, and arbitrary coweight $\mu \in \Lambda_{\PGL_n}$, we let $\Gr_\mu^\lambda$ denote the \emph{generalized slice} between orbits in the affine Grassmannian of $\PGL_n$, and we prove:

\begin{thm} There exists an open immersion of Poisson, Hamiltonian $\IG_\alpha$-varieties
\[ \Gr_{-\alpha}^0 \times \Gr_{\mu+\alpha}^\lambda  \to \Gr_\mu^\lambda  \ , \]
where $\IG_\alpha$ acts trivially on $\Gr_{\mu+\alpha}^\lambda $, and by a certain canonical Hamiltonian action on $\Gr^0_{-\alpha}$ and $\Gr_\mu^\lambda $.
\end{thm}

We then use certain isomorphisms relating equivariant Slodowy slices in type $A$ with generalized slices in the affine Grassmannian of ${\rm PGL}_{2N}$, extending the Mirkovi\'c--Vybornov isomorphism \cite{MVK}, to deduce Theorem \ref{IRCFintrothm} above from the preceding result. Moreover, the results of the companion papers \cite{BuN2} and \cite{BuN3} can be understood as part of a broader program to extend the results of \cite{FKPRW} and \cite{WWY}, and in turn \cite{BrKl}, to the setting of affine Lie algebras, some aspects of which were conjectured in \cite{BR1}.

Next we discuss the relation of our results with those of \cite{Losev2007:quant}. The embeddings of Equation \ref{IRQFintroeqn} can be iteratively composed to yield
\[ \WW_{f_\mu}(\gl_N) \hookrightarrow \bigotimes_{k=1}^m D_{\alpha_k} \otimes \WW_{f_{\nu}}(\gl_N)  \]
for any $\mu ,\nu \in \Lambda^+(\PGL_n)$ such that $|\mu|=|\nu|=N$ and $\mu\leq \nu$ respect to the closure order on nilpotent orbits in $\gl_N$, as well as any choice of positive coroots $\alpha_1,...,\alpha_m$ of $\pgl_n$ such that $\nu=\mu+\sum_{k=1}^m \alpha_k$. Passing to formal completions with respect to the left ideal in $\WW_{f_\mu}(\gl_N)$ corresponding to $f_\nu$, we obtain an isomorphism of complete, topological associative algebras
\begin{equation}\label{Losgenintroeqn}
	  \WW_{f_\mu}(\gl_N)_{\mf m_\nu}^\wedge \xrightarrow{\cong} (A(V_{\mu,\nu})\otimes \WW_{f_\nu}(\gl_N))_{\mf m_\nu}^\wedge
\end{equation}
where $A(V_{\mu,\nu})$ denotes the Weyl algebra associated to the symplectic vector space 
\[V_{\mu,\nu}=T_{f_{\nu}}(\overline{\bb O}_{f_{\nu}} \cap S_{f_\mu}) \ , \]
the tangent space to the intersection $\overline{\bb O}_{f_{\nu}} \cap S_{f_\mu}$, which is symplectic by construction.

In the special case that $\mu=[1^N]$ so that $\nu$ may be arbitrary, and moreover $f_\mu=0$ and thus $\WW_{f_\mu}(\gl_N) \cong U(\gl_N)$, the isomorphism of Equation \ref{Losgenintroeqn} is precisely Theorem 1.2.1 from \cite{Losev2007:quant}. Thus, Theorem \ref{IRQFintrothm} can be considered as a generalization (under the more restrictive hypothesis that $\g$ be of type $A$) of \emph{loc. cit.} to general pairs of W-algebras associated to nilpotents $\mu \leq \nu$, as well as a strengthening of the result in the sense that we lift the isomorphism of formal completions to a Zariski neighborhood. Of course, one would also like to analogously generalize the other main results of \emph{loc. cit.}, especially Theorem 1.2.2. We hope to pursue this question in future work, as well as some related results about module categories in the affine setting, as we discuss below.

The original context of inverse Hamiltonian reduction was in the case of affine W-algebras, though a precise geometric definition did not appear prior to the present work. The first result of this kind was due to Semikhatov \cite{Sem}, who constructed an embedding
\[ V^\kappa(\slf_2) \into D_\alpha^\ch \otimes \WW^\kappa(\slf_2)  \]
by comparing the FMS bosonization \cite{FMS} of the Wakimoto resolution \cite{Wak} of $V^\kappa(\slf_2)$ with the Feigin--Frenkel resolution \cite{FF1} of $\WW^\kappa(\slf_2)$.

This result was recently re-popularized by Dra\v{z}en Adamovi\'c, who showed in \cite{Adamovic:2004zi, Adamović2019ER} that the result above in type $A_1$ descends to an embedding of the simple quotient vertex algebras, and used this to relate their categories of modules in order to express that category of modules over the affine algebra in terms of the more simple category of modules over the principal affine W-algebra for $\slf_2$. Following this, many papers were written studying this phenomena in increasingly general families of examples: see \cite{ACG}, \cite{Feh1}, \cite{Feh2}, \cite{CFLN}, \cite{FFFN}, and \cite{FKN}. The proof of the analogous embeddings in all of these papers rely on analogous ad hoc combinatorial arguments relating the explicit descriptions of the vertex algebras in terms of so-called screening operators. The goal of this series of papers is to establish the general result in type $A$ at generic level $\kappa$ and for arbitrary nilpotents, as well as to give a geometric proof providing a more conceptual explanation for the results.

Finally, both this work and its companion should be viewed as part of a series of recent works in the vertex algebra literature where free field realizations are constructed geometrically---via localization. In particular, the viewpoint that (sheaves of) vertex algebras should be viewed as geometric objects over a space and that free field realizations should be viewed as arising from restricting these sheaves to open subsets. This has been used to great effect in \cite{Arakawa2011:twist}, \cite{Beem:2019tfp}, \cite{Beem:2023uni}, \cite{BeemF}, \cite{Butson:2023fcv}, \cite{Furihata:2023qzp}, \cite{Arakawa:2023cki}, \cite{Beem:2024fom}.

\subsection*{Acknowledgments}

The authors would like to thank Dra\v{z}en Adamovi\'c, Tomoyuki Arakawa, Justine Fasquel, Zachary Fehily, Naoki Genra, Thibault Juillard, Joel Kamnitzer, Vasily Krylov, Shigenori Nakatsuka, Alex Weekes, and Yehao Zhou for useful discussions. We would like to especially thank Christopher Beem, whose collaboration in the early stages of this project was invaluable. The authors would also like to thank Ivan Losev for bringing to our attention an error in an earlier draft of this work.

We gratefully acknowledge the support of ERC grant \# 864828, and for D.B. the support of the Simons Collaboration - New Structures in Low-dimensional Topology grant.


\section{Generalized slices in the affine Grassmannian}

\subsection{The affine Grassmannian}

Let $G$ be an affine algebraic group, $X$ a smooth algebraic curve over $\ik$, and for a closed point $x\in X$ let $\bb D_x=X^{\wedge}_{\{x\}} = \Spf \mc O_{x}$ be the formal disk at $x$ and $\bb D_x^\times = X^{\wedge}_{\{x\}} - \{x\}=\Spf \mc K_x$, the punctured formal disk, where $\mc O_x=\mc O_{X,x}^\wedge$ denotes the complete local ring at $x\in X$ and $\mc K_x$ denotes its field of fractions. We also write simply $\mc O=\ik\Pt$ and $\mc K = \ik\Lt$ as well as $\bb D=(\bb A^1)^\wedge_{\{0\}}=\Spf \mc O $ and $\bb D^\times = (\bb A^1)^\wedge_{\{0\}}\setminus\{0\} = \Spf \mc K $ for the model examples on $\bb A^1$ at $0$, which are (non-canonically) isomorphic to those at any smooth point $x$ on any curve $X$, by the Cohen structure theorem.

\begin{dfn} The (thin) affine Grassmannian $\Gr_G$ is the stack that parameterizes pairs $(\mc P,\sigma)$ where
	\begin{itemize}
		\item $\mc P \in \Bun_G(\bb D)$ is a principal $G$-bundle on $\bb D$; and
		\item $\sigma: \mc P|_{\bb D^\times}  \xrightarrow{\cong} \mc P^\triv$ is a trivialization of the restriction $\mc P|_{\bb D^\times} \in \Bun_G(\bb D^\times)$ of $\mc P$ to $\bb D^\times$.
	\end{itemize}
\end{dfn}

As we will explain, the space $\Gr_G$ is in fact an indscheme. It is modeled concretely by the following quotient stack description, induced by the identifications
\begin{equation}
  \Gr_G =  \pt \fibre{\left[ \pt / G(\mc K) \right]} \left[ \pt / G(\mc O) \right]  \cong   \left[ G(\mc K)  / G(\mc O) \right]  ~,
\end{equation}
where $G(\mc O)$ and $G(\mc K)$ denote the $\mc O$ and $\mc K$ points of the algebraic group $G$. Concretely, one can choose a trivialization of $\mc P$ on $\bb D$, after which the data of $\sigma$ determines an element of $G(\mc K)$ that is well-defined up to a change of the choice of trivialization, or equivalently precomposition with an automorphism of the trivial bundle on $\bb D$, which is given by the action of $G(\mc O)$ on the right.

Note that $\Gr_G=\left[G(\mc K)/G(\mc O)\right]$ admits a left action of $G(\mc O)$, induced by that on $G(\mc K)$, and for $G$ reductive we have the stratifications
\begin{equation*}
	G(\mc K) = \bigsqcup_{\lambda \in \Lambda^+_G} G(\mc O) t^\lambda G(\mc O) \quad\quad \text{and thus} \quad\quad  \Gr_G = \bigsqcup_{\lambda\in \Lambda_G^+} \Gr_G^\lambda \quad\quad \text{where}\quad\quad \Gr_G^\lambda= G(\mc O) t^\lambda 
\end{equation*}
for $ \Lambda_G^+$ is the set of dominant coweights of $G$ and $t^\lambda\in \Gr_G$ defined by the restriction of the corresponding cocharacter $\lambda:\bb \ik^\times \to G$ to the formal neighborhood of $0 \in \ik \supset \ik^\times$.

Further, each $\Gr_G^\lambda$ is a finite type scheme, and the $\IG_\hbar$ action on $\Gr_G$ by `loop rotation' extends to a contraction $\Gr_G^\lambda \onto G\cdot t^\lambda = G/P_\mu$
where $P_\mu \subset G$ is the corresponding parabolic subgroup, with Levi quotient $\textup{Stab}_G(\lambda)$. Thus, $\Gr_G^\lambda$ is a finite dimensional vector bundle over $G/P^\lambda$, as the attracting set fibers are equidimensional affine spaces. Finally, we let
$$ \Gr_G^{\leq \lambda} := \overline{\Gr}_G^\lambda = \bigsqcup_{\mu\leq \lambda} \Gr_G^\mu \quad\quad \text{so that}\quad\quad \Gr_G = \colim_{\lambda \in \Lambda_G^+} \Gr_G^{\leq \lambda}\ ,$$
and note that each $\Gr_G^{\leq \lambda}$ is a projective scheme of dimension $\langle 2\rho^\vee,\lambda\rangle$, so that $\Gr_G$ is an ind-projective, ind-finite type indscheme.

More generally, for a smooth, projective curve $X$ and closed point $x\in X$, we can define the (non-canonically) equivalent space $\Gr_{G,x}=\left[ G(\mc K_x) / G(\mc O_x) \right]$ and we have a natural isomorphism
\begin{equation}
	  \Bun_G(X) \xrightarrow{\cong}    \left[ G(\mc O_{X\setminus \{x\}} )   \backslash G(\mc K_x) / G(\mc O_x)  \right] = \left[ G(\mc O_{X\setminus \{x\}} )   \backslash \Gr_{G,x} \right] 
\end{equation}
by Beauville--Laszlo descent \cite{BeauL}, so that we obtain the following equivalent modular description:

\begin{cor}\label{GrGxcor} The stack $\textup{Gr}_{G,x}$ is equivalent to that parameterizing pairs $(\mc P,\sigma)$ where
\begin{itemize}
\item $\mc P \in \Bun_G(X)$ is a principal $G$-bundle on $X$; and
\item $\sigma: \mc P|_{X\setminus\{x\}}  \xrightarrow{\cong} \mc P^\triv$ is a trivialization of the restriction $\mc P|_{X\setminus\{x\}} \in \Bun_G(X\setminus\{x\})$.
\end{itemize}
\end{cor}

In particular, we identify $\Gr_G$ with $\Gr_{G,0}$ for the point $0\in \bb A^1 \subset \bb P^1=X$, and we have
\begin{equation}\label{BLeqn}
	 \Bun_G(\bb P^1) \xrightarrow{\cong}   \left[ G[t^{-1}]  \backslash G\Lt / G\Pt \right]   \quad\quad \text{and thus} \quad\quad \Gr_G \cong \left[  G\Lt / G\Pt  \right] \ ,
\end{equation}
as above. Note that we also have the alternative, equivalent description by \v{C}ech descent
\begin{equation}\label{CechBLeqn}
	 \Bun_G(\bb P^1) \xrightarrow{\cong}   \left[ G[t^{-1}]  \backslash G[t^{\pm 1}] / G[t] \right]   \quad\quad \text{and thus} \quad\quad \Gr_G \cong \left[  G[t^{\pm 1}] / G[t] \right] \ .
\end{equation}

Next, we recall the definition of a few closely related moduli stacks of $G$-bundles on $\bb P^1$:

\begin{dfn} The thick affine Grassmannian $\Grr_G$ is the stack parameterizing pairs $(\mc P,\sigma)$ where
\begin{itemize}
	\item $\mc P\in \Bun_G(\bb P^1)$ is a principal $G$-bundle on $\bb P^1$; and
	\item $\sigma:\mc P|_{\bb D_\infty } \xrightarrow{\cong} \mc P^\triv$ is a trivialization of the restriction $\mc P |_{\bb D_\infty} \in \Bun_G(\bb D_\infty)$ of $\mc P$ to $\bb D_\infty$.
\end{itemize}
\end{dfn}

In the preceding definition we let $\bb D_\infty =(\bb P^1)^{\wedge}_{\{\infty\}}=\Spf \bb C\Pti $ denote the formal neighborhood of $\infty\in \bb P^1$, and applying Beauville-Laszlo descent at $x=\infty$ we have
\[ \Bun_G(\bb P^1) \xrightarrow{\cong}   \left[ G \Pti \backslash  G\Lti / G[t] \right]   \quad\quad \text{and thus} \quad\quad \Grr_G \cong \left[   G\Lti / G[t] \right]  \ .\]

Note that restriction of the trivialization $\sigma$ from $\bb A^1_\infty:=\bb P^1 \setminus \{0\}$ to $\bb D_\infty$ defines a map
\[ \Gr_G \to \Grr_G \quad\quad \text{which is modeled by} \quad\quad  \left[  G[t^{\pm 1}] / G[t] \right]  \to  \left[   G\Lti / G[t] \right]  \]
in terms of the identifications above. This map is left $G[t]$-equivariant, and restricts to a closed embedding on each of the $G(\mc O)$-orbits $\Gr_G^\lambda=G(\mc O)t^\lambda$, under which they are mapped to $G[t]$-orbits in $\Grr_G$ which we denote $\Grr_G^\lambda= G[t] t^\lambda \subset \Grr_G$, and similarly for their closures.

\subsection{Generalized slices in the affine Grassmannian} Recall that for a subgroup $H\subset G$ and principal $G$-bundle $\mc P\in\Bun_G(X)$, an $H$-structure on $\mc P$ is a principal $H$-bundle $\mc P_H$ together with an isomorphism $ \phi:\mc P_H\times_H G \xrightarrow{\cong} \mc P$. For example, in the case that $H=\{\one\}$ is the trivial group, an $H$-structure is equivalent to a trivialization. Evidently there are natural notions of $H$-structures defined along sub schemes $Y \subset X$ and restriction of $H$-structures to subschemes, generalizing those for trivialisations.

Let $\Bun_G(\bb P^1)_{\infty}$ denote the stack of $G$-bundles on $\bb P^1$ equipped with a trivialisations at the point $\infty\in \bb P^1$, and similarly $\Bun_G(\bb P^1)_{(B,\infty)}$ the stack of $G$-bundles on $\bb P^1$ with a reduction of structure to $B$ at the point $\infty\in \bb P^1$. By the preceding discussion, we have the quotient stack descriptions
\[\Bun_G(\bb P^1)_{\infty} \xrightarrow{\cong}    \left[ G_1 \Pti \backslash  G\Lti / G[t] \right] \quad\quad\text{and}\quad\quad  \Bun_G(\bb P^1)_{(B,\infty)}  \xrightarrow{\cong}    \left[ I_\infty \backslash  G\Lti / G[t] \right] \]
where we define the first congruence and Iwahori subgroups of $G\Lti$ as the fiber products
\[ G_1 \Pti = G\Pti \times_G \{\one\}  \quad\quad\text{and}\quad\quad I_\infty = G\Pti \times_G B  \]
with respect to the evaluation homomorphism $\ev_\infty:G\Pti \to G$. Note we have natural maps
\[ \Grr_G \to  \Bun_G(\bb P^1)_{\infty} \to \Bun_G(\bb P^1)_{(B,\infty)} \to \Bun_G(\bb P^1) \ , \]
defined by restricting the trivialization from $\bb D_\infty$ to $\{\infty\}$, forgetting the trivialization at $\{\infty\}$ to a choice of $B$-structure, and forgetting the $B$-structure entirely, respectively.

The map $B\to T$ defined by the quotient by the unipotent radical $N$ implies that every principal $B$-bundle $\mc P_B$ induces a principal $T$-bundle $\mc P_T = \mc P_B\times_B T$, defining a natural map $\Bun_B(\bb P^1) \to \Bun_T(\bb P^1)$ and thus a decomposition into components
\[ \Bun_B(\bb P^1) = \bigsqcup_{\mu \in \Lambda} \Bun_B^\mu(\bb P^1) \quad\quad \text{where}\quad\quad \Bun_B^\mu(\bb P^1) = \Bun_B(\bb P^1) \times_{\Bun_T(\bb P^1)} \Bun_T^\mu(\bb P^1)  \]
and $\Bun_T^\mu(\bb P^1)$ denotes the substack of $\Bun_T(\bb P^1)$ parameterizing $T$-bundles of degree $\mu\in \Lambda$. Further, note that the stack of $B$-bundles is evidently equivalent to the stack of $G$-bundles equipped with a global $B$-structure, so that restriction of the $B$-structure to $\infty \in \bb P^1$ defines a map
\[  \Bun_B(\bb P^1) \to \Bun_G(\bb P^1)_{(B,\infty)} \quad \quad \text{modeled by} \quad\quad    \left[ B \Pti \backslash  B\Lti / B[t] \right]  \to  \left[ I_\infty \backslash  G\Lti / G[t] \right]  \ , \]
the map induced by the inclusions $ B\Lti \to G\Lti  $, $B \Pti \subset I_\infty $, and $B[t] \subset G[t]$.

We now give the definition of generalized slices for $\lambda \in \Lambda^+$ and $\mu \in \Lambda$:
\begin{dfn}\label{genslicedef} The generalized slice $\Gr_\mu^\lambda$ is the stack defined by
\[\Gr_\mu^\lambda = \overline{\Gr}_G^\lambda \times_{\Bun_G(\bb P^1)_{(B,\infty)}}  \Bun_B^{w_0\mu}(\bb P^1)\ . \]
Equivalently, $\Gr_\mu^\lambda$ is the stack parameterizing tuples $(\mc P,\sigma,\mc P_B,\phi)$, where
\begin{itemize}
\item $\mc P \in \Bun_G(\bb P^1)$ is a principal $G$-bundle on $\bb P^1$;
\item $\sigma: \mc P|_{\bb A^1_\infty}  \xrightarrow{\cong} \mc P^\triv$ is a trivialization of the restriction $\mc P|_{\bb A^1_\infty } \in \Bun_G(\bb A^1_\infty )$, with pole at $0$ of degree bounded by $\lambda$;
\item $\mc P_B \in \Bun_B^{w_0 \mu}(\bb P^1)$ is a principal $B$-bundle on $\bb P^1$ of degree ${w_0\mu}$; and
\item $\phi: \mc P_B\times_B G \xrightarrow{\cong } \mc P$ is an isomorphism of principal $G$-bundles, such that the composition
\[  \mc P_B|_{\bb A^1_\infty}  \to   \mc P_B |_{\bb A^1_\infty} \times_B G \xrightarrow{\phi |_{\bb A^1_\infty} } \mc P|_{\bb A^1_\infty}  \xrightarrow{\sigma} \mc P^\triv  = \bb A^1_\infty \times G \]
maps the fiber $\mc P_B|_{\{\infty\}} $ of $\mc P_B$ at $\infty \in \bb P^1$ to $B_- \subset G$.
\end{itemize}
\end{dfn}

Similarly to $\Gr_G^\lambda$ the spaces $\Gr_\mu^\lambda$ are finite type schemes. 

Next, we define the convolution Grassmannian, which will provide a natural resolution of the spaces $\Gr_\mu^\lambda$. Let $X$ be a smooth, projective algebraic curve, $x\in X$ a closed point, and $n\in \bb N$:

\begin{dfn}\label{convGrdfn} The $n$-ary convolution Grassmannian $\Gr_{X,x}^{(n)}$ at $x\in X$ is the stack parameterizing tuples $(\mc P_1,...,\mc P_n,\sigma_1,\sigma_2,...,\sigma_n) $ where
\begin{itemize}
	\item $\mc P_i \in \Bun_G(X)$ is a principal $G$-bundle on $X$ for $i=1,...,n$,
	\item $\sigma_1:\mc P_1|_{X\setminus\{x\}}\xrightarrow{\cong} \mc P^\triv$ is a trivialization, and
	\item $\sigma_i: \mc P_{i}|_{X\setminus\{x\}} \xrightarrow{\cong} \mc P_{i-1}|_{X\setminus\{x\}}$ is an isomorphism of principal $G$-bundles for $i=2,...,n$.
\end{itemize}
\end{dfn}

As for the usual affine Grassmannian, these spaces are (non-canonically) equivalent for any smooth point $x$ on any curve $X$, and the model example $\Gr_G^{(n)}:=\Gr_{G,\bb P^1,0}^{(n)}$ is canonically equivalent to the stack parameterizing such data for $(X,x)=(\bb D,0)$.

As for the preceding spaces, we again have the natural quotient stack description
\begin{align*}
	\Gr_{G}^{(n)}  & \cong  \pt \times_{ \left[ \pt / G(\mc K) \right]} \left[ \pt / G(\mc O) \right]  \times_{ \left[ \pt / G(\mc K) \right]}  \hdots  \times_{ \left[ \pt / G(\mc K) \right]} \left[ \pt / G(\mc O) \right]  \\
	&  =   \left[ G(\mc K) \times_{G(\mc O)} G(\mc K) \times_{G(\mc O)} \hdots \times_{G(\mc O)} G(\mc K)   / G(\mc O) \right]
\end{align*}
where we let $G(\mc K) \times_{G(\mc O)} G(\mc K)= [G(\mc K)/G(\mc O)] \times_{[\pt/G(\mc O)]} [ G(\mc O) \backslash G(\mc K) ]= \left[ (G(\mc K) \times G(\mc K))/G(\mc O) \right]$, the diagonal quotient by the product of the $G(\mc O)$-actions on the right on the first factor and on the left on the second, and similarly the further iterated products are defined inductively. Note there is a natural convolution map
\[ m: \Gr_G^{(n)} \to \Gr_G \quad\quad\text{defined by} \quad\quad (\mc P_1,...,\mc P_n,\sigma_1,...,\sigma_n) \mapsto (\mc P_n, \sigma_1 \circ \hdots \circ  \sigma_n) \ . \]

For $\underline{\lambda}=(\lambda_1,\lambda_2,\dots,\lambda_n)$ an $n$-tuple of dominant coweights with sum $\lambda$, we have the natural substack $\overline{\Gr}^\ul_G \subset \Gr^{(n)}_G$ defined by the condition that each map $\sigma_i$ has a pole at $0$ bounded by $\lambda_i$, or equivalently in terms of the quotient stack description
\begin{eqnarray}\label{convgrconceqn}
	 \overline{\Gr}^\ul_G \cong  \left[ \overline{G(\mc K)}^{\lambda_1}  \times_{G(\mc O)}\overline{G(\mc K)}^{\lambda_2}  \times_{G(\mc O)} \hdots \times_{G(\mc O)} \overline{G(\mc K)}^{\lambda_n}    / G(\mc O) \right] \ , 
\end{eqnarray}
where $\overline{G(\mc K)}^{\lambda} = \overline{G(\mc O) t^{\lambda} G(\mc O)} \subset G(\mc K)$ denotes the closure of the $G(\mc O)\times G(\mc O)$-orbit of $t^\lambda$ in $G(\mc K)$. Evidently the convolution map restricts to a map $m^\ul: \overline{\Gr}_G^\ul \to \overline{\Gr}_G^\lambda$.

We now give the definition of the convolution generalized slice for $\underline{\lambda}=(\lambda_1,\lambda_2,\dots,\lambda_n)$ be an $n$-tuple of dominant coweights with sum $\lambda\in \Lambda^+$ and $\mu \in \Lambda$:

\begin{dfn} The convolution generalized slice $\Gr^\ul_\mu$ is the stack defined by
\[ \Gr^\ul_\mu = \overline{\Gr}_G^\ul \times_{\Bun_G(\bb P^1)_{(B,\infty)}}  \Bun_B^\mu(\bb P^1)\ .  \]
Equivalently, $\Gr^\ul_\mu$ is the stack parameterizing tuples $(\mc P_1,...,\mc P_n,\sigma_1,...,\sigma_n,\mc P_B,\phi) $ where
\begin{itemize}
	\item $\mc P_i \in \Bun_G(\bb P^1)$ is a principal $G$-bundle on $\bb P^1$ for $i=1,...,n$;
	\item $\sigma_1:\mc P_1|_{\bb A^1_\infty}\xrightarrow{\cong} \mc P^\triv$ is a trivialization with pole at $0$ bounded by $\lambda_1$;
\item $\sigma_i: \mc P_{i}|_{\bb A^1_\infty} \xrightarrow{\cong} \mc P_{i-1}|_{\bb A^1_\infty}$ is an isomorphism of principal $G$-bundles with pole at $0$ bounded by $\lambda_i$ for $i=2,...,n$;
	\item $\mc P_B \in \Bun_B^{w_0\mu}(\bb P^1)$ is a principal $B$-bundle on $\bb P^1$ of degree ${w_0\mu}$; and
	\item $\phi: \mc P_B\times_B G \xrightarrow{\cong } \mc P_n$ is an isomorphism of principal $G$-bundles, such that the composition
\[  \mc P_B|_{\bb A^1_\infty}  \to   \mc P_B |_{\bb A^1_\infty} \times_B G \xrightarrow{\phi |_{\bb A^1_\infty} } \mc P_n|_{\bb A^1_\infty}  \xrightarrow{\sigma_1\circ \hdots \circ \sigma_n} \mc P^\triv  = \bb A^1_\infty \times G \]
maps the fiber $\mc P_B|_{\{\infty\}} $ of $\mc P_B$ at $\infty \in \bb P^1$ to $B_- \subset G$.
\end{itemize}
\end{dfn}

Note that the convolution map $\overline{\Gr}_{G}^\ul \to  \overline{\Gr}_G^\lambda$ induces a convolution map $ {\varpi}^{\underline{\lambda}}_{\mu}:{\Gr}_\mu^\ul \to \Gr_\mu^\lambda  \ .$

\subsection{Beilinson--Drinfeld generalized slices} Next, we define the Beilinson--Drinfeld Grassmannian, a factorization analogue of the usual affine Grassmannian, which will provide a natural deformation of the spaces $\Gr_\mu^\lambda$.

Let $X$ be a smooth algebraic curve and $I$ be a finite set.

\begin{dfn} The Beilinson--Drinfeld Grassmannian $\Gr^\bd_{G,X^I}$ over $X^I$ is the stack parameterizing tuples $((x_i)_{i\in I}, \mc P,\sigma)$ where
\begin{itemize}
	\item $(x_i)_{i\in I} \in X^I$ is a closed point of the $I$-fold product $X^I$;
	\item $\mc P \in \Bun_G(X)$ is a principal $G$-bundle on $X$; and
	\item $\sigma: \mc P|_{X\setminus\{x_i\}_{i\in I}}  \xrightarrow{\cong} \mc P^\triv$ is a trivialization of the restriction $\mc P|_{X\setminus\{x_i\}_{i\in I}} \in \Bun_G(X\setminus\{x\})$.
\end{itemize}
	
\end{dfn}

Evidently there is a natural map $\Gr^\bd_{G,X^I}\to X^I$, and for $I=\{1\}$ so that $X^I=X$ we have
\[ \Gr^\bd_{G,X} \times_X \{x\} = \Gr_{G,x} \ ,\]
for each closed point $x\in X$, where $\Gr_{G,x}$ is as in Corollary \ref{GrGxcor} above. More generally, for each $(x_i)\in X^I$, the fiber of $\Gr_{G,X^I}^\bd$ is given by
\[\Gr^\bd_{G,X^I}\times_{X^I}  \{(x_i)\} \cong \prod_{y \in \{x_i\}} \Gr_{G,y}   \ , \]
where the product is over the distinct points $y \in \{x_i\} \subset X$. Note that this space depends on the tuple $(x_i)_{i \in I}$ only via $\{x_i\}\subset X$, and in particular ignores any repetitions $x_i=x_j$ for $i,j\in I$ distinct. This property is expressed globally by the existence of natural, compatible isomorphisms
\[ \Gr_{G,X^I} \times_{X^I} X^J \xrightarrow{\cong} \Gr_{G,X^J} \]
for any partial diagonal embedding $\Delta(\pi):X^J\into X^I$ corresponding to a surjection $\pi:I\onto J$.

We consider the case $X=\bb P^1$, and for each finite set $I$ define the space $\Gru^I_G \to \bb A^I$ by
\begin{equation}\label{GruIdefeqn} \Gru^I_G = \Gr_{G, (\bb P^1)^I}^\bd \times_{(\bb P^1)^I} \bb A^I  \ . \end{equation}
Note that the support of the trivialization $\sigma$ determined by a closed point in $\Gru^I_G $ always contains $\infty \in \bb P^1$, so that there exist natural maps 
\[  \Gru^I_G \to \Bun_G(\bb P^1)_\infty \quad\quad\text{and thus}\quad\quad \Gru^I_G \to \Bun_G(\bb P^1)_{(B,\infty)} \]
given by restriction of the trivialization $\sigma$ to $\infty$, and further forgetting the trivialization to a choice of $B$ structure at $\infty$, respectively.

Given a finite set $I$ and an $I$-tuple $(\lambda_i)_{i \in I}$ of dominant coweights, there is a natural substack
$\overline{\Gru}^\ul_G \subset \Gru_G^I$ defined by the condition that the trivialization $\sigma:\mc P|_{X\setminus\{x_i\}_{i\in I}}  \xrightarrow{\cong} \mc P^\triv$ has a pole at each point $y\in \{x_i\}$ bounded by \begin{equation}\label{BDGrlambdaeqn}
	 \lambda_y = \sum_{\{ i \in I |  x_i=y \} } \lambda_i \quad\quad\text{so that} \quad\quad  \overline{ \Gru}^\ul_G \times_{\bb A^I} \{( x_i)\} \cong \prod_{y \in \{x_i\}} \overline{\Gr}_{G,y}^{\lambda_y}  \ , 
\end{equation}
where the product is again over the distinct points $y \in \{x_i\} \subset X$.

\begin{dfn}\label{BDgenslicedef} The Beilinson--Drinfeld generalized slice $\underline{\Gr}_\mu^\ul$ is the stack defined by
\[\underline{\Gr}_\mu^\ul = \overline{\Gru}^\ul_G  \times_{\Bun_G(\bb P^1)_{(B,\infty)}}  \Bun_B^{w_0\mu}(\bb P^1)\ .  \]
Equivalently, $\underline{\Gr}^\ul_\mu$ is the stack parameterizing tuples $((x_i)_{i\in I},\mc P,\sigma,\mc P_B,\phi) $ where
\begin{itemize}
		\item $(x_i) \in \bb A^I$ is a closed point of the $I$-fold product $\bb A^I$;
\item $\mc P \in \Bun_G(\bb P^1)$ is a principal $G$-bundle on $\bb P^1$;
	\item $\sigma: \mc P|_{\bb A^1 \setminus\{x_i\}}  \xrightarrow{\cong} \mc P^\triv$ is a trivialization of the restriction $\mc P|_{\bb A^1 \setminus\{x_i\}} \in \Bun_G(\bb A^1 \setminus\{x_i\}_{i\in I})$ with pole at $x_i$ bounded by $\lambda_{x_i}$;
\item $\mc P_B \in \Bun_B^{w_0\mu}(\bb P^1)$ is a principal $B$-bundle on $\bb P^1$ of degree ${w_0\mu}$; and
\item $\phi: \mc P_B\times_B G \xrightarrow{\cong } \mc P$ is an isomorphism of principal $G$-bundles, such that the composition
\[  \mc P_B|_{\bb D_\infty}  \to   \mc P_B |_{\bb D_\infty} \times_B G \xrightarrow{\phi |_{\bb D_\infty} } \mc P|_{\bb D_\infty}  \xrightarrow{\sigma} \mc P^\triv  = \bb D_\infty \times G \]
maps the fiber $\mc P_B|_{\{\infty\}} $ of $\mc P_B$ at $\infty \in \bb P^1$ to $B_- \subset G$.
\end{itemize}
\end{dfn}

There is also an extension of the convolution Grassmannian to the factorization setting, which provides a simultaneous deformation and resolution, defined as follows for $n\in \bb N$:

\begin{dfn} The $n$-ary convolution Beilinson--Drinfeld Grassmannian $\widetilde{\Gr}^\bd_{G,X^n}$ is the stack parameterizing tuples 
$((x_i)_{i=1}^n, \mc P_1,...,\mc P_n,\sigma_1,...,\sigma_n) $, where
\begin{itemize}
		\item $(x_i)_{i=1}^n \in X^n$ is a closed point of $X^n$;
		\item $\mc P_i \in \Bun_G(X)$ is a principal $G$-bundle on $X$ for $i=1,...,n$;
	\item $\sigma_1:\mc P_1|_{X\setminus\{x_1\}}\xrightarrow{\cong} \mc P^\triv$ is a trivialization; and
\item $\sigma_i: \mc P_{i}|_{X\setminus\{x_i\}} \xrightarrow{\cong} \mc P_{i-1}|_{X\setminus\{x_i\}}$ is an isomorphism of principal $G$-bundles for $i=2,...,n$.
\end{itemize}
\end{dfn}

Evidently there is again a natural map $\widetilde{\Gr}^\bd_{G,X^n}\to X^n$, with fiber given by
\[\widetilde{\Gr}^\bd_{G,X^n}\times_{X^n}  \{(x_i)\} \cong \prod_{y \in \{x_i\}} \Gr_{G,y}^{(n_y)} \quad\quad\text{where}\quad\quad  n_y=\#\{ i\in \{1,...,n\}  \ | \ x_i=y \}   \ , \]
the product over distinct points $y\in \{x_i\}$ of the $n_y$-ary convolution Grassmannian $\Gr_{G,y}^{(n_y)}$ based at the point $y\in X$, in the sense of Definition \ref{convGrdfn}.

For $\underline{\lambda}=(\lambda_1,\lambda_2,\dots,\lambda_n)$ an $n$-tuple of dominant coweights with sum $\lambda$, we have a natural substack $\widetilde{\Gr}^{\bd,\ul}_{G,X^n} \subset \widetilde{\Gr}^\bd_{G,X^n}$ defined by the condition that each map $\sigma_i$ has a pole at $x_i$ bounded by $\lambda_i$. The fibers of this substack are given analogously by the substacks
\[ \widetilde{\Gr}^{\bd,\ul}_{G,X^n}\times_{X^n}  \{(x_i)\} \cong  \prod_{y \in \{x_i\}}\overline{ \Gr}_{G,y}^{\ul_y} \quad\quad\text{where}\quad\quad  \ul_y= (\lambda_i)_{\{ i\in \{1,...,n\}  \ | \ x_i=y \}}    \]
is the tuple of $n_y$ dominant coweights $\lambda_i$ corresponding to the points $x_i=y$, and $\overline{\Gr}_{G,y}^{\ul_y}$  is as defined preceding Equation \ref{convgrconceqn}. Analogously to the above, we let
\[\widetilde{\Gru}_G^n=\widetilde{\Gr}_{G, (\bb P^1)^n}^\bd \times_{(\bb P^1)^n} \bb A^n \quad\quad\text{and}\quad\quad \widetilde{\Gru}_{G}^\ul= \widetilde{\Gr}_{G,(\bb P^1)^n}^{\bd,\ul}\times_{(\bb P^1)^n} \bb A^n  \ , \]
and note that we have a natural map $\widetilde{\Gru}_G^n \to {\Gru}_G^n$ inducing $\widetilde{\Gru}_{G}^\ul \to  \overline{\Gru}_G^\lambda$.

\begin{dfn} The Beilinson--Drinfeld convolution generalized slice $\widetilde{\underline{\Gr}}_\mu^\ul$ is the stack defined by
	\[ \widetilde{\underline{\Gr}}_\mu^\ul =\widetilde{\Gru}_{G}^\ul  \times_{\Bun_G(\bb P^1)_{(B,\infty)}}  \Bun_B^{w_0\mu}(\bb P^1)\ .  \]
	Equivalently, $ \widetilde{\Wu}_\mu^\ul $ is the stack parameterizing tuples $((x_i)_{i=1}^n, \mc P_1,...,\mc P_n,\sigma_1,...,\sigma_n,\mc P_B,\phi) $ where
	\begin{itemize}
		\item $(x_i)_{i=1}^n \in \bb A^n$ is a closed point of $\bb A^n$;
		\item $\mc P_i \in \Bun_G(\bb P^1)$ is a principal $G$-bundle on $X$ for $i=1,...,n$;
			\item $\sigma_1:\mc P_1|_{\bb P^1\setminus\{x\}}\xrightarrow{\cong} \mc P^\triv$ is a trivialization with pole at $x_1$ bounded by $\lambda_1$;
		\item $\sigma_i: \mc P_{i}|_{\bb P^1 \setminus\{x_i\}} \xrightarrow{\cong} \mc P_{i-1}|_{\bb P^1\setminus\{x_i\}}$ is an isomorphism of principal $G$-bundles with pole at $x_i$ bounded by $\lambda_i$ for $i=2,...,n$;
		\item $\mc P_B \in \Bun_B^{w_0\mu}(\bb P^1)$ is a principal $B$-bundle on $\bb P^1$ of degree ${w_0\mu}$; and
		\item $\phi: \mc P_B\times_B G \xrightarrow{\cong } \mc P_n$ is an isomorphism of principal $G$-bundles, such that the composition
		\[  \mc P_B|_{\bb D_\infty}  \to   \mc P_B |_{\bb D_\infty}\times_B G \xrightarrow{\phi |_{\bb D_\infty}} \mc P_n|_{\bb D_\infty} \xrightarrow{\sigma_1|_{\bb D_\infty}\circ \hdots \circ \sigma_n|_{\bb D_\infty}} \mc P^\triv  = \bb D_\infty \times G \]
		maps the fiber $\mc P_B|_{\{\infty\}} $ of $\mc P_B$ at $\infty \in \bb P^1$ to $B_- \subset G$.
	\end{itemize}
\end{dfn}

Note that the convolution map $\widetilde{\Gru}_{G}^\ul \to  \overline{\Gru}_G^\lambda$ induces a convolution map $ \underline{\varpi}^{\underline{\lambda}}_{\mu}:\widetilde{\Wu}_\mu^\ul \to \Wu_\mu^\ul  \ .$

\subsection{Matrix realization of generalized slices}\label{ssec:matrix realisations}

It will be extremely convenient for certain explicit computations to use the matrix realization of the generalized slices $\Gr^\lambda_\mu$, which we recall below essentially following the argument of Kamnitzer summarized in Section 2(xi) of \cite{Braverman:2016pwk}. In the physics literature, these matrix realizations of Coulomb branches were studied as so-called scattering matrices for monopole operators in $\IR\times \IC$ by \cite{Bullimore:2015Coul}.

Recall from Definition \ref{BDgenslicedef} that for a finite set $I$, an $I$-tuple $(\lambda_i)_{i \in I}$ of dominant coweights, and $\mu\in \Lambda$ an arbitrary (integral) coweight, we defined the Beilinson-Drinfeld generalized slice by
\[\underline{\Gr}_\mu^\ul = \overline{\Gru}^\ul_G  \times_{\Bun_G(\bb P^1)_{(B,\infty)}}  \Bun_B^{w_0\mu}(\bb P^1)\ , \]
where $\overline{\Gru}^\ul_G \subset \Gru_G^I$ denotes the substack defined in the sentence ending in Equation \ref{BDGrlambdaeqn}, of the Beilinson-Drinfeld affine Grassmannian defined in Equation \ref{GruIdefeqn}. In particular, for any point $\ut=(t_i)_{i\in I} \in \bb A^I$ we let $\overline{ \Gru}^\ul_{G,\ut} =\overline{ \Gru}^\ul_G \times_{\bb A^I} \{( t_i)\}$ and $\underline{\Gr}_{\mu,\ut}^\ul =\underline{\Gr}_\mu^\ul  \times_{\bb A^I} \{( t_i)\} $ denote the fibers at $\ut$.

We also recalled the Beauville--Laszlo type descriptions for various stacks of $G$-bundles, among which in the present setting we will use
\[ \Bun_B(\bb P^1) \xrightarrow{\cong}    \left[ B \Pti \backslash  B\Lti / B[t] \right] \qquad\text{and}\qquad\Bun_G(\bb P^1)_{(B,\infty)}  \xrightarrow{\cong}    \left[ I_\infty \backslash  G\Lti / G[t] \right] \]
where $I_\infty = G\Pti \times_G B$, as well as the fact that the relevant map in the fiber product
\[  \Bun_B(\bb P^1) \to \Bun_G(\bb P^1)_{(B,\infty)} \quad \quad \text{is modeled by} \quad\quad    \left[ B \Pti \backslash  B\Lti / B[t] \right]  \to  \left[ I_\infty \backslash  G\Lti / G[t] \right]  \ . \]

Consider the subschemes of $G(\!(t^{-1})\!)$, defined by
\begin{equation}
 N_1[\![t^{-1}]\!]T_1[\![t^{-1}]\!]t^\mu N_{-,1}[\![t^{-1}]\!] \qquad \text{and} \qquad  \overline{ G[t] \ut^\ul G[t] }   ,
\end{equation}
where the latter denotes the image under $G(t) \to G\Lti$ of the closure of the double coset $ G[t] \ut^\ul G[t]$ of the point $\ut^\ul=\prod_{i\in I} (t-t_i)^{\lambda_i}$. Then we have the following identification:
\begin{prop}
The fiber of the Beilinson--Drinfeld generalized slice admits an isomorphism
\begin{equation}
\Gru_{\mu,\ut}^\lambda \xrightarrow{\cong } \overline{ G[t] \ut^\ul G[t] } \times_{ G\Lti } N_1[\![t^{-1}]\!]T_1[\![t^{-1}]\!]t^\mu N_{-,1}[\![t^{-1}]\!]   \ .
\end{equation}
\end{prop}
\begin{proof} We first note some natural identifications of the fiber product of the ambient stacks of bundles: in the present context the affine Grassmannian $\Gr_G$ is naturally identified as in Equation \ref{CechBLeqn} with $\left[G[t^{\pm 1}]/G[t]\right]$, so that we can compute for example
	\begin{align*}
	 \Gr_G \times_{\Bun_G(\bb P^1)_{(B,\infty)}}  \Bun_B(\bb P^1) & =  \left[ G[t^{\pm 1}]/G[t] \right] \times_{\left[ I_\infty \backslash  G\Lti / G[t] \right] } \left[ B \Pti \backslash  B\Lti / B[t] \right]  \\ 
	 & \cong \left[\left( G[t^{\pm 1}]\times_{\left[ I_\infty \backslash  G\Lti  \right] } \left[ B \Pti \backslash  B\Lti  \right]\right)  / B[t]\right] \\
	 & \cong \left[   \left( G[t^{\pm 1}]\times_{G\Lti } (N_{-,1}\Pti \times B\Lti )\right) / B[t]\right]  
	\end{align*}

where the later identification is given by a choice of section
\begin{equation}\label{matrixrepseceqn}
	  \left[ B\Pti \backslash (I_\infty \times B\Lti ) / B[t] \right] \xrightarrow{\cong }  \left[ (N_{-,1}\Pti \times B\Lti ) / B[t]  \ . \right]
\end{equation}
	
 Similarly, we can identify the analogous substack to that of Equation \ref{matrixrepseceqn} after further intersecting the component $\Bun_B^{w_0\mu}(\bb P^1)$ with the subvariety
 \[  N_1[\![t^{-1}]\!]T_1[\![t^{-1}]\!]t^\mu N_{-,1}[\![t^{-1}]\!] \to N_{-,1}\Pti  B\Lti   \quad \subset \quad G\Lti   \ . \]
 
 More generally, each factor $\Gr_{G,t_i}$ of the fiber of the Beilinson-Drinfeld Grassmannian $\Gru_{G,\ut}$ is naturally identified with $\left[G[(t-t_i)^{\pm 1}]/G[t]\right]$, the product of which naturally includes into $G(z)/G[t]$, under which
 \[\overline{ \Gru}^\ul_{G,\ut} \xrightarrow{\cong } \overline{ G[t] \ut^\ul G[t] }/G[t] \ .  \]
 Thus, in summary we obtain the desired result:
 \begin{align*}
 	\Gru_{\mu,\ut}^\lambda & = \overline{ \Gru}^\ul_{G,\ut}  \times_{\Bun_G(\bb P^1)_{(B,\infty)}}  \Bun_B^{w_0\mu}(\bb P^1)     \\
 	& \xrightarrow{\cong}   \overline{ G[t] \ut^\ul G[t] }   \times_{G\Lti } N_1[\![t^{-1}]\!]T_1[\![t^{-1}]\!]t^\mu N_{-,1}[\![t^{-1}]\!] \ . 
 \end{align*}
\end{proof}

Now fix $G\equiv \PGL_n$ for some $n$. This description allows us to give matrix realizations of elements $x\in \Gr^\lambda_\mu$, \ie, give a $G[t]$ representative of $x$. First, the condition $x\in \Gr^\mu$ implies that $x$ has a Gauss decomposition 
\begin{equation}
	x = n h t^\mu n_-
\end{equation}
where $n\in N_1[\![t^{-1}]\!]$, $h\in T_1[\![t^{-1}]\!]$ and $n_-\in N_{-,1}[\![t^{-1}]\!]$. Now $x \in \overline{\Gr}^\lambda$ imposes a constraint on the minors of $x$. Namely, the $m$-th minor must satisfy $0\ge {\rm deg}\,{\rm Min}_m \le    t^{\sum_{i=1}^m \lambda_i}$.

\subsection{Open Zastava and weakly antidominant slices}\label{ssec:Zastava}

Of particular interest to us will be generalized slices of the form $\Gr^0_{-\alpha}$, for $\alpha\in \Lambda^{\geq 0}_G$ a positive coroot. Such generalized slices were first studied in their incarnation as open Zastava spaces \cite{Finkelberg1999:MonopoleSymplectic,Braverman2002:zastava}, as we now recall:

Recall that the second integral homology of the flag variety $H_2(G/B;\bb Z)$ is canonically identified with the coweight lattice $\Lambda_G$, and the cone of effective curve classes is identified with the positive coroots $\Lambda_G^{\geq 0}$, those $\alpha \in \Lambda_G$ such that
\[ \alpha = \sum_{i\in I} a_i \alpha_i \qquad \text{for} \qquad a_i \geq 0 \ .  \]

\begin{dfn} The open Zastava space $\open{Z^\alpha}$ is the space of maps $\varphi:\bb P^1\to G/B$ such that $\varphi(\infty)=[B_-] \in G/B$ and $\varphi_*[\bb P^1]=\alpha \in H_2(G/B;\bb Z) \cong \Lambda_G^{\geq 0}$.
\end{dfn}

Note that we can identify the free quotient variety $G/B$ with the stack $[G/B]$ and thus obtain the equivalent description of the open Zastava as parameterizing pairs $(\mc P_B,\varphi)$, where
\begin{enumerate}
	\item $\mc P_B\in \Bun_B^{\alpha}(\PP^1)$ a $B$ bundle of degree $\alpha$, and
	\item $\varphi$ a section (or equivalently a trivialization) of the associated $G$-bundle $\mc P_B \times_B G$,
\end{enumerate}
such that the composition
\[ B=\mc P_B|_{\{\infty\}} \to( \mc P_B \times_B G )|_{\{\infty\}} \xrightarrow{\varphi}  G \]
maps $B$ to $B_-\subset G$. Thus, we obtain:
\begin{prop} There is a canonical identification
	\[ \Gr_{-\alpha}^0 = \open{Z^{\alpha^*}}\]
where $\alpha^*=-w_0 \alpha$.
\end{prop}

One particular class of examples that play an important role in our main results are the following: 
Let $\alpha$ be a positive coroot in $\rm PGL_n$; any such positive coroot is of the form $\alpha\equiv \alpha_{(r_1,r_2)} = \sum_{i=r_1}^{r_2} \alpha_i$, with $r_1\le r_2\le n$ a pair of positive integers and $r_2-r_1 = k-1$ for some integer $k>0$.

One can define a system of (quasi)-Darboux co-ordinates on $\Gr^0_{-\alpha}$ as follows. A matrix representative $x \in \Gr^0_{-\alpha}$, can be parameterized as a block diagonal matrix 
\begin{equation}
	x = 
	\begin{pmatrix}
	\id_{r_1-1} & 0\\
	0 & B_{k+1} & 0 \\
	0 & 0 & \id_{n-r_2-1} \\
	\end{pmatrix}
\end{equation}
where $B_{k+1}$ is a $(k+1)\times (k+1)$ block, parameterized as 
\begin{equation}\label{eq:coords on zastava}
	\begin{pmatrix}
	0 & 0 & 0 &  \dots & 0 & e\\
	0 & 1& 0 & \dots & 0 & b_{k-1}\\
	0 & 0& 1 & \dots & 0 & b_{k-2}\\
	\vdots & \vdots & \vdots & \vdots & \vdots & \vdots \\
	0 & 0 & 0 &\dots & 1 & b_{1} \\
	-1/e & -g_{k-1} & -g_{k-2} & \dots & -g_1 & t - p -\sum_{i=1}^{k-1} b_i g_i
	\end{pmatrix}~.
\end{equation}
The symplectic form in these co-ordinates is 
\begin{equation}
	\Omega^0_k =  \frac1e dp\wedge de + \sum_{i=1}^{k-1} db_i\wedge dg_i~.
\end{equation}
With associated, non-vanishing, Poisson brackets given by
\begin{equation}
	\{p,e^\pm\} = \pm e^\pm \text{ and } \{b_i,g_j\} = \delta_{ij}~.
\end{equation}
We shall make repeated use of this co-ordinate system. 

\begin{lem}\label{lem:antidominant_abelian}
Let $\IG_\alpha\equiv \IG_a^{k}$ denote the rank $k$ additive group.
The generalized slice $\Gr^0_{-\alpha}$ is an open symplectic subvariety of $T^*\IG_\alpha$
\end{lem}
\begin{proof}
	This follows immediately from the co-ordinate system \eqref{eq:coords on zastava}.
\end{proof}

\subsection{Coulomb branch description of generalized slices}\label{Coulombsec} It will also occasionally be helpful to use the description of generalized slices in terms of Coulomb branches of three dimensional $\mc N=4$ quiver gauge theories, in the sense of \cite{Braverman:2016pwk}. Given $\lambda \in \Lambda^+$ and $\mu \in  \Lambda$ such that $\lambda-\mu\in \Lambda_G^{\geq 0}$, we follow the conventions of Section 1.2 of \cite{WWY} letting
\begin{equation}\label{lambdamumconveqn}
	 \lambda = \sum_{i=1}^{n-1} \lambda_i \omega_{n-i} \qquad \text{and}\qquad\lambda - \mu = \sum_{i=1}^{n-1} m_i \alpha_{n-i}  \qquad \text{for some $m_i, \lambda_i \in \bb Z_{\geq 0}$}  \ .
\end{equation}
Let $V_i=\bb C^{m_i}$, $\GL(V)=\prod_{i\in I} \GL(V_i)$, $W_i=\bb C^{\lambda_i}$, $\GL(W)=\prod_{i\in I} \GL(W_i)$, and moreover
\[ {\bf{G}}=\GL(V)\qquad {\bf G}_F=\GL(W) \qquad \text{and} \qquad {\bf{N}} = \bigoplus_{i=1}^{n-2} \Hom(V_i,V_{i+1}) \oplus \bigoplus_{i=1}^{n-1} \Hom(W_i,V_i) \ . \]
The main construction of the  companion paper \cite{BFN1}, preceding \emph{loc. cit.} above, provides a commutative algebra structure on
\[ \bb C[ \mc M_C( {\bf{G}}, {\bf{N}})] := \mc A( {\bf{G}}, {\bf{N}}) = H_\bullet^{{\bf G}(\mc O)}( \mc R( {\bf{G}}, {\bf{N}})) \]
the ${\bf G}(\mc O)$-equivariant cohomology of a certain ind-pro-finite type variety $\mc R( {\bf{G}}, {\bf{N}})$ which is pro-finite type over the ind-scheme $\Gr_{{\bf G}}$, the spectrum of which we define as the resulting Coulomb branch $\mc M_C( {\bf{G}}, {\bf{N}})=\Spec  H_\bullet^{{\bf G}(\mc O)}( \mc R( {\bf{G}}, {\bf{N}})) $. The variety $\mc M_C( {\bf{G}}, {\bf{N}})$ admits a natural deformations
\[\Mu_C( {\bf{G}}, {\bf{N}})=\Spec  H_\bullet^{{\bf G}(\mc O)\times {\bf G}_F(\mc O)}( \mc R( {\bf{G}}, {\bf{N}})) \qquad\text{and} \qquad  \Mu_C( {\bf{G}}, {\bf{N}})_{\tf_F}=\Spec  H_\bullet^{{\bf G}(\mc O)\times  {\bf T}_F(\mc O)}( \mc R( {\bf{G}}, {\bf{N}})) \]
over $\tf_F/W_F=\Spec H^\bullet_{{\bf G}_F}(\pt)$ and $\tf_F=\Spec H^\bullet_{{\bf T}_F}(\pt)$, respectively, where ${\bf T}_F\subset {\bf G}_F$ a maximal torus $T(W) \subset \GL(W)$.

In our example at hand, we have $\tf_F\cong \bb A^N$ for $N=\sum_{i=1}^{n-1} \lambda_i $, and we let
\begin{equation}\label{ulcaneqn}
	\ul = (\underbrace{\omega_1,...,\omega_1}_\text{$\lambda_{n-1}$ times},\underbrace{\omega_2,...,\omega_2}_\text{$\lambda_{n-2}$ times}, \ ... \ , \underbrace{\omega_{n-1},...,\omega_{n-1}}_\text{$\lambda_{1}$ times} ) 
\end{equation}
so that the corresponding Beilinson-Drinfeld generalized slice $ \Gru^\ul_\mu $ is also a variety over $\bb A^N$. The main result of \cite{Braverman:2016pwk} is the following:
\begin{thm}\label{BFNslicethm} There is an isomorphism of varieties over $\tf_F\cong \bb A^N$
\[ \Mu_C( {\bf{G}}, {\bf{N}})_{\tf_F}\xrightarrow{\cong }  \Gru^\ul_\mu   \ . \]
\end{thm}

One can also deduce an analogous corollary over  $\tf_F/W_F$: Let $\mf{S}_\ul=\prod_{i\in I} \mf{S}_{\lambda _i}$ the product of symmetric groups of size $\lambda_i$, so that $\tf_F/W_{F}\cong \bb A^N/ \mf{S}_\ul$.

We define $ \Gru_{\mu,\bb A^N/\mf{S}_\ul}^\ul$ as the variety over $\bb A^N/\mf{S}_\ul$ determined by precisely the same modular data as in the Definition \ref{BDgenslicedef} of $ \Gru_{\mu,\bb A^N}^\ul$, but where the tuple $(x_i)_{i\in I}$ of points $x_i \in \bb A^1$ are taken as unordered among each subset $I_i\subset I$ of size $\lambda_i$, the points at which the trivialization of the relevant $G$ bundle is required to be bounded by exactly $\omega_i$, so that they determine a point $(x_i)\in \bb A^I/ \mf{S}_\ul$.

\begin{cor}\label{BFNdeslicecoro} There is an isomorphism of varieties over $\tf_F/W_F \cong \bb A^N/\mf{S}_\ul$
	\[ \Mu_C( {\bf{G}}, {\bf{N}}) \xrightarrow{\cong} \Gru_{\mu,\bb A^N/\mf{S}_\ul}^\ul \ . \]
\end{cor}
We will return to these identifications in some particular examples of interest below.

\section{Applications to Lie theory in type \texorpdfstring{$A$}{A}}\label{Liethysec}

Let $\mu=(\mu_1 \geq ... \geq \mu_{n-1} \geq 0)$ be a dominant coweight of $\PGL_n$, the set of which we identify with partitions of length at most $n-1$, where each such $\mu$ is a partition of the integer $N=|\mu_i|=\sum_{i\in I} \mu_i$.

Throughout this section we let $G \equiv {\rm GL}_N$, with $\gf \equiv {\rm Lie}\, G = \mathfrak{gl}_N$ and fix a Cartan decomposition
$\gf = \nf_-\oplus \mathfrak{h}\oplus \nf$. Let $B$ and $B_-$ denote the upper and lower Borel subgroups with splittings, $ B = H \ltimes N $ and $B_-= H\ltimes N_-$. Similarly, let $\mathfrak{b}$ and $\mathfrak{b}_-$ be the upper and lower Borel subalgebras, \ie, ${\mathfrak b} \coloneqq \mathfrak{h} \oplus \nf$ and  ${\mathfrak b}_-\coloneqq \mathfrak{h} \oplus \nf_-$.
We also occasionally use a choice of Lie algebra splitting $\g=\glf_N \cong \slf_N \oplus \uf_1$. The Cartan decomposition is the decomposition of $\slf_N$, with the Cartan subalgebra augmented as $\hf = \hf_{\slf_N}\oplus \uf_1$. We also let $\gf^*$ be the linear dual to $\gf$, and use the Killing form to identify $\gf\xrightarrow{\cong}\gf^*$.

\begin{warn} In this section we refer to several generalized slices in the affine Grassmannian of the group $\PGL$, determined by the data of $\mu$ above. The group denoted $G$ in the previous section that was used in defining $\Gru_\mu^\ul$ will in this section always be taken equal to $\PGL_m$ for various $m\in \bb N$, while the notation $G$ will refer to the group $\GL_N$, as stated in the preceding paragraph.
\end{warn}

\subsection{Nilpotent orbits and their transverse slices}\label{ssec:nilpotent orbits}

Let $\NN\subset \gf^*$ be the nilcone of $\gf^*$. The nilpotent coadjoint orbits in $\gf^*$ are in bijection with the partitions of $n-1$ (according to the Jordan type). Let $\mu$ be such a partition, as above, and let $\chi_\mu=\langle f_\mu , \cdot \rangle$ be an element of the nilpotent coadjoint orbit $\IO_\mu$ determined by $f_\mu\in \g$ a nilpotent of Jordan type $\mu$. Note that the dominance order on coweights induces the order on partitions corresponding to the closure order on orbits.

The orbit $\IO_\mu$ admits a transversal slice $S_\mu$ at $\chi_\mu$---see, \eg, 3.2.19 in \cite{Chriss2010}. There are many such transversal slices at $\chi$, but a particularly nice and familiar class are the Slodowy slices, which have the following construction: By the Jacobson--Morozov theorem, we may complete $f=f_\mu$ to an $\slf_2$-triple $(h,e,f)$, and the Slodowy slice $S_\mu$ is defined as the image of the affine linear subspace
\begin{equation}
f + \ker\, {\rm ad}_e
\end{equation}
under the Killing isomorphism. Note that $\ker\, \ad_e$ is always transverse to $\im\, \ad_f=  T_{f} \IO_f$, so that $S_\mu$ defines a transverse slice. Changing the choice of $\slf_2$-triple gives isomorphic slices, related by $G$-conjugation.

Observe that if $\chi=0\in \IO_{[1^N]}$, the $\slf_2$ triple is $(0,0,0)$ and $S_{[1^N]}\cong \gf^*$. Alternatively, if $\chi\in \IO_{[N]}$, \ie is regular, then $S_{\mu}$ is called the \textit{principal} Slodowy slice and $S_{[1^N]}\cong \mf h/W$.

The transversal slices $S_\mu$ are Poisson, which is more apparent in their definition via Hamiltonian reduction. Given an $\slf_2$-triple, $(h,e,f)$, the operator ${\rm ad}_h$ gives a grading on $\gf$. which we assume is even without loss of generality in type $A$. We let $\gf_{>0}$ denote the Lie subalgebra that is strictly positively graded under ${\rm ad}_h$ and define 
\begin{equation}
	N_\chi \coloneqq {\rm Exp}\, \gf_{>0}~,
\end{equation}
the corresponding unipotent Lie group. Then the coadjoint action of $G$ restricts to a Hamiltonian $N_\chi$-action on $\g^*$ with moment map  $\mu_{N_\chi}:\g^* \to \gf_{>0}^*$ the projection, and we have
\begin{equation}
\begin{split}
	N_\chi \times S_\mu &\xrightarrow{\sim} \{\chi\} + \gf_{\ge0}^* = \mu_{N_\chi}^{-1}(\{\chi\}) \\
	(n,s) &\mapsto {\rm Ad}^*_n s
\end{split} \ ,
\end{equation}
the Gan--Ginzburg isomorphism \cite{GG}. Therefore, we see that
\begin{equation}
	S_\mu \cong \gf^*\red{\chi}N_\chi 
\end{equation}
and thus transverse slice $S_\mu$ inherits the Poisson structure induced by Hamiltonian reduction from the Kirillov--Kostant--Souriau Poisson structure on $\gf^*$.

The action of ${\rm ad}_h$ descends to the transverse slice, giving rise to the Kazhdan grading. On an element $s\in S_\mu$, the action of $t\in\IG_m$ is given by
\begin{equation}\label{eq:Kazhdan}
t\cdot s = t^{1-\tfrac12 \rm{ad}_h} s
\end{equation}
Note that the action, contracts to the fixed point $\chi\in S_\mu$. This action of $\IG_m$ scales the Poisson bracket with weight $-1$.

Finally, we recall the Grothendieck-Springer resolution
\[ \tilde \g : = T^*(G/N)/T \cong (G\times \mf b_-^*)/B \xrightarrow{\Ad^*} \g \qquad \text{and more generally} \qquad \tilde S_\mu:= \tilde \g \times_{\g} S_\mu  \to S_\mu   \]
the universal deformation of the symplectic resolution $T^*(G/B)\cong (G\times \mf n_-^*) \xrightarrow{\Ad^*} \mc N$.

\subsection{The Mirkovi\'c--Vybornov isomorphism}
In this section, we explain the fundamental results of \cite{MVK}, which are the basis of the connection between transverse slices $S_\mu$ and generalized slices in the affine Grassmanniany t, and in turn Coulomb branches of type $A$ quiver gauge theories, by the results of \cite{Braverman:2016pwk} recalled in Section \ref{Coulombsec} above.

Let $\mu=(\mu_1 \geq ... \geq \mu_{n-1} \geq 0)$ be a dominant coweight of $\PGL_n$, viewed as a partition of the integer $N=|\mu_i|=\sum_{i\in I} \mu_i$ of length at most $n-1$, as above. Let $\lambda = N \omega_1 \in \Lambda_{\PGL_n}^+$ so that following the notation of Equation \ref{lambdamumconveqn} we have
\[ \lambda-\mu = \sum_{i=1}^{n-1} ( N- \sum_{j=1}^{n-i}\mu_j ) \alpha_{n-i} \qquad \text{and thus}\qquad m_{i} = N- \sum_{j=1}^{n-i}\mu_j = \sum_{j=n-i}^{n-1} \mu_j  \]
so that $p_i := m_i-m_{i-1}=\mu_{n-i}$ for $i=2,...,n-1$, $p_1:=m_1=0$, $p_n:=N-m_{n-1}= \mu_1$, and in summary $\pi:=(p_1 \leq p_2 \leq \cdots \leq p_{n})$ the partition of Section 2.1 in \cite{WWY} is simply the partition $\mu$ thought of in ascending order.

\begin{warn}
Note that in writing $\mu=(\mu_1\geq ... \geq \mu_{n-1})$ we have departed from the conventions of \textit{loc. cit.}; their convention for the definition of $\mu_i$ is the one which we follow for $\lambda$ in writing $ \lambda =\sum_{i=1}^{n-1} \lambda_i \omega_{n-i}$.
Our notation has the benefit that $\mu_i$ simply give the components of the partition $\pi$, as noted above, which labels the relevant nilpotent orbit in the following theorem. Thus, we do not refer to the latter notation, and since we essentially always take $\lambda = N \omega_1$ or $0$ in the following applications, this discrepancy in convention between $\lambda$ and $\mu$ will not be especially relevant.
\end{warn}

Let $\lambda= N \omega_1$ as above so that in the notation of Section \ref{Coulombsec} above we have $\lambda_{n-1}=N$ and $\lambda_i=0$ otherwise, and thus in particular $\ul_N:= \ul=(\omega_1,...,\omega_1)$ in the sense defined in Equation \ref{ulcaneqn}. The fundamental results of \cite{MVK}, combined with Theorem \ref{BFNslicethm} above recalled from \cite{Braverman:2016pwk}, give:

\begin{thm}\label{MVtheo}
	There are isomorphisms of Poisson varieties over $\bb A^N \cong \mf h$ and $\bb A^N/\mf{S}_N \cong \mf h/W$,
	\[  \Mu_C( {\bf{G}}, {\bf{N}})_{\tf_F} \xrightarrow{\cong } \Gru_\mu^{\ul_N}  \xrightarrow{\cong } S_\mu \times_{\mf h / W} \mf h \qquad \text{and}\qquad  \Mu_C( {\bf{G}}, {\bf{N}})\xrightarrow{\cong }  \Gru_{\mu, \bb A^N/ \mf{S}_N}^{\ul_N} \xrightarrow{\cong } S_\mu   \ ,\]
	where $\Gru_{\mu, \bb A^N/ \mf{S}_N}^{\ul_N}$ is the variety over $\bb A^N/ \mf{S}_N$ as defined preceding Corollary \ref{BFNdeslicecoro}, where ${\bf{G}}$ and ${\bf{N}}$ are as in Section \ref{Coulombsec} with 
\begin{equation}
		\dim V  =(0 \ , \ \mu_{n-1}\ ,\ \mu_{n-2}+\mu_{n-1}\ , \ ... \ ,\mu_2+...+\mu_{n-2}+\mu_{n-1}) 
\end{equation}
	and $\dim W=(0 \ , \ ... \ , \ 0 \ , \ N)$.
\end{thm}

The following compatible identification of the natural resolutions was also established in \emph{loc. cit.}.
\begin{thm}\cite{MVK} There is an isomorphism of Poisson varieties as in the first column of the following diagram
	\begin{equation}\begin{tikzcd}
			\widetilde{\Gru}_\mu^{\ul_N} \arrow[r] \ar[d,"\cong"] & \Gru_\mu^{\ul_N} \arrow[r]\ar[d,"\cong"] &   \Gru^{\ul_N}_{\mu, \bb A^N/ \mf{S}_N} \ar[d,"\cong"] \\
			\tilde S_\mu  \arrow[r] &  S_\mu \times_{\mf h / W} \mf h \arrow[r]  &  S_\mu 
		\end{tikzcd}
	\end{equation}
	such that it naturally commutes.
\end{thm}

\subsection{Equivariant Slodowy slices}\label{ssec:equiv slodowy}

The transversal slices $S_\mu$ are Poisson but not symplectic. One way to remedy this is to resolve them via the Grothendieck--Springer resolution. Alternatively, we can consider the equivariant Slodowy slices of \cite{Losev2007:quant}. 

In the case that $\mu=N\omega_1$ corresponding to a principal nilpotent in $\gl_N$, we have $S_{G,[1^n]}=T^*G$ with its canonical symplectic structure. Viewing $T^*G$ in its \emph{right} trivialization, we have  $T^*G\cong \gf^*\times G$, with left and right moment maps $\mu_L:\gf^*\times G$ and $\mu_R:\gf^*\times G \rightarrow G$, where 
\begin{equation}
	\mu_L(x,g)= x~, \quad \mu_R(x,g) = {\rm Ad}^*_g x~,
\end{equation}
for $(x,g)\in \gf^*\times G$.

More generally, for any $\chi\in \IO_\mu$, let $N_\chi$ be as in Subsection \ref{ssec:nilpotent orbits}. The equivariant Slodowy slice is the Hamiltonian reduction
\begin{equation}
	S_{G,\mu}\coloneqq T^*G\red{\chi} N_\chi ~,
\end{equation}
using the left moment map. Since the left moment map is just projection to $\gf^*$, we see that
\begin{equation}
	T^*G\red{\chi} N_\chi \cong G\times S_\mu~.
\end{equation}
hence the name, equivariant Slodowy slice. Since the moment map is flat and its fibers are $N_\chi$-torsors, each reduction is smooth and inherits a symplectic structure from $T^*G$.

Furthermore, since we are in the right trivialization, the right action of $G$ on itself leaves the cotangent fibers invariants. Therefore, $G$ acts freely on the right on each $S_{G,\mu}$ and
\begin{equation}
	S_{G,\mu}/G \cong S_{\mu}~.
\end{equation}

Note that $T^*G$ has a natural action of $\IG_m$ that scales the fibers with weight $1$. The reductions $S_{G,\mu}$ inherit a modified version of this $\IG_m$-action, via the Kazhdan grading---which extends straightforwardly to $\gf^*\times G$.


\subsection{Zastava presentation of equivariant Slodowy slices in type \texorpdfstring{$A$}{A}}

The cotangent bundle $T^*G$ for $G=\GL_N$ has a realization as a Zastava space $Z^{\alpha_0^*}_{\PGL_{2N}}$ by \cite{BielP}, thus in particular a generalized slice in the affine Grassmannian of $\PGL_{2N}$. Further, by the Theorem \ref{BFNslicethm} recalled above from \cite{Braverman:2016pwk}, $T^*G$ is also realized as the Coulomb branch of a type $A_{2N}$ quiver, a linear quiver of length $2N-1$. We now state this formally for future reference:

\begin{thm}\label{bfnctgtGthm} For $\alpha= 2\omega_N \in \Lambda_{\PGL_{2N}}^{\geq 0}$, there are isomorphisms of Poisson varieties
\[ \mc M_C({\bf{G}},{\bf{N}})\xrightarrow{\cong}	\Gr_{-\alpha}^0 \xrightarrow{\cong} T^* \GL_N  \ , \]
where ${\bf{G}}$ and ${\bf{N}}$ are as in Section \ref{Coulombsec} with $\dim V=(1,2,...,N-1,N,N-1,...,2,1)$ and $\dim W=0$.
\end{thm}

From its presentation as the classical limit of a truncated shifted Yangian, see for example \cite{FKPRW}, we see that $\OO(\Gr^0_{-\alpha_{2\omega_N}})$ admits two morphisms \[\mu_L^\#: {\rm Sym}\,\gf \rightarrow \OO(\Gr^0_{-\alpha_{2\omega_N}}) \qquad \text{and}\qquad \mu_R^\#: {\rm Sym}\,\gf \rightarrow \OO(\Gr^0_{-\alpha_{\omega_n}}) \] of Poisson algebras, arising from the composition $U(\gf)\otimes U(\gf)\hookrightarrow U({\glf}_{2n})\rightarrow Y(\glf_{2n})$. More precisely, the left comoment map $\mu_L$ sends the Chevalley generators of $\gf$ to the Drinfeld modes $(E_i^{(1)},H_i^{(1)},F_i^{(1)})$ for $i=1,\dots,n$ while the right comoment map, $\mu_R^\#$ sends them to $(E_{n+i}^{(1)},H_{n+i}^{(1)},F_{n+i}^{(1)})$. The Serre relations imply that the images of these comoment maps are mutually Poisson commuting subalgebras of $\OO(\Gr^0_{-\alpha_{2\omega_N}})$. These are obvious candidates for the natural left and right moment maps of $T^*G$. Indeed, the discussion in Section $5(i)$ and Appendix A of \cite{Braverman:2017ofm} confirms this.

There is also a closely related description of the principal equivariant Slodowy slice, which follows from similar considerations in \cite{Braverman:2017ofm}, via the results of \cite{BielP}:

\begin{thm}\label{BFNEqPrinthm} For $\alpha=(N+1)\omega_N \in \Lambda_{\PGL_{N+1}}^{\geq 0}$, there are isomorphisms of Poisson varieties
\[ \mc M_C({\bf{G}},{\bf{N}})\xrightarrow{\cong}	\Gr_{-\alpha}^0 \xrightarrow{\cong} \GL_N \times S_{\textup{prin}}  \ , \]
where ${\bf{G}}$ and ${\bf{N}}$ are as in Section \ref{Coulombsec} with $\dim V=(1,2,...,N-1,N)$, $\dim W=0$, and $S_{\textup{prin}} \subset \gl_N$ is the principal Slodowy slice.
\end{thm}

In order to deduce the analogous statement for more general equivariant Slodowy slices, we recall more of the general perspective around the preceding theorems, again following \cite{Braverman:2017ofm}.

First, we recall that by construction there exists a canonical ${\bf{G}}(\mc O)$-equivariant map
\[ \mc R({\bf{G}},{\bf{N}}) \xrightarrow{\pi_{F}}  \Gr_{\GL_N}  \]
where ${\bf{G}}$ and ${\bf{N}}$ are as in the preceding Theorem \ref{BFNEqPrinthm}, and in turn the definitions recalled in briefly in Section \ref{Coulombsec} above from \cite{BFN1}. Since the resulting representation is trivial on the natural diagonal character $\bb C^\times \to {\bf{G}}=\GL(V)=\prod_{i=1}^N \GL_i$, the above construction naturally descends to define a $\widetilde{\bf{G}}$-equivariant map $\mc R(\widetilde{\bf{G}},{\bf{N}}) \xrightarrow{\tilde \pi_{F}}  \Gr_{\PGL_N}$
where $\widetilde{\bf{G}}={\bf{G}}/\bb C^\times$. In summary, 
we obtain sheaves
\[\mc A_{\reg,\GL_N}:= \pi_{F,*} \omega_{\mc R(\bf{G},{\bf{N}})} \in D_{\GL_N(\mc O)}(\Gr_{\GL_{N}}) \qquad \text{and}\qquad  \mc A_{\reg} := \tilde \pi_{F,*} \omega_{\mc R(\widetilde{\bf{G}},{\bf{N}})} \in D_{\PGL_N(\mc O)}(\Gr_{\PGL_{N}}) \ , \]
which are by construction algebra objects with respect to the convolution monoidal structure on $D_{G(\mc O)}(\Gr_G)$, such that the induced equivariant cohomology algebras are given by
\begin{align}
 \Spec	H^\bullet(\Gr_{\GL_N} , \mc A_{\reg,\GL_N})  & =\mc M_C({\bf{G}},{\bf{N}}) \xrightarrow{\cong }\GL_N \times S_{\textup{prin}}  & \text{and}\\ 
 \Spec H^\bullet(\Gr_{\PGL_N} , \mc A_{\reg}) & =\mc M_C(\widetilde{\bf{G}},{\bf{N}}) \xrightarrow{\cong } \SL_N\times S_{\textup{prin},\slf_N} & \ , 
\end{align}
essentially by construction.

%


For a partition $\mu$ of $N$, consider the positive coroot $\alpha_\mu$ of ${\rm PGL}_{2N}$ defined by 
\begin{equation}\label{eq:equiv Slodowy as Zastava}
	\alpha_\mu = 2\omega_N -\mu 
\end{equation}
where we think of $\mu$ as a weight of ${\rm PGL}_{2N}$ by including it into the first $N\times N$ block. It will also be useful to introduce a notation for a corresponding sheaf
\[  \mc A_{\mu} := \tilde \pi_{F,*} \omega_{\mc R(\widetilde{\bf{G}}_\mu,{\bf{N}}_\mu)} \in D_{\PGL_N(\mc O)}(\Gr_{\PGL_{N}})\]
where ${\bf{G}}_\mu$ and ${\bf{N}}_\mu$ are given by the pair ${\bf{G}},{\bf{N}}$ from Theorem \ref{MVtheo}, and $\widetilde{\bf{G}}_\mu=({\bf{G}}_\mu\times \GL_N)/\bb C^\times $, where we note that again the diagonal character acts trivially on ${\bf{N}}_\mu$.

The following provides the bridge between the affine Grassmannian and equivariant W-algebras that we shall leverage to prove our main theorem.

\begin{thm}\label{Eqmuthm}
For $\alpha_\mu\in \Lambda_{\PGL_{N+n}}^{\geq 0 }$ as in \eqref{eq:equiv Slodowy as Zastava}, there are isomorphisms of Poisson varieties
\begin{equation}
\mc M_C({\bf{G}},{\bf{N}})\xrightarrow{\cong}	\Gr_{-\alpha_\mu}^0 \xrightarrow{\cong} \GL_N \times S_{\mu}   \ ,
\end{equation}
where ${\bf{G}}$ and ${\bf{N}}$ are as in Section \ref{Coulombsec} with
\[ \dim V =  (\ 0  \ , \  \mu_{n-1}\ ,\ \mu_{n-2}+\mu_{n-1}\ , \ ... \ ,\mu_2+...+\mu_{n-2}+\mu_{n-1} \ , \  N  \ , \ N-1 \ , \ N-2 \ , \ ... \ , \ 1 \ )    \]
and $\dim W=0$.
\end{thm}
\begin{proof} The first isomorphism is again an example of Theorem \ref{BFNslicethm} recalled from \cite{Braverman:2016pwk}, so it remains to prove the latter. Using the former description as a Coulomb branch, we have a natural identification
\[  \mc M_C({\bf{G}},{\bf{N}})= \Spec \mc A({\bf{G}},{\bf{N}}) = T^*\bb C^\times \times \Spec  H^\bullet_{\PGL_N(\mc O)}( \Gr_{\PGL_N}, \iota_{\Delta}^!(\mc A_\mu \boxtimes \mc A_R) ) \ . \]
We will compute the right hand side using aspects of the derived geometric Satake equivalence of \cite{BezF}, again following the arguments of \cite{Braverman:2017ofm}.

To begin, note that there is a manifest $\GL_N$ action on $\mc M_C({\bf{G}},{\bf{N}})$ as a factor of the group $L_\textup{bal}$ corresponding to the balanced subquiver $(N,N-1,N-2,..., 1)$. Moreover, this defines a trivial $\GL_N$-torsor, as we can directly calculate
\begin{align*}
	 \Spec H^\bullet_{\PGL_N(\mc O)}( \Gr_{\PGL_N}, \iota_{\Delta}^!(\mc A_\mu \boxtimes \mc A_R) )  & \cong \Spec  H^\bullet_{\PGL_N(\mc O)}( \Gr_{\PGL_N}, \iota_{\Delta}^!(\mc A_\mu)) \times_{S_\textup{prin} \times_{[\slf_N/\SL_N]} S_\textup{prin} }  (\SL_N \times S_\textup{prin})  \\
	 & \cong \left( \Spec  H^\bullet_{\PGL_N(\mc O)}( \Gr_{\PGL_N}, \iota_{\Delta}^!(\mc A_\mu)) \times_{S_\textup{prin} \times_{[\slf_N/\SL_N]} S_\textup{prin} }  S_\textup{prin}\right) \times\SL_N 
\end{align*}
Moreover, we can thus compute the corresponding object in $\HC_0(\SL_N)$ explicitly as follows: in the notation of \cite{Braverman:2017ofm}, by Lemma 5.13 of \emph{loc. cit.} we have
\[ H^\bullet_{\PGL_N(\mc O)}( \Gr_{\PGL_N}, \iota_{\Delta}^!(\mc A_\mu \boxtimes \mc A_R) ) \xrightarrow{\cong }  \Phi(\mc A_\mu) = \mf{C}_{\SL_N}\circ \Psi^{-1}(\mc A_\mu) \]
where $\mf{C}_{\SL_N}:\HC_0(\SL_N) \to \HC_0(\SL_N)$ denotes the Cartan involution, and in turn
\begin{align*}
\Psi^{-1}(\mc A_\mu) & \xrightarrow{\cong}  \mathbb{R}\Hom_{\HC_0(\SL_N)}( \Sym^\bullet(\slf_N), \Psi^{-1}(\mc A_\mu) \otimes \bb C[T^*\SL_N]   )  \\ 
&   \xrightarrow{\cong}  \mathbb{R}\Hom_{\HC_0(\SL_N)}( \Sym^\bullet(\slf_N), \Psi^{-1}(\mc A_\mu) ) \otimes_{\Sym^\bullet(\slf_N)} \bb C[T^*\SL_N]   \\ 
& \xrightarrow{\cong}  \mathbb{R}\Hom_{D_{\PGL_N(\mc O)}(\Gr_{\PGL_N})}( \one_{\Gr_{\PGL_N}} , \mc A_\mu) \otimes_{\bb C} \bb C[\SL_N]
\end{align*}
where the first isomorphism is also from the proof of Lemma 5.13 of \emph{loc. cit.}, the second follows since the first tensor factor is a trivial $\SL_N$-torsor, and the third by using the inverse Satake equivalence and the identification $\bb C[T^*\SL_N]\cong \Sym^\bullet(\slf_N)\otimes_{\bb C} \bb C[\SL_N]$. Further, by the definition of the unit $\one_{\Gr_{\PGL_N}}=\iota_* \bb K$ we have
\[ \mathbb{R}\Hom_{D_{\PGL_N(\mc O)}(\Gr_{\PGL_N})}( \one_{\Gr_{\PGL_N}} , \mc A_\mu) \cong H_\bullet^{({\bf{G}}_\mu\times \GL_N)/\bb C^\times }( \mc R( {\bf{G}}_\mu, {\bf{N}}_\mu)  )\cong  S_{\mu,\slf_N } \]
where ${\bf{G}}_\mu$ and ${\bf{N}}_\mu$ are as in Theorem \ref{MVtheo}, so that the final isomorphism follows from \emph{loc. cit.}, though noting that restriction of equivariance to the descended action of $({\bf{G}}_\mu\times \GL_N)/\bb C^\times $ corresponds to the restriction to $S_{\mu,\slf_N } = S_\mu \times_{\gl_N^\vee} \slf_N^\vee$. In summary, we obtain
\[  \mc M_C({\bf{G}},{\bf{N}}) = T^*\bb C^\times \times \Spec  H^\bullet_{\PGL_N(\mc O)}( \Gr_{\PGL_N}, \iota_{\Delta}^!(\mc A_\mu \boxtimes \mc A_R) )\cong T^*\bb C^\times \times \SL_N \times  S_{\mu,\slf_N } = \GL_N \times S_\mu  \ ,  \]
as desired.
\end{proof}

We also recall the following definition of the general genus zero Moore--Tachikawa varieties, with punctures labeled by dominant coweights $\mu^1,...,\mu^b \in \Lambda^+_{\PGL_{N}}$ with $|\mu^k|=N$ for each $k=1,...,b$, or $\mu^k=0$ in which case the puncture is called `full' and corresponds to those discussed in \cite{Braverman:2017ofm}:

\begin{dfn}\label{MTdefn} The Moore--Tachikawa variety $W_{\mu^1 ,..., \mu^b}(\SL_N)$ associated with the dominant coweights $\mu^1,...,\mu^b \in \Lambda^+_{\PGL_{N}}$ is the affine Poisson scheme defined by
\[ \bb C[ W_{\mu^1 ,..., \mu^b}(\SL_N) ] = H^\bullet_{\PGL_N(\mc O)}(\Gr_{\PGL_N}, \iota_{\Delta}^!( \boxtimes_{k=1}^b \mc A_{\mu^k} ))  \ .  \]
\end{dfn}

We can view these varieties as determined inductively by a two-dimensional topological field theory, according to the following:

\begin{prop}\label{MTglueprop} There is a canonical isomorphism of Poisson varieties
\[  \left( W_{\mu^1 ,..., \mu^{b_1-1},0}(\SL_N) \times_{\slf_N^\vee} W_{\mu^{b_1} ,..., \mu^{b_1+b_2-2},0}(\SL_N)\right)/ \SL_N \xrightarrow{\cong} W_{\mu^1 ,..., \mu^{b_1+b_2-2}}(\SL_N) \]
\end{prop}

\section{Quantization and finite W-algebras}

\subsection{Quantization in the algebraic setting}

First, let us review some preliminaries on quantization. We largely follow the language of \cite{Bezrukavnikov2004:Fed, Losev2012:quantresolv}. For now, suppose that $Y$ is a smooth, symplectic, affine variety over $\ik$. By definition, $\OO_Y$ is a sheaf of Poisson algebras over $Y$. Suppose also that $Y$ has an action of $\IG_m$ such that the Poisson tensor is scaled with weight%
\footnote{Usually one asks that $\IG_\hbar$ scales the symplectic form and $\hbar$ with weight $2$, but one is then forced to change variables to $\sqrt{\hbar}$. We opt to avoid this by halving the weights. In characteristic zero, this amounts to little more than a matter of convention.}
$-1$, we use $\IG_{\hbar}$ for this distinguished action. This group also also acts on $\ik\Ph$ by scaling $\hbar$ with unit weight. Much of this discussion holds in a more general setting but in the spirit of brevity, we specialize to the setting we need. We use the term quantization for what we call a formal quantization in \cite{BuN3}.

\begin{dfn}
	A \emph{quantization} of the Poisson algebra $\OO(Y)$ is a pair $(A,\varphi)$ of an associative $\ik\Ph$ algebra $A$, flat over $\ik\Ph$ and complete in the $\hbar$-adic topology, and an isomorphism of Poisson algebras $\varphi: A/\hbar A \rightarrow \OO(Y)$. 
	

	A \emph{graded} quantization is a pair $(A,\varphi)$ as above together with an action of $\IG_\hbar$ on $A$ by algebra automorphisms, such that $\IG_\hbar$ scales $\hbar$ with weight $1$ and $\varphi: A/\hbar A \rightarrow \OO(Y)$ is $\IG_\hbar$-equivariant.
\end{dfn}
Often, we will omit the isomorphism $\varphi$ in our notation when it is not relevant. These definitions have natural sheaf-theoretic analogues in quantizations of the structure sheaf $\OO_Y$:

\begin{dfn}
	A \emph{quantization} of $\OO_Y$ is a sheaf $\AA$ on $Y$ of associative, flat, $\ik\Ph$ algebras, complete in the $\hbar$-adic topology, and an isomorphism $\AA/\hbar \AA\rightarrow \OO_Y$ of sheaves of Poisson algebras.
	
	A \emph{graded} quantization of $\OO_Y$ is a pair $(\AA,\varphi)$ as above, together with $\IG_\hbar$-equivariant structure on the sheaf $\AA$ such that $\IG_\hbar$ scales $\hbar$ with weight $1$ and $\varphi$ is $\IG_\hbar$-equivariant. 
\end{dfn}
Note that the global sections of a graded quantization of $\OO_Y$ are a graded quantization of $\OO(Y)$.

\begin{prop}[\protect{\cite[Proposition 2.3.2]{Losev2012:quantresolv}}]\label{prop:classifying quantizations}
The moduli space of isomorphism classes of graded quantizations of $\OO_Y$ is isomorphic to the second de Rham cohomology $H^2_{\rm dR}(Y)$.
\end{prop}
This is the familiar result for quantizations of symplectic manifolds in the $C^\infty$ setting, see for example \cite{deWilde1983, Fedosov1994:quant}. The result holds, in the algebraic context, for more general varieties but we only need the affine case.

Suppose $Y$ is affine and symplectic, then there is a distinguished isomorphism class of graded quantizations, corresponding to $0\in H^2_{\rm dR}(Y)$, any such quantization is called \textit{canonical}.

Henceforth, we assume that every symplectic or Poisson scheme $Y$ has an appropriate action of $\IG_\hbar$ and all quantizations are graded. A map between graded quantizations $\AA\rightarrow \AA'$ is a $\IG_\hbar$-equivariant map of sheaves of associative $\ik\Ph$-algebras, such that in the limit $\hbar=0$ it is the identity on $\OO_Y$. 

\subsection{Strongly equivariant quantizations}

For now, let $H$ be an affine, algebraic Lie group with Lie algebra $\hf$. Suppose $Y$ has a Hamiltonian action of $H$ with moment map $\Phi:Y\rightarrow \hf^*$. This gives rise to a comoment map $\Phi^\#: {\rm Sym}\,\hf\rightarrow \Gamma(Y,\OO_Y)$ such that the Hamiltonian vector fields corresponding to the image of $\hf\subset {\rm Sym}\,\hf$ exponentiate to the action of $H$. Borrowing terminology from $D$-modules, we say that the algebra $\OO(Y)$ is strongly $H$-equivariant. We want to consider quantizations that are compatible with the data of both the $H$-action and the moment map, \ie, we need the notion of a strongly $H$-equivariant quantization.

Write $U_\hbar(\hf)$ for the homogenized universal enveloping algebra of $\hf$, \ie, the graded tensor $\ik\Ph$-algebra on $\hf$ subject to the relation $x\otimes y -y\otimes x = \hbar[x,y]$ for any $x,y\in \hf$. 

A graded quantization $(A,\varphi)$ of $\OO(Y)$ is $H$-equivariant, if $H$ acts by grading preserving isomorphisms and $\varphi: A/\hbar A \rightarrow \OO(Y)$ is $H$-equivariant. A graded quantization $(\AA,\varphi)$ of $\OO_Y$ is $H$-equivariant if it is a $H$-equivariant sheaf on $Y$, with $H$ acting by algebra automorphisms and $\varphi: \AA/\hbar\AA\rightarrow \OO_Y$ is $H$-equivariant.

In both cases, by differentiating the $H$-action, we have maps of Lie algebras $d\rho: \hf\rightarrow {\rm Der}(A)$ or $d\rho:\hf\rightarrow {\rm Der}(\Gamma(Y,\AA))$.

\begin{dfn}
An $H$-equivariant quantization $(A,\varphi)$ is said to be strongly $H$-equivariant if it has a comoment map, \ie, a $H$-equivariant morphism of associative $\ik\Ph$-algebras, $\Phi^\#:U_\hbar(\hf)\rightarrow A$, such that for any $x\in\hf$, its image under the differential of the action, $d\rho(x)$ agrees with the derivation $[\Phi^\#(x),\cdot]$.

A $H$-equivariant quantization, $(\AA,\varphi)$, of $\OO_Y$ is strongly $H$-equivariant if its global sections are a strongly $H$-equivariant quantization.
\end{dfn}

We say that a map, $A\rightarrow A'$, of strongly equivariant algebras (Poisson or associative) is strongly $H$-equivariant if it is $H$-equivariant and intertwines the comoment maps, \ie the following diagram commutes
\begin{equation}
	\begin{tikzcd}
	A\ar[r] & A' \\
	U_\hbar(\hf) \ar[u] \ar[r,"{\rm id}"] & U_\hbar(\hf) \ar[u]
	\end{tikzcd}
\end{equation}

Such strongly equivariant quantizations have been classified in the $C^\infty$ setting via Fedosov-style arguments by \cite{Reichert2016:strongequiv}. Recently, a deformation theory argument has appeared in \cite{Esposito2025:Equiv}. However, the assumptions required in the $C^\infty$ setting do not survive into the algebraic setting%
\footnote{The authors wish to thank I.\ Losev for bringing this to our attention}

To this end, we prove a version of equivariant deformation theory in \cite{BuN3}, and we recall the main result here.

\begin{prop}\label{prop:classifying strong equiv}
Suppose a strongly $H$-equivariant, graded quantization of $Y$ exists. Then the moduli space of such quantizations, up to isomorphism, is a torsor for the second $H$-equivariant cohomology
\begin{equation}
H^2_{H}(Y)~.
\end{equation}
In particular,  there exists a strongly $H$-equivariant, graded isomorphism between any two strongly $H$-equivariant, graded quantizations $A$ and $A'$ with the same equivariant cohomology class.
\end{prop}


\begin{prop}[\cite{BuN3}]\label{prop:BuN3 vanishing of obstruction}
Suppose $H$ is unipotent and abelian, and the Hamiltonian action of $H$ on $Y$ is such that
\begin{enumerate}[(i)]
	\item the action of $H$ is free, and
	\item the moment map $\Phi:Y\rightarrow \hf^*$ is smooth.
\end{enumerate}
Then every graded quantization of $Y$ can be made strongly $H$-equivariant. Moreover, such a lift is unique up to strongly $H$-equivariant isomorphism.
\end{prop}



\subsection{Quantum Hamiltonian reduction}

Let $\chi \in (\hf^*)^H$ be a character of $\hf$ and suppose that $H$ acts on $\Phi^{-1}(\chi)$ freely, then we may define the symplectic reduction
\begin{equation}
	Y/\!\!/\!\!/_\chi H \coloneqq \Phi^{-1}(\chi)/H~.
\end{equation}
Thanks to the freeness hypothesis, the reduced space is smooth, symplectic and we denote its structure sheaf as $\OO_{Y}/\!\!/\!\!/_\chi H$, which is a sheaf of Poisson algebras.

We need a quantized version of this reduction that produces a quantization of $\OO_{Y/\!\!/\!\!/_\chi H}$ from a quantization of $\OO_Y$. We follow the exposition of \cite{Losev2012:quantresolv}.

At the level of global sections, we have the following construction. Let $(A,\varphi)$ be a strongly equivariant quantization with comoment map $\Phi^\#)$. Consider the subspace, $\hf_\chi \subset U_\hbar(\hf)$, of elements of the form $x- \chi(x) \hbar$ for $x\in \hf^*$ and write $\langle \Phi^\#(\hf_\chi)\rangle$ for the two sided ideal generated by its image under $\Phi^\#$.

The quantum Hamiltonian reduction of $A$ by $H$ at $\chi$ is
\begin{equation}
	A\red{\chi}H \coloneqq \big(A/\langle\Phi^\#(\hf_\chi)\rangle\big)^H~.
\end{equation}
It is not true in general that $A\red{\chi}H$ is a quantization of $Y\red{\chi}H$. However, for a nice enough $H$-action this is the case---we shall give sufficient conditions shortly.

Now let us consider the sheaf theoretic version of this construction. Let $(\AA,\varphi,\Phi^\#)$ be a strongly $H$-equivariant quantization of $\OO_Y$. The ideal $\langle\Phi^\#(\hf)\rangle$ gives rise to a sheaf, $\II_{\chi}$ of $\AA$-modules. The quotient sheaf $\AA/\II_\chi$ is a sheaf of $\AA$-modules supported, by construction, on $\Phi^{-1}(\chi)$. Let $p_H: \Phi^{-1}(\chi)\rightarrow \Phi^{-1}(\chi)/H$ be the natural projection, then
\begin{equation}
	\AA/\!\!/\!\!/_\chi H \coloneqq p_{H,*}\big(\AA/\II_\chi\big)
\end{equation}
is a sheaf on $Y/\!\!/\!\!/_\chi H$, called the quantum Hamiltonian reduction of $\AA$. The reduction, $\AA\red{\chi} H$ is a sheaf of equivariant associative algebras, flat over $\ik\Ph$ and complete in the $\hbar$-adic topology. However, it is not clear that its classical limit agrees with $\OO_{Y}\red{\chi} H$.

\begin{prop}\label{prop:quantization commutes with reduction}
	Suppose $H$ acts on $Y$ freely and the moment map $\Phi:Y\rightarrow \hf^*$ is flat in a neighborhood around $\chi$. Then, if $(\AA,\varphi,\Phi^\#)$ is any strongly $H$-equivariant quantization of $\OO_Y$, $\AA\underset{\chi}{/\!\!/\!\!/}H$ is a quantization of $\OO_{Y_{\rm KW}}$. Moreover, if $A$ is any strongly $H$-equivariant quantization of $\OO(Y)$, then $A\red{\chi} H$ is a quantization of $\OO(Y\red{\chi}H)$.
\end{prop}
Therefore, when the $H$-action is sufficiently nice, quantization commutes with reduction.

\subsection{Equivariant W-algebras}

Consider the cotangent bundle $T^*G$, with its natural action of $\IG_\hbar$ that scales the fibers with weight $1$. Moreover it has an obvious graded quantization given by $\DD_\hbar(G)$, the sheaf of homogeneous differential operators on $G$.

Let $\mu$ be some partition of $n$ and $\IO_\mu$ the corresponding nilpotent orbit in $\gf^*$. Pick a representative $\chi\in \IO_\mu$ and construct $N_\chi$ as in Section \ref{ssec:nilpotent orbits}. Since the action of $N_\chi$ is everywhere free and the moment map is flat, Proposition \ref{prop:quantization commutes with reduction} implies that,
\begin{equation}
	\AA_{G,\mu}\coloneqq \DD_{\hbar,G} \red{\chi}N_\chi~,
\end{equation}
is a quantization of the equivariant Slodowy slice $S_{G,\mu}$. In fact, $\AA_{G,\mu}$ is a graded quantization where the $\IG_\hbar$-action comes from the Kazhdan grading of \eqref{eq:Kazhdan}.

Its global sections are 
\begin{equation}
	\WW_{\hbar,G,\mu}\equiv \Gamma(S_{G,\mu}, \AA_{G,\mu}) ~,
\end{equation}
the finite equivariant, $\hbar$-adic W-algebra of \cite{Losev2007:quant}. This is a graded quantization of $S_{G,\mu}$ which is weakly $G$-equivariant for the right action of $G$ on $S_{G,\mu}$. 

We recover the usual ($\hbar$-adic) finite W-algebra by taking invariants with respect to the right $G$ action:
\begin{equation}
\WW_{\hbar,\mu} \cong \WW_{\hbar,G,\mu}^G~,
\end{equation}
and we may recover the non $\hbar$-adic finite W-algebra by
\begin{equation}
	\WW_{\mu} \cong \big(\WW_{\hbar,\mu}\tensor{\ik\Ph} \ik(\!(\hbar)\!)\big)^{\IG_m}~,
\end{equation}
since the Kazhdan grading is bounded below.

\begin{lem}\label{lem:equiv W is unique}
For any partition $\mu$, $\AA_{G,\mu}$ is the unique graded quantization of $S_{G,\mu}$.
\end{lem}
\begin{proof}
	The space of graded quantizations of $S_{G,\mu}$ is given by $H^2_{\rm dR}(S_{G,\mu})$. But $S_{G,\mu}$ contracts to $G$ and so $H^2_{\rm dR}(S_{G,\mu})= H^2_{\rm dR}(G)=0$. Thus, there is a unique isomorphism class of quantizations.
\end{proof}

\subsection{Quantizing Zastava}

We shall also require quantizations of the Zastava spaces from Section \ref{ssec:Zastava}. Once again, let $\alpha$ be a positive coroot in ${\rm PGL}_n$ of height $k$.

From Lemma \ref{lem:antidominant_abelian}, we note that ${\rm Gr}^0_{-\alpha}$ is an open symplectic subvariety of $T^*\IG_{\alpha}$ and so has a natural quantization given by localizing the sheaf of differential operators on $\IG_\alpha$.

Write $\DD_{\hbar,\alpha}$ for the sheaf of (homogenized) differential operators on $T^*\IG_\alpha$. We denote its restriction to $\Gr^0_{-\alpha}$ by $\DD_{\hbar, \alpha}^{\rm loc}$. Its global sections, $D_{\hbar,\alpha}^{\rm loc}=\Gamma(\Gr^0_{-\alpha}, \DD_{\hbar,\alpha}^{\rm loc})$ is the free, associative $\ik\Ph$-algebra on the generators $(p,e^\pm, b_i,g_i)$ subject to the relations
\begin{equation}
\begin{split}
	e^+ e^- &=  e^- e^+ = 1\\
	[p,e^\pm] = \pm \hbar e^\pm ~&,  \qquad [b_i,g_i] = \hbar ~.
\end{split}
\end{equation}
The Zastava has a natural action of $\IG_\hbar$ by loop rotation, which scales $p,b,e^+$ with weight $1$ and leaves the $g_i$ invariant. This corresponds to the restriction of the natural $\IG_\hbar$-action that scales the cotangent fibers of $T^*\IG_\alpha$ with weight $1$. Clearly, this action preserves the algebra structure on the global sections.

\begin{lem}
Up to isomorphism, $\DD_{\hbar,\alpha}^{\rm loc}$ is the unique, graded quantization of $\Gr^0_{-\alpha}$.
\end{lem}
\begin{proof}
The space of graded quantizations of $\Gr^0_{-\alpha}$ is isomorphic to $H^2_{\rm dR}(\Gr^0_{-\alpha})$ by Proposition \ref{prop:classifying quantizations}. But, $\Gr^0_{-\alpha}$ contracts to $S^1$ and so its second cohomology vanishes and there is only one quantization up to isomorphism.
\end{proof}


\section{Inverse Hamiltonian reduction}

Now let us specialize so that $H$ is abelian. The $\ik$-scheme $Y$ continues to be affine, symplectic, and equipped with a Hamiltonian action of $H$, with moment map $\Phi:Y\rightarrow \hf^*$. Moreover, we continue to assume that $Y$ is equipped with an action of $\IG_\hbar$, that scales the symplectic form with weight $-1$. 

\begin{rem}
Consider $T^*H$ with its canonical symplectic structure. The left moment map $T^*H\rightarrow \hf^*$, is $H$-equivariant and Poisson. Suppose $U$ is some open Poisson subset of $\hf^*$, then $H\times U$ has a natural Poisson structure arising from viewing it as the preimage of $U$ under the left moment map of $T^*H$. It is equipped with a strong action of $H$, inheriting the moment map of $T^*H$. Moreover, if $U$ is $\IG_\hbar$-stable, for the action that scales $\hf^*$ with unit weight, then $H\times U$ inherits an action of $\IG_\hbar$. Henceforth, $H\times U$ is assumed to have this symplectic structure in the rest of this section.
\end{rem}

\begin{dfn}\label{dfn:IHR}
 An inverse Hamiltonian reduction of $Y\red{\chi} H$ is a pair $(U,f)$, where 
\begin{enumerate}[(i)]
	\item $U\subset\hf^*$ is an open, $\IG_\hbar$-stable Poisson subvariety, such that $\chi\in U$.
	\item $f:H\times U \times Y\red{\chi} H \xrightarrow{\sim} Y\fibre{\Phi,\hf^*}U$ is a strongly $H$-equivariant, isomorphism of Poisson schemes. 
\end{enumerate}
\end{dfn}
We say that $Y\red{\chi}H$ admits an inverse Hamiltonian reduction (IHR), if there exists some IHR $(U,f)$ satisfying the conditions of Definition \ref{dfn:IHR}. 


The map $\Phi:Y\rightarrow \hf^*$ is $H$-equivariant but $Y$ is not necessarily a $H$-torsor over $\hf^*$. However, if $Y\red{\chi}H$ admits an IHR, there is some dense open set $U\subset \hf^*$ over which $Y$ is a $H$-torsor and moreover is trivial over $U$, with $f$ an explicit trivialization.

The pullback of functions along $H\times U\times Y\red{\chi}H\xrightarrow{f}Y\fibre{\mu,\hf^*}U\rightarrow Y$ gives a strongly $H$-equivariant embedding $\OO(Y)\hookrightarrow \OO(H\times U)\otimes \OO(Y\red{\chi}H)$ of Poisson algebras. Hence, the nomenclature inverse Hamiltonian reduction.

It follows from definition that
\begin{equation}
\big( Y \fibre{\mu,\hf^*} U\big) \red{\chi}H \cong Y\red{\chi}H ~.
\end{equation}
This is the characterizing property of IHR tuple and it lifts to the quantum setting.

\begin{prop}\label{prop:IHR implies quantisation commutes with reduction}
Suppose the reduction $Y\red{\chi}H$ admits some IHR $(U,f)$, and $\AA$ is a strongly $H$-equivariant graded quantization of $\OO_Y$, then
	\begin{enumerate}[(i)]
		\item $\AA|_{\mu^{-1}(U)}\underset{\chi}{/\!\!/\!\!/}H \cong\AA\red{\chi} H$
		\item $\AA\red{\chi} H$ is a quantization of $Y\red{\chi} H$.
	\end{enumerate}
\end{prop}
\begin{proof}
Note that $(i)$ follows immediately from the fact that $\AA$ is a sheaf and $\mu^{-1}(\chi)$ is contained in $\mu^{-1}(U)$.

For $(ii)$, we note that $\mu^{-1}(U)$ admits, by construction, a flat moment map to $\hf^*$ and is a $H$-torsor over it, thus by Lemma 3.3.1 of \cite{Losev2012:quantresolv}, the quantum Hamiltonian reduction of $\AA|_{\mu^{-1}(U)}$ is a quantization of the reduction of $\mu^{-1}(U)$, which is $Y\red{\chi}H$.
\end{proof}

\begin{lem}
Let $U$ be an open, $\IG_\hbar$-stable Poisson subvariety of $\hf^*$ containing $\chi$, and consider $H\times U$, equipped with the Poisson structure described above. Let $\DD$ be any strongly $H$-equivariant quantization of $H\times U$, then 
\begin{equation}
	\DD\red{\chi} H \cong \ik~.
\end{equation}
\end{lem}
\begin{proof}
	Since $H\times U$ is a smooth Poisson scheme with a flat moment map to $\hf^*$ whose fibers are $H$-torsors, Lemma 3.3.1 of \cite{Losev2012:quantresolv} applies; thus, $\DD\red{\chi}H$ is a quantization of $(H\times U)\red{\chi} H\cong {\rm pt}$ and we have the desired result.
\end{proof}

\begin{dfn}
Let $\AA$ be a strongly $H$-equivariant graded quantization of $Y$. A quantum inverse Hamiltonian reduction for $\AA$ is a tuple $(U,f,\DD,\psi)$
\begin{enumerate}[(i)]
	\item $(U,f)$ are IHR data for $(Y,H,\mu,\chi)$ 
	\item $\DD$ is a strongly $H$-equivariant quantization of $H\times U$
	\item $\psi: f^*(\AA|_{\mu^{-1}(U)}) \xrightarrow{\sim} \DD\boxtimes \AA_{KW}$ is a strongly $H$-equivariant isomorphism of quantizations of $(H\times U)\times Y\red{\chi}H$.
\end{enumerate}
\end{dfn}

\begin{rem}
 Since $U$ and $f$ are $\IG_\hbar$-equivariant, $f^*(\AA|_{\Phi^{-1}(U)})$ is a graded quantization of $H\times U \times Y\red{\chi}H$. By transport of structure $\psi$ endows $\DD\boxtimes \AA\red{\chi} H$ with the structure of a graded quantization. Furthermore, this grading induces one on $\AA\red{\chi} H$.
\end{rem}

Taking global sections, we obtain a strongly $H$-equivariant embedding of graded associative $\ik\Ph$-algebras
\begin{equation}
	\Gamma(Y,\AA) \hookrightarrow \Gamma(H\times U,\DD)\tensor{\ik\Ph} \Gamma(Y\red{\chi} H,\AA\red{\chi} H)~.
\end{equation}


\section{Multiplication of slices}

To show that the equivariant Slodowy slices admit inverse Hamiltonian reductions, we shall make use of their incarnation as generalized slices in the affine Grassmannian. In the following section we introduce the necessary machinery to systematically construct inverse Hamiltonian reduction data for generalized slices in type $A$. A number of the results of this section on multiplication maps between generalized slices closely follow those from  \cite{Braverman:2016pwk}, \cite{Krylov:2021conv}.

For the majority of this section we return to using $G$ to denote a reductive algebraic group. We specialise to $G\equiv \PGL_n$ at the end of this section.

\subsection{Partial convolution maps} Recall the convolution map
\[ m: \Gr_G^{(n)} \to \Gr_G \quad\quad\text{defined by} \quad\quad (\mc P_1,...,\mc P_n,\sigma_1,...,\sigma_n) \mapsto (\mc P_n, \sigma_1 \circ \hdots \circ  \sigma_n) \ , \]
as well as the induced convolution maps
\[ \overline{\Gr}_{G}^\ul \to  \overline{\Gr}_G^\lambda \qquad \text{and}\qquad {\varpi}^{\underline{\lambda}}_{\mu}:{\mc W}_\mu^\ul \to \mc W_\mu^\lambda \ , \]
for $\underline{\lambda}=(\lambda_1,\lambda_2,\dots,\lambda_n)$ an $n$-tuple of dominant coweights with sum $\lambda$.

More generally, for any order-preserving surjection $\pi:\{1,...,n\} \onto \{1,...,m\}$ we define the induced partial summation $\ul_\pi$, an $m$-tuple
\[ \ul_\pi = ( \lambda_j )_{j=1}^m \quad\quad \text{where} \quad\quad \lambda_j = \sum_{\{i  | \pi(i)=j \}} \lambda_i = \lambda_{i_{j-1}+1} + \hdots + \lambda_{i_{j}-1} + \lambda_{i_{j}} \]
where $i_j$ is the largest $i\in \{1,...,n\}$ such that $\pi(i)=j$ and $i_0=0$. There exists a natural map
\[ \Gr_{G}^{(n)} \to \Gr_G^{(m)} \qquad\text{defined by}\qquad  (\mc P_1,...,\mc P_n,\sigma_1,...,\sigma_n) \mapsto ((\mc P_{i_j})_{j=1}^m,( \sigma_{i_{j-1}+1} \circ \hdots \circ  \sigma_{i_j} )_{j=1}^m )  \]
and thus induced maps
\[  \overline{\Gr}^\ul_G \to \overline{\Gr}^{\ul_\pi}_G   \qquad\text{and}\qquad  \varpi_\mu^{\ul,\pi} :\Gr_\mu^\ul \to \Gr_\mu^{\ul_\pi} \]
generalizing the usual convolution maps, which correspond to the special case that $m=1$ and $\pi$ is the unique map to a point.

\begin{prop}\label{spoprop}
	For $\pi_0:\{1,...,n\} \onto \{1,...,m\}$ and $\pi_1:\{1,...,m\}\to \{1,...,\ell\}$ order preserving surjections as above, the diagram
	\begin{equation}
		\begin{tikzcd}
			\WW^{\underline{\lambda}}_{\mu} \ar[rr,twoheadrightarrow, "\varpi^{\underline{\lambda},\pi_1\circ \pi_0 }_{\mu}"]\ar[dr,twoheadrightarrow,"\varpi^{\underline{\lambda},\pi_0}_{\mu}"] & & \WW^{\underline{\lambda}_{\pi_1\circ \pi_0}}_\mu \\
			& \WW^{\underline{\lambda}_{\pi_0}}_\mu \ar[ur,swap, twoheadrightarrow,"\varpi^{\underline{\lambda}_{\pi_0},\pi_1}_{\mu}"]
		\end{tikzcd}
	\end{equation}
naturally commutes.
\end{prop}

Finally, we note a tautological property of the above partial convolution maps in the case that $\lambda_i=0$ for some $i$. For $\ul=(\lambda_1,\cdots,\lambda_n)$ an $n$-tuple of dominant coweights $\lambda_i$, an order-preserving surjection $\pi:\{1,...,n\} \onto \{1,...,m\}$ is called a trivial partition of $\ul$ if for each $j\in  \{1,...,m\}$ there is a unique non-zero $\lambda_i$ among $i$, or equivalently, a unique non-zero summand in the definition of each component $\lambda_j$ of $\ul_\pi$.

\begin{prop}\label{sptprop} Let $\pi:\{1,...,n\} \onto \{1,...,m\}$ be a trivial partition of $\ul$. Then the induced map $\varpi_\mu^{\ul,\pi}:\WW_\mu^\ul \to \WW_\mu^{\ul_\pi} $ is an isomorphism.
\end{prop}

These constructions naturally generalize to the Beilinson--Drinfeld generalized slices as follows: Given an order-preserving surjection $\pi:\{1,...,n\} \onto \{1,...,m\}$ and $\ul=(\lambda_1,\cdots,\lambda_n)$ an $n$-tuple of dominant coweights $\lambda_i$, we have:

\begin{prop} There exists a natural convolution map

\[\widetilde{\Gru}_G^{n} \times_{\bb A^n} \bb A^m  \to \widetilde{\Gru}_G^{m} \qquad\text{defined by}\qquad ((x_i)_{i=1}^n, \mc P_1,...,\mc P_n,\sigma_1,...,\sigma_n) \mapsto  \mapsto ((\mc P_{i_j})_{j=1}^m,( \sigma_{i_{j-1}+1} \circ \hdots \circ  \sigma_{i_j} )_{j=1}^m ) \ , \]
where $\Delta(\pi):\bb A^m \to \bb A^n$ is the diagonal embedding determined by $\pi$. 
\end{prop}
\begin{proof} The restriction along the diagonal embedding $\Delta(\pi):\bb A^m\to \bb A^n$ implies that $\sigma_{i_{j-1}+1} \circ \hdots \circ  \sigma_{i_j}$ determines an isomorphism of principal $G$-bundles away from the single point $x_j$ for $j=2,...,m$, and determines a trivialization away from the single point $x_1$ for $j=1$.
\end{proof}

As above, this induces natural maps
	\[\widetilde{\Gru}^\ul \times_{\bb A^n} \bb A^m  \to \widetilde{\Gru}^{\ul_\pi} \qquad\text{and}\qquad  \widetilde{\Wu}_\mu^\ul \times_{\bb A^n} \bb A^m  \to \widetilde{\Wu}_\mu^{\ul_\pi}  \ . \]

Similarly, we have the following generalizations of Propositions \ref{spoprop} and \ref{sptprop} above:

\begin{prop}
	For $\pi_0:\{1,...,n\} \onto \{1,...,m\}$ and $\pi_1:\{1,...,m\}\to \{1,...,\ell\}$ order preserving surjections as above, the diagram
	\begin{equation}
		\begin{tikzcd}
			\widetilde{\Wu}^{\underline{\lambda}}_{\mu}\times_{\bb A^n} \bb A^\ell \ar[rr,twoheadrightarrow, "\varpi^{\underline{\lambda},\pi_1\circ \pi_0 }_{\mu}"]\ar[dr,twoheadrightarrow,"\varpi^{\underline{\lambda},\pi_0}_{\mu}"] & & \widetilde{\Wu}^{\underline{\lambda}_{\pi_1\circ \pi_0}}_\mu \\
			& \widetilde{\Wu}^{\underline{\lambda}_{\pi_0}}_\mu\times_{\bb A^m} \bb A^\ell \ar[ur,swap, twoheadrightarrow,"\varpi^{\underline{\lambda}_{\pi_0},\pi_1}_{\mu}"]
		\end{tikzcd}
	\end{equation}
	naturally commutes.
\end{prop}

\begin{prop} Let $\pi:\{1,...,n\} \onto \{1,...,m\}$ be a trivial partition of $\ul$. Then the induced map $\underline{\varpi}_\mu^{\ul,\pi}:\widetilde{\Wu}^{\underline{\lambda}}_{\mu}\times_{\bb A^n} \bb A^m  \to \widetilde{\Wu}^{\underline{\lambda}_\pi}_{\mu} $ is an isomorphism.
\end{prop}

\subsection{Symmetric convolution slices}

Let $\underline{\lambda} =(\lambda_1,\lambda_2,\dots,\lambda_N)$ be an $N$-tuple of dominant coweights with sum $\lambda$, and let $\mu_+,\mu_-$ be arbitrary coweights.

\begin{dfn}
 The symmetric convolution generalized slice $\widetilde{\WW}_{\mu_-,\mu_+}^{\underline{\lambda}}$ is the stack parameterizing tuples $(\PP_0,\PP_1,\dots \PP_N,\sigma_1,\sigma_2,\ \dots \sigma_N,s,\mc P_{B_-},\phi_-,\mc P_B,\phi_+)$, where 
\begin{itemize}
\item $\mc P_i \in \Bun_G(\bb P^1)$ is a principal $G$-bundle on $\bb P^1$ for $i=0,...,n$;
	\item $\sigma_i: \mc P_{i}|_{\bb A^1_\infty} \xrightarrow{\cong} \mc P_{i-1}|_{\bb A^1_\infty}$ is an isomorphism of principal $G$-bundles with pole at $0$ bounded by $\lambda_i$ for $i=1,...,n$;
	\item $s:\mc P_0|_{\infty}\xrightarrow{\cong} G$ is a trivialization of $\PP_0$ at the point $\infty$;
	\item $\mc P_{B_-} \in \Bun_{B_-}^{-w_0\mu_-}(\bb P^1)$ is a principal $B_-$-bundle on $\bb P^1$ of degree $-w_0\mu_-$; and
	\item $\phi_-: \mc P_{B_-}\times_{B_-} G \xrightarrow{\cong } \mc P_0$ is an isomorphism of principal $G$-bundles, such that the composition
	\[  \mc P_{B_-}|_{\infty}  \to   \mc P_{B_-} |_{\infty} \times_{B_-} G \xrightarrow{\phi_- |_{\infty} } \mc P_0|_{\infty}  \xrightarrow{s} G \]
	maps the fiber $\mc P_{B_-}|_{\infty} $ of $\mc P_{B_-}$ at $\infty \in \bb P^1$ to $B \subset G$;
	\item $\mc P_B \in \Bun_B^{w_0\mu_+}(\bb P^1)$ is a principal $B$-bundle on $\bb P^1$ of degree $w_0\mu_+$; and
	\item $\phi_+: \mc P_B\times_B G \xrightarrow{\cong } \mc P_n$ is an isomorphism of principal $G$-bundles, such that the composition
	\[  \mc P_B|_{\infty}  \to   \mc P_B |_{\infty} \times_B G \xrightarrow{\phi_+|_{\infty}} \mc P_n|_{\infty}  \xrightarrow{s\circ \sigma_1|_{\infty} \circ \hdots \circ \sigma_n|_{\infty} } G \]
	maps the fiber $\mc P_B|_{\infty} $ of $\mc P_B$ at $\infty \in \bb P^1$ to $B_- \subset G$.
\end{itemize}
\end{dfn}

In the case $N=1$, so that $\ul$ is given by a single dominant coweight $\lambda$, the resulting space $\widetilde{\WW}^{\lambda}_{\mu_-,\mu_+}$ is called simply the symmetric generalized slice.

We also have natural convolution maps
\begin{equation}
	\varpi^{\underline{\lambda}}_{\mu_-,\mu_+}: \widetilde{\WW}^{\underline{\lambda}}_{\mu_-,\mu_+}\twoheadrightarrow \WW^{\lambda}_{\mu_-,\mu_+} 
\end{equation}
as well as their partial variants, $\varpi^{\underline{\lambda},\pi}_{\mu_-,\mu_+}: \widetilde{\WW}^{\underline{\lambda}}_{\mu_-,\mu_+}\twoheadrightarrow \widetilde{\WW}^{\underline{\lambda}_\pi}_{\mu_-,\mu_+}$, for $\underline{\lambda}_\pi=(\lambda_j)_{j=1}^m$ the partial summation of $\underline{\lambda}$ determined by an order-preserving surjection $\pi:\{1,...,n\}\to\{1,...,m\}$.

Following \cite{Braverman:2016pwk} and \cite{Krylov:2021conv}, we have the following isomorphisms
\begin{prop}\label{symtrividprop} There is a natural identification
\[ \widetilde{\WW}^{\underline{\lambda}}_{0,\mu_+}\cong \widetilde{\WW}^{\underline{\lambda}}_{\mu_+}  \ .\]
\end{prop}
\begin{proof}
Note that $\mu_-=0$ if and only if $\mc P_{B_-}$ is trivializable, which implies $\mc P_0$ is trivializable and moreover the choice of trivialization $s$ of the fiber $\mc P_0|_{\infty}$ extends to a unique global trivialization $\sigma_0$ of $\mc P_0$. Under this identification, $\sigma_1:\mc P^1|_{\bb A^1_\infty}\xrightarrow{\cong} \mc P^0|_{\bb A^1_\infty}$ is equivalent to a trivialization $\tilde \sigma_1 = \sigma_0\circ \sigma_1:\mc P_1|_{\bb A^1_\infty}\xrightarrow{\cong} \mc P^\triv$. The conditions on $\phi_+$ for $\widetilde{\WW}^{\underline{\lambda}}_{0,\mu_+}$ and $\widetilde{\WW}^{\underline{\lambda}}_{\mu_+} $ correspond under this identification since $s\circ \sigma_1|_{\infty}  \circ ... \circ \sigma_n|_{\infty} = \sigma_0 |_{\infty}  \circ \sigma_1|_{\infty}  \circ ... \circ \sigma_n|_{\infty}  = \tilde \sigma_1 |_{\infty} \circ ... \circ \sigma_n|_{\infty} $.

\end{proof}

\begin{prop}\label{Heckeisoprop}
There is a natural isomorphism
$$
	{\rm Hec}^{\underline{\lambda}}_{\mu_-,\mu_+}:\widetilde{\WW}^{\underline{\lambda}}_{\mu_-,\mu_+} \xrightarrow{\cong} \widetilde{\WW}^{\underline{\lambda}}_{0,\mu_-+\mu_+} \cong \widetilde{\WW}^{\underline{\lambda}}_{\mu_+} 
$$
defined by modifying the underlying $G$-bundles by a Hecke transformation at $\infty$ of weight $-w_0\mu_-$.
\end{prop}

\subsection{Multiplication maps}

Given lists of dominant coweights $\ul^1$ and $\ul^2$ of lengths $N_1$ and $N_2$, and $\mu_1$ and $\mu_2$ arbitrary coweights, with corresponding convolution generalized slices $\widetilde{\WW}^{\underline{\lambda}^1}_{\mu_1}$ and $\widetilde{\WW}^{\underline{\lambda}^2}_{\mu_2}$, we define the multiplication map
\begin{equation}
	\widetilde{\bf m}: \widetilde{\WW}^{\ul^1}_{\mu_1}\times \widetilde{\WW}^{\ul^2}_{\mu_2} \rightarrow \widetilde{\WW}^{(\ul^1,\ul^2)}_{\mu_1+\mu_2}~,
\end{equation}
for $(\ul^1,\ul^2)$ the concatenated list of length $n_1+n_2$, as the composition
\begin{equation}
	\widetilde{\WW}^{\ul^{1}}_{\mu_1}\times \widetilde{\WW}^{\ul^{2}}_{\mu_2} \xrightarrow{\rm Hec\times Hec} \widetilde{\WW}^{\ul^{1}}_{\mu_1,0}\times \widetilde{\WW}^{\ul^{2}}_{0,\mu_2} \xrightarrow{\bf c} \widetilde{\WW}^{(\ul^{1},\ul^{2})}_{\mu_1,\mu_2}\xrightarrow{\rm Hec} \widetilde{\WW}^{(\ul^{1},\ul^{2})}_{\mu_1+\mu_2}
\end{equation}
where $\mathbf{c}:\widetilde{\WW}^{\ul^{1}}_{\mu_1,0}\times \widetilde{\WW}^{\ul^{2}}_{0,\mu_2} \to \widetilde{\WW}^{(\ul^{1},\ul^{2})}_{\mu_1,\mu_2}$ denotes the concatenation map, defined by mapping
\begin{equation}
	\left((\PP_0,\PP_1,\dots ,\PP_{N_1}=\PP_{\rm triv}, \sigma_1,\dots,\tilde \sigma_{N_1},\mc P_{B_-},\phi_-) , (\PP_0'= \PP_{\rm triv},\PP_1',\dots ,\PP_{N_2}', \tilde \sigma_1',\dots,\sigma_{N_2}',\mc P_{B_+},\phi_+)\right)\ ,
\end{equation}
to the tuple 
\begin{equation}
	(\PP_0,\PP_1\dots, \mc P_{N_1},\PP'_1,\dots,\PP'_{N_2},\sigma_1,\dots,\sigma_{N_1},\sigma_1',\sigma_2',\dots, \sigma_{N_2}',s,\mc P_{B_-},\phi_-,\mc P_{B_+},\phi_+)
\end{equation}
where we have used the identifications and notation from the proof of Proposition \ref{symtrividprop} for the points of $\widetilde{\WW}^{\ul^{1}}_{\mu_1,0}$ and $\widetilde{\WW}^{\ul^{2}}_{0,\mu_2}$, $\sigma_1':\mc P_1'|_{\bb A^1_\infty}\xrightarrow{\cong} \mc P_{N_1}|_{\bb A^1_\infty}$ is defined as the composition of $\tilde \sigma_1': \mc P_1'|_{\bb A^1_\infty}\to \mc P^\triv$ with the identification $\mc P_{N_1}=\mc P^\triv$, and $s$ is determined by the trivialization of $\mc P_{N_1}$ under the composition of the isomorphisms $\sigma_i|_{\infty}$.

\begin{prop}[\cite{Krylov:2021conv}]
The map $\widetilde{\bf m}$ is an open immersion.
\end{prop}
\begin{proof}
The concatenation map $\mathbf{c}$ is evidently injective and its image is equal to the subset where $\PP_{N_1}$ is trivial. This is an open condition and thus the map $\mathbf{m}$ is open.
\end{proof}

Further, we define the multiplication map
\begin{equation}
\mathbf{m}: \WW^{\lambda_1}_{\mu_1}\times \WW^{\lambda_2}_{\mu_2} \rightarrow \WW^{\lambda_1+\lambda_2}_{\mu_1+\mu_2}~,	
\end{equation}
as the composition $\WW^{\lambda_1}_{\mu_1}\times \WW^{\lambda_2}_{\mu_2} \xrightarrow{\widetilde{\bf m}} \widetilde{\WW}^{(\lambda_1,\lambda_2)}_{\mu_1+\mu_2} \xrightarrow{\varpi} \WW^{\lambda_1+\lambda_2}_{\mu_1+\mu_2}$. Note that unlike $\widetilde{\bf m}$, $\mathbf{m}$ fails to be an open immersion in general, though it is always birational \cite{Braverman:2016pwk}.

\begin{rem}\label{rem:antidominant_multiplication_is_open}
Since $\varpi: \widetilde{\WW}^{(\lambda_1,0)}_{\mu_1+\mu_2}\xrightarrow{\sim} \WW^{\lambda_1}_{\mu_1+\mu_2}$ is an isomorphism, $\mathbf{m}:\WW^{\lambda_1}_{\mu_1}\times \WW^0_{\mu_2}\rightarrow \WW^{\lambda_1}_{\mu_1+\mu_2}$ is an open immersion.
\end{rem}

\begin{rem}
Let $\lambda_2$ be a dominant coweight and consider the multiplication map for the point $\WW^{\lambda_2}_{\lambda_2}=t^{\mu_2}$, $\WW^{\lambda_1}_{\mu_1}\times \WW^{\lambda_2}_{\lambda_2} \rightarrow \WW^{\lambda_1+\lambda_2}_{\mu_1+\mu_2}$. This gives rise to a shift map,
\begin{equation}
	\iota_{\mu_2}:\WW^{\lambda_1}_{\mu_1}\hookrightarrow \WW^{\lambda_1+\lambda_2}_{\mu_1+\mu_2}~.
\end{equation}
which is a Poisson, open immersion.
\end{rem}

More generally, for $\lambda_1$ a dominant coweight and $\ul^2=(\lambda_{1}^2,...,\lambda_{n}^2)$ an $n$-tuple of such, we let $(\lambda_1+\ul^2)$ denote the $n$-tuple $(\lambda_1+\lambda_{1}^2,\lambda^2_2,...,\lambda_{n}^2)$, and define the multiplication map
\begin{equation}
\mathbf{m}: \WW^{\lambda_1}_{\mu_1}\times \WW^{\ul^2}_{\mu_2} \rightarrow \WW^{(\lambda_1+\ul^2)}_{\mu_1+\mu_2}~,
\end{equation}
as the composition $\WW^{\lambda_1}_{\mu_1}\times \WW^{\ul^2}_{\mu_2} \xrightarrow{\widetilde{\bf m}} \widetilde{\WW}^{(\lambda_1,\ul^2)}_{\mu_1+\mu_2} \xrightarrow{\varpi^\pi} \WW^{\lambda_1+\lambda_2}_{\mu_1+\mu_2}$, where $\varpi^\pi$ denotes the partial convolution map determined by the surjection $\pi:\{1,...,n\}\to \{1,...,n-1\}$ defined by $\pi(1)=1$ and $p(i)=i-1$ for $i\geq 2$.

\begin{lem}[\protect{\cite[Lemma 5.10]{Krylov:2021conv}}]\label{lem:multiplication is hbar equiv}
The multiplication map $\widetilde{\mathbf{m}}: \WW^{\lambda^1}_{\mu_1}\times\WW^{\lambda_2}_{\mu_2}\hookrightarrow \WW^{(\lambda_1,\lambda_2)}_{\mu_1+\mu_2}$ is $T\times\IG_\hbar$-equivariant for the natural torus action and for the natural loop-rotation action on $\WW^{\lambda_1}_{\mu_1}$ and $\WW^{(\lambda_1,\lambda_2)}_{\mu_1+\mu_2}$ and the action on $\WW^{\lambda_2}_{\mu_2}$ is twisted by the torus action via the cocharacter $\mu_1$.
\end{lem}

\subsection{Multiplication maps on Beilinson--Drinfeld generalized slices}

\begin{dfn} The Beilinson--Drinfeld symmetric convolution generalized slice $\underline{\widetilde{\WW}}_{\mu_-,\mu_+}^{\underline{\lambda}}$ is the stack parameterizing tuples $((x_i)_{i=1}^n,\PP_0,\PP_1,\dots \PP_N,\sigma_1,\sigma_2,\ \dots \sigma_N,s,\mc P_{B_-},\phi_-,\mc P_B,\phi_+)$, where 
	\begin{itemize}
	\item $(x_i)_{i=1}^n \in \bb A^n$ is a closed point of $\bb A^n$;
	\item $\mc P_i \in \Bun_G(\bb P^1)$ is a principal $G$-bundle on $\bb P^1$ for $i=0,...,n$;
	\item $\sigma_i: \mc P_{i}|_{\bb P^1 \setminus\{x_i\}} \xrightarrow{\cong} \mc P_{i-1}|_{\bb P^1\setminus\{x_i\}}$ is an isomorphism of principal $G$-bundles with pole at $x_i$ bounded by $\lambda_i$ for $i=1,...,n$;
	\item $s:\mc P_0|_{\infty}\xrightarrow{\cong} G$ is a trivialization of $\PP_0$ at the point $\infty$;
	\item $\mc P_{B_-} \in \Bun_{B_-}^{-w_0\mu_-}(\bb P^1)$ is a principal $B_-$-bundle on $\bb P^1$ of degree $-w_0\mu_-$;
	\item $\phi_-: \mc P_{B_-}\times_{B_-} G \xrightarrow{\cong } \mc P_0$ is an isomorphism of principal $G$-bundles, such that the composition
	\[  \mc P_{B_-}|_{\infty}  \to   \mc P_{B_-} |_{\infty} \times_{B_-} G \xrightarrow{\phi_- |_{\infty} } \mc P_0|_{\infty}  \xrightarrow{s} G \]
	maps the fiber $\mc P_{B_-}|_{\infty} $ of $\mc P_{B_-}$ at $\infty \in \bb P^1$ to $B \subset G$;
	\item $\mc P_B \in \Bun_B^{w_0\mu_+}(\bb P^1)$ is a principal $B$-bundle on $\bb P^1$ of degree $w_0\mu_+$; and
	\item $\phi_+: \mc P_B\times_B G \xrightarrow{\cong } \mc P_n$ is an isomorphism of principal $G$-bundles, such that the composition
	\[  \mc P_B|_{\infty}  \to   \mc P_B |_{\infty} \times_B G \xrightarrow{\phi_+|_{\infty}} \mc P_n|_{\infty}  \xrightarrow{s\circ \sigma_1|_{\infty} \circ \hdots \circ \sigma_n|_{\infty} } G \]
	maps the fiber $\mc P_B|_{\infty} $ of $\mc P_B$ at $\infty \in \bb P^1$ to $B_- \subset G$.
\end{itemize}
\end{dfn}

Analogously to Propositions \ref{symtrividprop} and \ref{Heckeisoprop}, we have:

\begin{prop}\label{symtrividfactprop} There is a natural identification
	\[ \widetilde{\Wu}^{\underline{\lambda}}_{0,\mu_+}\cong \widetilde{\Wu}^{\underline{\lambda}}_{\mu_+}  \ .\]
\end{prop}
\begin{prop}\label{Heckeisofactprop}
	There is a natural isomorphism
	$$
	\underline{\rm Hec}^{\underline{\lambda}}_{\mu_-,\mu_+}:\widetilde{\Wu}^{\underline{\lambda}}_{\mu_-,\mu_+} \xrightarrow{\cong} \widetilde{\Wu}^{\underline{\lambda}}_{0,\mu_-+\mu_+} \cong \widetilde{\Wu}^{\underline{\lambda}}_{\mu_+} 
	$$
\end{prop}

As above, given lists of dominant coweights $\ul^1$ and $\ul^2$ of lengths $N_1$ and $N_2$, and $\mu_1$ and $\mu_2$ arbitrary coweights, with corresponding Beilinson--Drinfeld convolution generalized slices $\widetilde{\WW}^{\underline{\lambda}^1}_{\mu_1}$ and $\widetilde{\WW}^{\underline{\lambda}^2}_{\mu_2}$, we define the multiplication map
\begin{equation}
	\underline{\widetilde{\bf m}}: \widetilde{\Wu}^{\ul^1}_{\mu_1}\times \widetilde{\Wu}^{\ul^2}_{\mu_2} \rightarrow \widetilde{\Wu}^{(\ul^1,\ul^2)}_{\mu_1+\mu_2}~,
\end{equation}
for $(\ul^1,\ul^2)$ the concatenated list of length $n_1+n_2$, as the composition
\begin{equation}
	\widetilde{\Wu}^{\ul^{1}}_{\mu_1}\times \widetilde{\Wu}^{\ul^{2}}_{\mu_2} \xrightarrow{\rm Hec\times Hec} \widetilde{\Wu}^{\ul^{1}}_{\mu_1,0}\times \widetilde{\Wu}^{\ul^{2}}_{0,\mu_2} \xrightarrow{\bf c} \widetilde{\Wu}^{(\ul^{1},\ul^{2})}_{\mu_1,\mu_2}\xrightarrow{\rm Hec} \widetilde{\Wu}^{(\ul^{1},\ul^{2})}_{\mu_1+\mu_2}
\end{equation}
where $\mathbf{c}:\widetilde{\Wu}^{\ul^{1}}_{\mu_1,0}\times \widetilde{\Wu}^{\ul^{2}}_{0,\mu_2} \to \widetilde{\Wu}^{(\ul^{1},\ul^{2})}_{\mu_1,\mu_2}$ denotes the concatenation map, defined by mapping
\begin{equation}
	\left(((x_i)_{i=1}^{N_1},\PP_0,\PP_1,\dots ,\PP_{N_1}=\PP_{\rm triv}, \sigma_1,\dots,\tilde \sigma_{N_1},\mc P_{B_-},\phi_-) , ((x_i')_{i=1}^{N_2},\PP_0'= \PP_{\rm triv},\PP_1',\dots ,\PP_{N_2}', \tilde \sigma_1',\dots,\sigma_{N_2}',\mc P_{B_+},\phi_+)\right)\ ,
\end{equation}
to the tuple 
\begin{equation}
	(((x_i)_{i=1}^{N_1},(x_i')_{i=1}^{N_2}),\PP_0,\PP_1\dots, \mc P_{N_1},\PP'_1,\dots,\PP'_{N_2},\sigma_1,\dots,\sigma_{N_1},\sigma_1',\sigma_2',\dots, \sigma_{N_2}',s,\mc P_{B_-},\phi_-,\mc P_{B_+},\phi_+)
\end{equation}
where $\sigma_1':\mc P_1'|_{\bb P^1\setminus \{ x_1'\}}\xrightarrow{\cong} \mc P_{N_1}|_{\bb P^1\setminus \{ x_1'\}}$ and $s$ are defined as above.

\begin{prop}
	The map $\underline{\widetilde{\bf m}}$ is an open immersion.
\end{prop}

Further, we define the multiplication map
\begin{equation}
	\underline{\mathbf{m}}: \Wu^{\lambda_1}_{\mu_1}\times_{\bb A^1} \Wu^{\lambda_2}_{\mu_2} \rightarrow \Wu^{\lambda_1+\lambda_2}_{\mu_1+\mu_2}~,	
\end{equation}
as the composition
 \[\Wu^{\lambda_1}_{\mu_1}\times_{\bb A^1} \Wu^{\lambda_2}_{\mu_2} \cong  (\Wu^{\lambda_1}_{\mu_1}\times \Wu^{\lambda_2}_{\mu_2})\times_{\bb A^2} \bb A^1 \xrightarrow{\widetilde{\bf m}|_{\bb A^1}} \widetilde{\Wu}^{(\lambda_1,\lambda_2)}_{\mu_1+\mu_2}\times_{\bb A^2} \bb A^1  \xrightarrow{\varpi} \Wu^{\lambda_1+\lambda_2}_{\mu_1+\mu_2} \ .
\]
As above, $\underline{\mathbf{m}}$ fails to be an open immersion in general, though it is again birational \cite{Braverman:2016pwk}.

\begin{rem}\label{rem:BD_antidominant_multiplication_is_open}
	Since $\varpi: \widetilde{\Wu}^{(\lambda_1,0)}_{\mu_1+\mu_2}\times_{\bb A^2} \bb A^1 \xrightarrow{\sim} \Wu^{\lambda_1}_{\mu_1+\mu_2}$ is an isomorphism, $\mathbf{m}:\Wu^{\lambda_1}_{\mu_1}\times_{\bb A^1} \Wu^0_{\mu_2}\rightarrow \Wu^{\lambda_1}_{\mu_1+\mu_2}$ is an open immersion.
\end{rem}

\begin{rem}
	Let $\lambda_2$ be a dominant coweight and consider the multiplication map for the point $\Wu^{\lambda_2}_{\lambda_2}=\{t^{\mu_2}\} \times \bb A^1$, $\Wu^{\lambda_1}_{\mu_1}\times_{\bb A^1} \Wu^{\lambda_2}_{\lambda_2} \rightarrow \Wu^{\lambda_1+\lambda_2}_{\mu_1+\mu_2}$. This gives rise to a shift map,
	\begin{equation}
		\iota_{\mu_2}:\Wu^{\lambda_1}_{\mu_1}\hookrightarrow \Wu^{\lambda_1+\lambda_2}_{\mu_1+\mu_2}~.
	\end{equation}
	which is a Poisson, open immersion.
\end{rem}

More generally, for $\lambda_1$ a dominant coweight and $\ul^2=(\lambda_{1}^2,...,\lambda_{n}^2)$ an $n$-tuple of such, and $(\lambda_1+\ul^2)$ the $n$-tuple $(\lambda_1+\lambda_{1}^2,\lambda^2_2,...,\lambda_{n}^2)$, we define the multiplication map
\begin{equation}
\underline{	\mathbf{m}}: \Wu^{\lambda_1}_{\mu_1}\times_{\bb A^1}  \Wu^{\ul^2}_{\mu_2} \rightarrow \Wu^{(\lambda_1+\ul^2)}_{\mu_1+\mu_2}~,
\end{equation}
as the composition
\[\Wu^{\lambda_1}_{\mu_1}\times_{\bb A^1} \Wu^{\ul^2}_{\mu_2} \xrightarrow{\widetilde{\bf m}} \widetilde{\Wu}^{(\lambda_1,\ul^2)}_{\mu_1+\mu_2}  \times_{\bb A^2} \bb A^1 \cong \widetilde{\Wu}^{(\lambda_1,\ul^2)}_{\mu_1+\mu_2}  \times_{\bb A^{n+1}} \bb A^n \xrightarrow{\varpi^\pi} \Wu^{(\lambda_1+\ul^2)}_{\mu_1+\mu_2} \ , \]
where $\varpi^\pi$ denotes the partial convolution map determined by the surjection $\pi:\{1,...,n\}\to \{1,...,n-1\}$ defined by $\pi(1)=1$ and $p(i)=i-1$ for $i\geq 2$.

\subsection{Multiplication maps on matrix realisations}

Once again, we restrict to the case of generalised slices in $\PGL_n$. The multiplication map on matrix representatives $x_1\in \Gr_{\mu_1}^{\lambda_1}$ and $x_2\in \Gr_{\mu_2}^{\lambda_2}$ takes a particularly nice form. We follow the exposition of \cite{Kamnitzer:2022ham} Section $2$.

For any coweight $\mu$ of $\rm PGL_n$, we write
\begin{equation}
	X_\mu  = N(\!(t^{-1})\!) T_1(\!(t^{-1})\!) t^\mu N_{-}(\!(t^{-1})\!)~.
\end{equation}
We recall the subscheme, $\Gr_\mu = N_1[\![t^{-1}]\!] T_1[\![t^{-1}]\!] t^\mu N_-\Pt$ of $\PGL_n[\![t^{-1}]\!]$, involved in the matrix realisation of generalised slices.
\begin{rem}
We have a natural projection map 
\begin{equation}\label{eq:projection to Grmu}
	\pi: X_\mu \rightarrow \Gr_\mu~,
\end{equation}
which on any $x\in X_\mu$ is of the form $\pi(x) = n xn_-$ for unique $n\in N[t]$ and $n_-\in N_-[t]$.
\end{rem}

Given two matrix representatives $x_1 \in \Gr_{\mu_1}$ and $x_2\in \Gr_{\mu_2}$, we note that $x_1\cdot x_2\in X_{\mu_1+\mu_2}$, where the dot denotes matrix multiplication. Thus, by composing with \eqref{eq:projection to Grmu} we have a well-defined map $ \Gr_{\mu_1}\times \Gr_{\mu_2}\rightarrow \Gr_{\mu_1+\mu_2}$. This is precisely the multiplication map, \ie, given matrix representatives $x_1\in \Gr^{\lambda_1}_{\mu_1}$ and $x_2\in \Gr^{\lambda_2}_{\mu_2}$, 
\begin{equation}
	{\bf m}(x_1,x_2) = \pi(x_1\cdot x_2)~.
\end{equation}

\section{Inverse Hamiltonian reduction for generalized slices}\label{IHRslicesec}

We largely follow the strategy of \cite{Kamnitzer:2022ham} to prove our version of inverse Hamiltonian reduction. 


Throughout this section we fix some positive integer $n$ and use $\Gr$ to denote the affine Grassmannian of ${\rm PGL}_n$. We pick $\alpha\in\Delta_+$ to be some positive root in $\rm PGL_n$. Any such positive coroot is of the form $\alpha\equiv \alpha_{(r_1,r_2)} = \sum_{i=r_1}^{r_2} \alpha_i$, where $r_1\le r_2\le n$ are a pair of positive integers with $r_2-r_1 = k-1$ for some (non-zero) integer $k$. A choice of such an $\alpha$ gives rise to a group homomorphism
\begin{equation}
	\tau_{\alpha}:{\rm GL}_{k+1} \hookrightarrow {\rm PGL}_n~,
\end{equation}
defined by embedding ${\rm GL}_{k+1}$ into the $(k+1)\times (k+1)$ block generated by the roots $(\alpha_{r_1},\alpha_{r_1+1},\dots, \alpha_{r_2})$ and their negatives. The positive coroot $\alpha$ is the highest root in ${\rm GL}_{k+1}$.

Now, denote by $x_{\alpha}$ and $x_{-\alpha}$ the restrictions of $\tau_\alpha$ to the upper and lower unipotent groups of ${\rm GL}_{k+1}$, respectively. Define $N^\alpha\subset N\subset {\rm PGL}_n$ as the cokernel of $x_\alpha$, \ie,
\begin{equation}
	N^\alpha = N/{\rm im}\,x_\alpha
\end{equation}
By functoriality, $x_{\pm \alpha}$ and $\tau_\alpha$ lift to maps on the loop spaces.

Crucial to the proofs is the introduction of a subgroup $N(\!(t^{-1})\!)^\alpha_1\subset N(\!(t^{-1})\!)$, generalizing the first congruent subgroup $N_1[\![t^{-1}]\!]\subset N[\![t^{-1}]\!]$. We define, 
\begin{equation}
	N(\!(t^{-1})\!)^\alpha_1 = x_\alpha(N_{1,k}[\![t^{-1}]\!])\cdot N^\alpha(\!(t^{-1})\!)~,
\end{equation}
where $N_{1,k}[\![t^{-1}]\!]\subset N_k[\![t^{-1}]\!]$ is the first congruent subgroup of the upper unipotent subgroup of ${\rm GL}_k[\![t^{-1}]\!]$~.

\begin{prop}\label{prop:conjugating Nalpha}
Here are a number of basic results on the stability of $N^\alpha$ under conjugation.
\begin{enumerate}[(i)]
	\item Let $u\in N^\alpha$ and $g\in {\rm GL}_{k+1}$, then 
	$$ \tau_\alpha(g)u\tau_\alpha(g)^{-1} \in N^\alpha~.$$
	\item Let $u\in N_1^\alpha[\![t^{-1}]\!]$, where $N^\alpha_1\Pti$ is the first congruent subgroup of $N^\alpha[\![t^{-1}]\!]$, and $g\in {\rm GL}_{k+1}[\![t^{-1}]\!]$ then 
	$$ \tau_\alpha(g)u\tau_\alpha(g)^{-1} \in N^\alpha_1[\![t^{-1}]\!]~.$$
	\item Let $u \in N^\alpha[t], g\in {\rm GL}_{k+1}[t]$, then 
	$$ \tau_\alpha(g)u\tau_\alpha(g)^{-1} \in N^\alpha[t]~. $$
\end{enumerate}
\end{prop}
\begin{proof}
Note that $(i)$ implies $(ii)$ and $(iii)$. To prove $(i)$, it suffices to look at the block decomposition of some element of $N^\alpha$.

Let $V$ be the defining, $n$ dimensional representation of ${\rm GL}_n$. Our choice of $\alpha$ gives rise to a splitting $V = V_1\oplus V_2$ where $V_2$ is a $k+1$ dimensional subspace stabilized by the image of $\tau_\alpha$.  



A matrix $u\in N^\alpha$ is, schematically, of the form
\begin{equation}
	u = 
	\begin{pmatrix}
	u_{11} & u_{12} & u_{13}\\
	0 & \id_{V_{2}} & u_{23}\\
	0 & 0 & u_{33}
	\end{pmatrix}~,
\end{equation}
where $u_{ij}\in {\rm Hom}(V_j,V_i)$ with $u_{ii}$ unipotent. Now if  $g\in {\rm GL}_{k+1}$, then conjugation by $\tau_{\alpha}(g)$ is 
\begin{equation}
\begin{split}
	\tau_\alpha(g)u\tau_\alpha(g)^{-1} &= 
	\begin{pmatrix}
		\id_{V_1} & 0 & 0 \\
		0 & g & 0\\
		0 & 0 & \id_{V_3} 
	\end{pmatrix}
	\cdot 
	\begin{pmatrix}
	u_{11} & u_{12} & u_{13}\\
	0 & \id_{V_{2}} & u_{23}\\
	0 & 0 & u_{33}
	\end{pmatrix}
	\cdot
	\begin{pmatrix}
		\id_{V_1} & 0 & 0 \\
		0 & g^{-1} & 0\\
		0 & 0 & \id_{V_3} 
	\end{pmatrix}
	\\
	&= 
	\begin{pmatrix}
	u_{11} &  u_{12}\circ g^{-1} & u_{13}\\
	0 & \id_{V_2} & g\circ u_{23} \\
	0 & 0 & u_{33}
	\end{pmatrix}
\end{split}
\end{equation}
which is in $N^\alpha$, as desired.
\end{proof}

\subsection{An action of the translation group}\label{ssec:action of translation group}

Let $\IG_a$ be the additive group, we write $\IG_\alpha \equiv \IG_a^k$ for the group of translations on $\IA^k$. 

Let $\underline{v}\in \IG_\alpha$, thought of as being embedded in ${\rm GL}_{k+1}$ via
\begin{equation}
	\underline{v} = (v_1,v_2,\dots,v_k) \mapsto 
	\begin{pmatrix}
	 1 & 0 & \dots & 0 & 0\\
	 0 & 1 & \dots & 0 & 0\\
	 \vdots & \vdots & \vdots & \vdots & \vdots\\
	 v_k & v_{k-1} & \dots & v_1 & 1
	\end{pmatrix}
\end{equation}
To that end, we abuse notation to write $x_{-\alpha}(\underline{v})$ for the composition
\begin{equation}
	\underline{v}\mapsto x_{-\alpha} (
	\begin{pmatrix}
	 1 & 0 & \dots & 0 & 0\\
	 0 & 1 & \dots & 0 & 0\\
	 \vdots & \vdots & \vdots & \vdots & \vdots\\
	 v_k & v_{k-1} & \dots & v_1 & 1
	\end{pmatrix}
	)
	~.
\end{equation}

\begin{prop}[\protect{\cite[Remark 2.8]{Krylov:2021conv}}]\label{prop:action of translation}
For any coweight $\mu$ of $\rm PGL_n$, the transversal slice, $\Gr_\mu$ has an action of $\IG_\alpha$ given by 
\begin{equation}
	\underline{v}\cdot g = \pi(x_{-\alpha}(\underline{v})g)~,
\end{equation}
for $\underline{v}\in \IG_\alpha$, $g\in \Gr_\mu$ and where $\pi: N(\!(t^{-1})\!) t^\mu T_1[\![t^{-1}]\!] N_-(\!(t^{-1})\!)\rightarrow \Gr_\mu$ is the map from \ref{eq:projection to Grmu}.
\end{prop}

Indeed, this action will later be shown to be Hamiltonian, for now we construct a candidate for its moment map.

Let $\Gf_\alpha\coloneqq {\rm Lie}\, \IG_\alpha$ thought of as a commutative, Lie subalgebra of $\nf$ generated by the roots $(\alpha_{r_1},\alpha_{r_1}+\alpha_{r_1+1},\dots,\sum_{i=r_1}^{r_2} \alpha_i)$. Define  $\psi_{\alpha}: N\rightarrow \mathfrak{G}_\alpha^*$  by 
\begin{equation}
	a \mapsto (a_{r_2,r_2+1},a_{r_2-1,r_2+1}, \dots, a_{r_1,r_2+1})~,
\end{equation}
\ie, we extract the matrix entries in the rightmost column of the image of $\tau_{\alpha}$. 

This lifts, naturally to a map on the loop spaces $\psi_{\alpha}: N(\!(t^{-1)})\!)\rightarrow \Gf_\alpha^*(\!(t^{-1})\!)$. Moreover, we can extend $\psi_{\alpha}$ to a map $\psi_\alpha:\Gr_\mu\rightarrow \Gf_\alpha^*\Lt$ by
\begin{equation}
	\psi_\alpha(uht^\mu u_-) = \psi_\alpha(u)~,
\end{equation}
where $u\in N_1[\![t^{-1}]\!]$, $h\in T_1[\![t^{-1}]\!]$ and  $u_-\in N_{-,1}[\![t^{-1}]\!]$ parameterize elements of the slice by their triangular decomposition.

We shall be interested in the Fourier modes of $\psi_\alpha$ and its components, so we write 
\begin{equation}
	\psi_\alpha(u) = \bigg( \sum_j\psi_{\alpha,1}^{(j)}(u)t^{-j}, \sum_j\psi_{\alpha,2}^{(j)}(u)t^{-j},\dots, \sum_j\psi_{\alpha,k}^{(j)}(u)t^{-j}\bigg)~. 
\end{equation}
In particular, we use $\Phi_\alpha: N(\!(t^{-1})\!)\rightarrow \Gf^*_\alpha$ for the restriction to the first Fourier modes, \ie,
\begin{equation}
	\Phi_\alpha(u) = \bigg(\psi_{\alpha,1}^{(1)}(u),\psi_{\alpha_2}^{(1)}(u),\dots,\psi_{\alpha,k}^{(1)}(u)\bigg)
\end{equation}

\begin{lem}\label{lem:moment map smooth}
	The map $\Phi_\alpha$ is smooth.
\end{lem}

Similarly, we introduce $\zeta_\alpha: N(\!(t^{-1})\!) \rightarrow \ik^k$ defined on an element $n\in N(\!(t^{-1})\!)$ by
\begin{equation}
	\zeta_\alpha(n) = (  n_{r_1,r_1+1}^{(1)}, n_{r_1,r_1+2}^{(1)},\dots, n_{r_1,r_2-1}^{(1)}, n_{r_1,r_2}^{(2)})~,
\end{equation}
where $n_{ij}^{(l)}$ is the $l$th Fourier mode of the entry at position $(i,j)$ Like $\psi_\alpha$, these maps extend naturally to the slice $\Gr_\mu$ by
\begin{equation}
	\Phi_\alpha(uht^\mu u_-)\equiv \Phi_\alpha(u), \text{ and } \zeta_\alpha(uht^\mu u_-) \equiv \zeta_\alpha(u)~.
\end{equation}

\subsection{Action on Zastava}

Now we focus on the action of $\IG_\alpha$ on the slice $\Gr_{-\alpha}^0$. Recall the co-ordinates $(p,e,b_i,g_i)$ from \eqref{eq:coords on zastava}.

\begin{lem}\label{lem:Ga action on antidominant is translation}
The action of $\IG_\alpha$ on $\Gr_{-\alpha}^0$ defined by Proposition \ref{prop:action of translation} simplifies to 
\begin{equation}
	\underline{v}\cdot g = \pi(x_{-\alpha}(\underline{v}) g) = x_{-\alpha}(\underline{v}) g~.
\end{equation}
and corresponds to translating the co-ordinates as 
\begin{equation}
	g_i \mapsto g_i + v_i~, \qquad p \mapsto p+ v_ke~, 
\end{equation}
\end{lem}


\begin{lem}\label{lem:antidominant UTU decomp}
Let $y\in \Gr_{-\alpha}^0$ be a matrix representative, then $\Phi_\alpha(y)$ and $\zeta_{\alpha}(y)$ are given by
\begin{equation}
\begin{split}
	\Phi_{\alpha}(y) &= \bigg(b_1(y),b_2(y),\dots,b_{k-1}(y),e(y)\bigg)~,\\
	\zeta_{\alpha}(y) &= \bigg(g_{k-1}(y)e(y), g_{k-2}(y)e(y), \dots, g_1(y) e(y), p(y)e(y)+ e(y)\sum_{i=1}^{k-1}b_i(y)g_i(y) \bigg) ~,
\end{split}
\end{equation}
where $(b_i,g_i,e,p)$ are the co-ordinates introduced in \eqref{eq:coords on zastava}.
\end{lem}
\begin{proof}
A proof essentially amounts to having a good understanding of the Gaussian decomposition of the matrix representative $y\in \Gr^0_{-\alpha}$. 

Write $ y = U D L$, we want to compute a number of the matrix entries of $U\in N_1[\![t^{-1}]\!]$. To do so we can proceed recursively. Let $y_m$ be the lower $m\times m$ block of $y$.

Suppose $y_m = U_m D_m L_m$ is a Gaussian decomposition of the lower block. Then we can compute the decomposition of $y_{m+1} = U_{m+1} D_{m+1} L_{m+1}$ as follows. Write, 
\begin{equation}
	y_{m+1} = 
	\begin{pmatrix}
	y_{n-m-1,n-m-1} & \underline{y}_{n-m,\bullet}\\
	\underline{y}_{\bullet,n-m} & y_m
	\end{pmatrix}
\end{equation}
with 
\begin{equation}
	U_{m+1} = 
	\begin{pmatrix}
		1 & \underline{u}_{m+1}\\
		0 & U_m 
	\end{pmatrix}
	~,\qquad
	D_{m+1} = 
	\begin{pmatrix}
		d_{m+1} & 0 \\
		0 & D_m
	\end{pmatrix}
	~, \qquad
	L_{m+1} = 
	\begin{pmatrix}
	1 & 0 \\
	\underline{l}_{m+1} & L_m
	\end{pmatrix}
	~,
\end{equation}
Then, one can show that $d_{m+1} = {\rm det}(y_{m+1})/{\rm det}(y_m)$, and 
\begin{equation}\label{eq:UDL recursive}
	\underline{u}_{m+1} = \underline{y}_{n-m,\bullet} \cdot L_m^{-1} \cdot D_m^{-1}~,\qquad \underline{l}_{m+1} =  D_m^{-1}\cdot U_m^{-1}\cdot \underline{y}_{\bullet,n-m}
\end{equation}

Suppose, $m<n-1$, then $\underline{y}_{n-m,\bullet} = (0,0,\dots, b_{m-1})$ and $\underline{y}_{\bullet,n-m} = (0, 0,\dots, g_{m-1})^{\rm T}$. So from \eqref{eq:UDL recursive}, we see that 
\begin{equation}
	\underline{u}_{m+1} = (\ast,\ast,\dots,\ast,b_{m-1}/D_1)~,\qquad \underline{l}_{m+1} = (\ast, \ast, \dots, \ast, -g_{m-1}/D_1)~,
\end{equation}
where $\ast$ represent nontrivial terms yet to be determined.

Looking at the realization $y$, we see that ${\rm det} \, y_m = t- p - \sum_{i=m}^{k-1} b_i g_i $, and so 
\begin{equation}\label{eq:D decomposition}
	D = {\rm diag}\bigg( \frac{1}{t-p}, \frac{t-p}{t-p-b_{k-1} g_{k-1}},\frac{t-p-b_{k-1}g_{k-1}}{t-p+b_{k_2}g_{k-1}-b_{k-1}g_{k-1}}, \dots, \frac{t-p-\sum_{i=2}^{k-1} b_ig_i}{t-p-\sum_{i=1}^{k-1} b_ig_i}, t-p - \sum_{i=1}^{k-1} b_ig_i\bigg)~,
\end{equation}
which of course factorizes as $D = t^{-\alpha} D'$ with $D'\in T_1[\![t^{-1}]\!]$, as expected. Moreover, we see that the terms $D_{ii}\sim O(t^0)$ for  $1<i<k$, with $D_{kk}\sim O(t)$ and $D_{11}\sim O(t^{-1})$.

We know that $y_1 = t-p - \sum_{i=1}^{k-1} b_i g_i$ has a Gaussian decomposition with 
\begin{equation}
U_1=L_1 =1 ~,\qquad D_1 = t-p - \sum_{i=1}^{k-1}b_ig_i~.
\end{equation}

We see that 
\begin{equation}
\underline{u}_j = (\ast,\ast, \dots, \ast, \frac{b_{j-1}	}{t-p-\sum_{i=1}^{k-1} b_i g_i})~, \qquad  \underline{l}_j = (\ast,\ast,\dots,\ast, \frac{-g_{j-1}}{t-p-\sum_{i=1}^{k-1} b_i g_i})~,
\end{equation}
for $j<k$. From this, we see that 
\begin{equation}
	\Phi_\alpha(y) = (b_1,b_2,b_3,\dots,b_{k-1},\ast)~,
\end{equation}
so we have yet to determine the final component of $\Phi_{\alpha}$ and the components of $\zeta_{\alpha}$. These are extracted from the top row of $U$, which is computed by 
\begin{equation}
\underline{u}_k = y_{1,\bullet}\cdot L_{k-1}^{-1} D_{k-1}^{-1} = (0,0,\dots, e)\cdot	L_{k-1}^{-1} D_{k-1}^{-1}~,
\end{equation}
which just picks out the bottom row of $L_{k-1}^{-1}D_{k-1}^{-1}$. We see immediately that
\begin{equation}
	\underline{u}_{k} = (\ast,\ast, \dots,\ast, \frac{e}{t-p-\sum_{i=1}^{k-1} b_i g_i})~,
\end{equation}
which fixes $\Phi_{\alpha}= (b_1,b_2,\dots,b_{k-1},e)$. Now we need to fix the rest of the entries. First note that $L_{k-1} \in N_{-,k,1}[\![t^{-1}]\!]$, since $L_k \in N_{-,1}[\![t^{-1}]\!]$. So its inverse is of the form
\begin{equation}
L_{k-1}^{-1} = 
\begin{pmatrix}
1 & 0 & \dots & 0 & 0 \\
0 & 1 & \dots & 0 & 0\\
\vdots & \vdots & \vdots & \vdots & \vdots\\
g_{k-1}t^{-1} + O(t^{-2}) & g_{k-2}t^{-1} + O(t^{-2}) & \dots & g_1t^{-1} +O(t^{-2}) & 1
\end{pmatrix}
\end{equation}
Moreover from \eqref{eq:D decomposition}, we see that 
\begin{equation}
	D_k^{-1} = \bigg( 1 +O(t{^-1}), 1+ O(t^{-1}), \dots ,1+O(t^{-1}), \frac{1}{t-p - \sum_{i=1}^{k-1} b_i g_i}\bigg)~.
\end{equation}
Putting this together, we see that
\begin{equation}
\begin{split}
\underline{u}_k &= 	(0,0,\dots, e)\cdot	L_{k-1}^{-1} D_{k-1}^{-1}~, \\
&= \bigg( eg_{k-1}t^{-1}+O(t^{-2}), eg_{k-2} t^{-1} + O(t^{-2}), \dots, e g_1t^{-1} + O(t^{-2}), \frac{e}{t-p-\sum_{i=1}^{k-1} b_i g_i})~,
\end{split}
\end{equation}
which fixes $\zeta_\alpha(y)$ as desired.
\end{proof}

As an immediate corollary, we have the following result.
\begin{prop}\label{prop:antidominant moment map}
The action of $\IG_\alpha$ on $\Gr_{-\alpha}^0$ is Hamiltonian with moment map $\Phi_\alpha$, which corresponds to projection to the co-ordinates $(b_i, e)$.
\end{prop}
\begin{proof}
From \ref{lem:Ga action on antidominant is translation} we see that the action of $\IG_\alpha$ preserves the symplectic form and that the moment map is the projection to the $(b_i,e)$ co-ordinates. What remains is to show that the map $\Phi_\alpha$ agrees with this projection but this is precisely the content of Lemma \ref{lem:antidominant UTU decomp}.
\end{proof}
We see, therefore, that the image of $\Phi_\alpha$ defines an open set 
\begin{equation}
	\overset{\circ}{\Gf^*_\alpha} \equiv \Phi_\alpha(\Gr^0_{-\alpha})\subset \Gf^*_\alpha~.
\end{equation}
We write $\chi_\alpha$ for the character in $\open{\Gf}_\alpha$ corresponding to $b_i=0,e=1$.

For convenience we write, $(\open{\Gf^*}_\alpha)_\mu$ for the preimage of $\overset{\circ}{\Gf^*}_\alpha$ under $\Phi_\alpha: \Gr_\mu \rightarrow \Gf_\alpha$.

Define a map $\xi_\alpha: (\open{\Gf^*}_\alpha)_\mu \rightarrow \Gr_{-\alpha}^0$ on the usual co-ordinates by
\begin{equation}
\begin{split}
b_i(\xi_\alpha(y)) &= \Phi_{\alpha,i}(y)~, \text{ for } i<k \\ e(\xi_\alpha(y)) &= \Phi_{\alpha,k}(y)~,\\
g_i(\xi_\alpha(y)) &= \zeta_{\alpha,{k-i}}(y)/ \Phi_{\alpha,k}(y) ~,\text{ for } i<k ~, \\ 
p(\xi_\alpha(y)) &= \bigg(\zeta_{\alpha,k}(y) - \sum_{i=1}^{k-1} \Phi_{\alpha,i}(y)\zeta_{\alpha,i}(y)\bigg)/ \Phi_{\alpha,k}(y)^{-1}~.
\end{split}
\end{equation}
Note that in $(\open{\Gf^*}_\alpha)_\mu$, $\Phi_{\alpha,k}(y)$ is nonzero, hence this map is well-defined.

\begin{prop}\label{prop:xi is id on antidominant}
Write $(\open{\Gf}_\alpha)_{-\alpha}^0$ for the preimage of $\open{\Gf}_\alpha$ under the moment map $\Phi_\alpha:\Gr^0_{-\alpha} \rightarrow \Gf_\alpha$. Then the map
\begin{equation}
 (\overset{\circ}{\Gf^*_\alpha})_{-\alpha}^0\xrightarrow{\xi_\alpha} \Gr^0_{\alpha}
\end{equation}
is the identity morphism.
\end{prop}
\begin{proof}
Proposition \ref{prop:antidominant moment map} implies that $(\open{\Gf^*}_\alpha)_{-\alpha}^0 = \Gr_{-\alpha}^0$ and $\xi_\alpha(y) = y$ follows from Lemma \ref{lem:antidominant UTU decomp} and the definition of $\xi_\alpha$.
\end{proof}

Essentially, the action of $\IG_\alpha$ on $\Gr^0_{-\alpha}$ is the natural left action coming from identifying it with an open subset of $T^*\IG_\alpha$ as in Lemma \ref{lem:antidominant_abelian}.


\subsection{Equivariance of the multiplication map}

To prove our result on inverse Hamiltonian reduction we will need to examine how the action of $\IG_\alpha$ interacts with the multiplication map.

First, let us glean what we can from the multiplication map ${\bf m }: \Gr^0_{-\alpha}\times \Gr_{\mu+\alpha} \hookrightarrow \Gr_\mu$.

\begin{lem}\label{lem:mult antidominant slice}
Let $\mu$ be a coweight of $\rm PGL_n$, $y_1\in \Gr^0_{-\alpha}$ and $y_2 \in \Gr_{\mu+\alpha}$, then
\begin{enumerate}[(i)]
	\item $y_1y_2 \in N(\!(t^{-1})\!)^\alpha_1 t^\mu T_1[\![t^{-1}]\!] N_-(\!(t^{-1})\!)$.
	\item There exist  $n\in N^\alpha[t]$ and $n_-\in N_-[t]$ such that $\mathbf{m}(y_1,y_2)= \pi(y_1 y_2)= ny_1y_2n_-$.
	\item $\Phi_\alpha(\mathbf{m}(y_1,y_2)) = \Phi_\alpha(y_1)$ and $\zeta_\alpha(\mathbf{m}(y_1,y_2) = \zeta_\alpha(y_1)$.
\end{enumerate}
\end{lem}
\begin{proof}
Let $y_1 = u_1t^{-\alpha} h_1 u_{-,1}$ and $y_2 = u_2 t^{\mu+\alpha} h_2 u_{-,2}$ be Gaussian decompositions, with $u_i \in N_1[\![t^{-1}]\!]$,  $h_i\in T_1[\![T^{-1}]\!]$, and $u_{-,i}\in N_{-,1}[\![t^{-1}]\!]$. Moreover, note that $u_1\in {\rm im}\, x_\alpha$.

So,
\begin{equation}
y_1 y_2 = u_1 t^{-\alpha} h_1 u_{-,1}u_2 h_2t^{\mu+\alpha} u_{-,2}~.	
\end{equation}
Since $h_1 u_{-,1}u_2h_2\in G_1[\![t^{-1}]\!]$, it has a Gaussian decomposition $h_1u_{-,1}u_2h_2= u_3h_3u_{-,3}$ where $u_3\in N_1[\![t^{-1}]\!]$, $h_3\in T_1[\![t^{-1}]\!]$, and $u_{-,3}\in U_{-,1}[\![t^{-1}]\!]$. Now
\begin{equation}
\begin{split}
	y_1y_2 &= u_1 t^{-\alpha} u_3 h_3 u_{-,3} t^{\mu+\alpha} u_{-,2}~,\\
	&= u_1 \big( t^{-\alpha} u_3 t^\alpha\big) h_3 t^\mu \underbrace{\big(t^{-\mu-\alpha} u_{-,3} t^{\mu+\alpha}\big)u_{-,2}}_{\in N_-(\!(t^{-1})\!)}~.
\end{split}
\end{equation}
So to establish $(i)$ we need to show that $u_1 (t^{-\alpha}u_3 t^{\alpha})\in N(\!(t^{-1})\!)^\alpha_1$. To do so, let us focus in on the lower $(k+1)\times (k+1)$ block. Write $u_1 = x_\alpha(u_1')$ and $u_3 = x_{\alpha}(u_3')\cdot u_3''$, with $u_3''\in N^\alpha_1[\![t^{-1}]\!]$ for
\begin{equation}
\begin{split}
	u_1' &= 
	\begin{pmatrix}
	1 & a_{1,1} & a_{1,2} & \dots & a_{1,k-1} & a_{1,k}\\
	0 & 1 & a_{2,2} & \dots & a_{2,k-1} & a_{2,k} \\
	\vdots & \vdots & \vdots & \vdots & \vdots & \vdots\\
	0 & 0 & 0 &  \dots & 1 & a_{k,k} \\
	0 & 0 & \dots & 0 & 0 & 1
	\end{pmatrix}
	~,\\
	u_3' &= 
	\begin{pmatrix}
	1 & c_{1,1} & c_{1,2} & \dots & c_{1,k-1} & c_{1,k}\\
	0 & 1 & c_{2,2} & \dots & c_{2,k-1} & c_{2,k} \\
	\vdots & \vdots & \vdots & \vdots & \vdots & \vdots\\
	0 & 0 & 0 &  \dots & 1 & c_{k,k} \\
	0 & 0 & \dots & 0 & 0 & 1
	\end{pmatrix}
	~,
\end{split}
\end{equation}
with $a_{i,j},c_{i,j}\in t^{-1}\ik[\![t^{-1}]\!]$. Then,
\begin{equation}\label{eq:conjugate u_3'}
	t^{-\alpha}u_3't^\alpha = 
	\begin{pmatrix}
	1 & t^{-1}c_{1,1} & t^{-1}c_{1,2} & \dots & t^{-1}c_{1,k-1} & t^{-2}c_{1,k}\\
	0 & 1 & c_{2,2} & \dots & c_{2,k-1} & t^{-1}c_{2,k} \\
	\vdots & \vdots & \vdots & \vdots & \vdots & \vdots\\
	0 & 0 & 0 &  \dots & 1 & t^{-1}c_{k,k} \\
	0 & 0 & \dots & 0 & 0 & 1
	\end{pmatrix}
	~,
\end{equation}
\ie, the outer shell gets scaled by $t^{-1}$ and the top right corner gets scaled by $t^{-2}$ while all other entries are unchanged. So 
\begin{equation}
	u_1 t^{-\alpha}u_3 t^{\alpha} = x_\alpha( u_1')t^{-\alpha} x_\alpha(u_3')t^\alpha\big( t^-\alpha  u_3'' t^\alpha\big) = x_\alpha(u') \big(t^{-\alpha} u_3'' t^{\alpha}\big)~,
\end{equation}
where $u'\in N_{k,1}[\![t^{-1}]\!]$ is in the first congruent subgroup of ${\rm GL}_k$. Noting that $\big(t^{-\alpha} u_3'' t^{\alpha}\big)\in N^\alpha(\!(t^{-1})\!)$ we see that $u_1 t^{-\alpha}u_3 t^{\alpha} \in N(\!(t^{-1})\!)^\alpha_1$, establishing $(i)$.


Now we move on to $(ii)$. We know that there are unique $n\in N[t]$ and $n_-\in N-[t]$ such that $ny_1y_2n_-\in \Gr_\mu$, so it suffices to show that $n\in N^\alpha[t]$. But this has to be the case, since $ u_1 t^{-\alpha}u_3 t^{\alpha}\in N(\!(t^{-1})\!)^\alpha_1$ already has the $(k+1)\times (k+1)$ block in ${\rm im}\, x_\alpha$ in the appropriate congruent subgroup. Thus, if $n$ had any component in the image of $x_\alpha$, it would cause this block to not be in the correct form.

So all that remains is to establish $(iii)$. Note that $\Phi_\alpha$ and $\zeta_\alpha$ only see the $(k+1)\times(k+1)$ block and so 
\begin{equation}
\begin{split}
	\Phi_\alpha(n u_1 t^{-\alpha}u_3 t^\alpha) &= \Phi_\alpha(u_1 t^{-\alpha}u_3 t^\alpha)~,\\
	\zeta_\alpha(n u_1 t^{-\alpha}u_3 t^\alpha) &= \zeta_\alpha(u_1 t^{-\alpha} u_3 t^\alpha)~.\\ 
\end{split}
\end{equation}
Zooming into the $(k+1)\times(k+1)$ block and using \eqref{eq:conjugate u_3'}, one sees that 
\begin{equation}
u'_1 t^{-\alpha}u_3't^\alpha = 
\begin{pmatrix}
	1 & a_{1,1} + O(t^{-2}) & a_{1,2} + O(t^{-2}) & \dots & a_{1,k-1} + O(t^{-2}) & a_{1,k} + O(t^{-3})\\
	0 & 1 & \ast & \dots & \ast & a_{2,k} + O(t^{-2}) \\
	\vdots & \vdots & \vdots & \vdots & \vdots & \vdots\\
	0 & 0 & 0 &  \dots & 1 & a_{k,k} + O(t^{-2}) \\
	0 & 0 & \dots & 0 & 0 & 1
\end{pmatrix}
~,
\end{equation}
where $\ast$ are undetermined terms. So, from the definitions of $\Phi_\alpha$ and $\zeta_\alpha$, we see that
\begin{equation}
\begin{split}
\Phi_\alpha(u_1 t^{-\alpha}u_3 t^\alpha) &= \big( a_{k,k}^{(1)}, a_{k-1,k}^{(1)},\dots, a_{1,k}^{(1)}\big) = \Phi_\alpha(y_1) ~,\\
\zeta_\alpha(u_1t^{-\alpha}u_3 t^\alpha) &= \big(a_{11}^{(1)}, a_{1,2}^{(1)},\dots, a_{1,k-1}^{(1)}, a_{1,k}^{(2)}\big) = \zeta_\alpha(y_1)~,
\end{split}
\end{equation}
as desired.
\end{proof}

This lemma is enough to establish one of our main results
\begin{thm}\label{thm:Ga action is equivariant}
	The map $\mathbf{m}: \Gr^0_{-\alpha} \times \Gr_{\mu+\alpha} \hookrightarrow \Gr_\mu$ is $\IG_\alpha$ equivariant, where the action on the domain is the product of the natural action on $\Gr^0_{-\alpha}$ and the trivial action on $\Gr_{\mu+\alpha}$.
\end{thm}
\begin{proof}
Let $\underline{v}\in \IG_\alpha$, $y_1\in \Gr^0_{-\alpha}$, and $y_2\in \Gr_{\mu+\alpha}$. On the left hand side, 
\begin{equation}
	\underline{v}\cdot(y_1,y_2) = (x_{-\alpha}(\underline{v})y_1,y_2)~,
\end{equation}
by Lemma \ref{lem:Ga action on antidominant is translation}. The action on the right hand side is given by 
\begin{equation}\label{eq:lhs action of Ga}
	\underline{v}\cdot \mathbf{m}(y_1,y_2) = \pi(x_{-\alpha}(\underline{v}){\bf m}(y_1,y_2))~.
\end{equation}
By Lemma \ref{lem:mult antidominant slice}\,$(ii)$, $\mathbf{m}(y_1,y_2) = ny_1y_2n_-$ for some $n\in N^\alpha[t]$ and $n_-\in N_-[t]$. Therefore,
\begin{equation}
\begin{split}
	x_{-\alpha}(\underline{v}){\bf m}(y_1,y_2) &= x_{-\alpha}(\underline{v})ny_1y_2n_-~,\\
	&= \big(x_{-\alpha}(\underline{v}) n x_{-\alpha}(\underline{v})^{-1}\big) x_{-\alpha}(\underline{v})y_1y_2 n_-~.
\end{split}
\end{equation}
But, from Proposition \ref{prop:conjugating Nalpha}\,$(iii)$, $x_{-\alpha}(\underline{v})nx_{-\alpha}(\underline{v})^{-1} = n' \in N^\alpha[t]$. Putting things together, 
\begin{equation}
	\underline{v}\cdot {\bf m}(y_1,y_2) = \pi(n' x_{-\alpha}
	(\underline{v}) y_1y_2 n_-)  = \pi(x_{-\alpha}(\underline{v}) y_1 y_2) = x_{-\alpha}(\underline{v}) y_1 y_2
\end{equation}
which agrees with the action on the left hand side in \eqref{eq:lhs action of Ga}.
\end{proof}


\subsection{Inverse Hamiltonian reduction for the translation group}

From Lemma \ref{lem:mult antidominant slice}$\,(iii)$, and Proposition \ref{prop:antidominant moment map}, we see that the composition $\Gr^0_{-\alpha} \times \Gr_{\mu+\alpha} \xrightarrow{\bf m} \Gr_\mu \xrightarrow{\Phi_\alpha} \Gf^*_\alpha$ factors through the inclusion $\open{\Gf}_\alpha\hookrightarrow \Gf^*_\alpha$. So, the image of the multiplication map lies inside the preimage of $\overset{\circ}{\Gf^*_\alpha}$, \ie,
\begin{equation}
	{\bf m}: \Gr_{-\alpha}^0 \times \Gr_{\mu+\alpha} \hookrightarrow  (\overset{\circ}{\Gf^*}_\alpha)_\mu~.
\end{equation}

We wish to show that this inclusion is actually an isomorphism and so we proceed to define an inverse to ${\bf m}$. First, we prove a small but useful result.
\begin{lem}\label{lem:xi only sees antidominant}
	Let $y_1\in \Gr^0_{-\alpha}$ and $y_2 \in \Gr_{\mu+\alpha}$, then 
	\begin{equation}
	\xi_\alpha(\mathbf{m}(y_1,y_2)) = y_1	~,
	\end{equation}
\end{lem}
\begin{proof}
	This follows immediately from Lemma \ref{lem:mult antidominant slice}\,$(iii)$ and Proposition \ref{prop:xi is id on antidominant}, since $\xi_\alpha$ is fully determined by the values of $\Phi_\alpha$ and $\zeta_\alpha$.
\end{proof}
So we see that $\xi_\alpha$ gives a putative section of $	{\bf m}: \Gr_{-\alpha}^0 \times \Gr_{\mu+\alpha} \hookrightarrow \Phi_\alpha^{-1} (\overset{\circ}{\Gf_\alpha^*})_\mu$ to the first factor $\Gr_{-\alpha}^0$. We want to extend this to a full section.

\begin{lem}\label{lem:inverse to m is well defined}
	Let $y\in  (\overset{\circ}{\Gf^*_\alpha})_\mu$, then 
	\begin{equation}
	\xi_\alpha(y)^{-1}(y)\in N(\!(t^{-1})\!)^\alpha_1 t^{\mu+\alpha} T_1[\![t^{-1}]\!] N_-(\!(t^{-1})\!)~.
	\end{equation}
	In particular, $\pi(\xi_\alpha(y)^{-1}y)\in \Gr_{\mu+\alpha}$ and there exist $n\in N^\alpha[t]$ and $n_-\in N_-[t]$ such that 
	\begin{equation}
		\pi(\xi_\alpha(y)^{-1}y) = n \xi_\alpha(y)^{-1})y n_-~.
	\end{equation}
\end{lem}
\begin{proof}
Write $\xi_{\alpha}(y) = \tau_\alpha(y')$, where 
	\begin{equation}
	y' = 
	\begin{pmatrix}
	0 & 0 & 0 &  \dots & 0 & e\\
	0 & 1& 0 & \dots & 0 & b_{k-1}\\
	0 & 0& 1 & \dots & 0 & b_{k-2}\\
	\vdots & \vdots & \vdots & \vdots & \vdots & \vdots \\
	0 & 0 & 0 &\dots & 1 & b_{1} \\
	-1/e & -g_{k-1} & -g_{k-2} & \dots & -g_1 & t - p- \sum_{i=1}^{k-1} b_i g_i
	\end{pmatrix}~.
	\end{equation}
	in the usual co-ordinates on $\Gr^0_{-\alpha}$. Since $y\in \Gr_\mu$, it has a Gaussian decomposition of the form,
	\begin{equation}
		y = x_\alpha(u) u' h t^\mu u_-~,
	\end{equation}
where $u\in N_{1,k}[\![t^{-1}]\!]$, $u'\in N^\alpha_1[\![t^{-1}]\!]$, $h\in T_1[\![t^{-1}]\!]$, and $u_- \in N_-[\![t^{-1}]\!]$. Then,
	\begin{equation}
		\xi_{\alpha}(y)^{-1}y = \tau_{\alpha}(y')^{-1} x_{\alpha}(u) u' t^\mu u_{-}~,
	\end{equation}

	Our strategy is the following. We claim that $(y')^{-1} u$ is in the ${\rm GL}_{k+1,1}[\![t^{-1}]\!]$ orbit of $t^{\alpha}$, \ie,
	\begin{equation}
	(y')^{-1} u = y''t^{\alpha}	~,
	\end{equation}
	for some $y''$ in the first congruent subgroup ${\rm GL}_{k+1,1}[\![t^{-1}]\!]$. Therefore, $y''t^{\alpha}$ has a Gaussian decomposition 
	\begin{equation}
		y''t^{\alpha} = u'' t^{\alpha} l_-~,
	\end{equation}
	where $u''\in N_{k+1,1}[\![t^{-1}]\!]$ and $l_-\in B_{k+1,-}(\!(t^{-1})\!)$. Thus,
	\begin{equation}
	\begin{split}
		\xi_{\alpha}(y)^{-1} y &= x_{\alpha}(u'') \tau_\alpha( t^{\alpha} l_-)u' ht^\mu u_-~,\\
		&= x_{\alpha}(u'') t^{\alpha} \big( \tau_\alpha(t^{\alpha}l_-) u'\tau_\alpha(t^{\alpha}l_-)^{-1}\big) \tau_\alpha(t^{\alpha}l_-) t^\mu h u_-~,
	\end{split}
	\end{equation}
	Now Proposition \ref{prop:conjugating Nalpha} implies that $ \big( \tau_\alpha(t^{\alpha}l_-) u'\tau_\alpha(t^{\alpha}l_-)^{-1}\big) \in N^\alpha(\!(t^{-1})\!)$ and so $x_{\alpha}(u'')\big( \tau_\alpha(t^{\alpha}l_-) u'\tau_\alpha(t^{\alpha}l_-)^{-1}\big) \in N^\alpha(\!(t^{-1})\!)^\alpha_1$ by definition. The term $\tau_\alpha(t^{\alpha}l_-) t^\mu h u_-$ has a Gaussian decomposition of the form $ t^{\mu+\alpha}T_1[\![t^{-1}]\!] N_-(\!(t^{-1})\!)$ as desired. 

	Thus, it is sufficient to show that the matrix $y''$ is in the first congruent subgroup, ${\rm GL}_{k+1,1}[\![t^{-1}]\!]$. To do so we have to massage the expression $(y')^{-1}u$ and, crucially, use the fact that the first few Fourier modes appearing in the entries of $u$ actually determine the entries of $y'$, by way of the map $\xi_\alpha$.

One computes the inverse of $y'$ to be 
\begin{equation}
	(y')^{-1} = 
	\begin{pmatrix}
	 t-p & -eg_{k-1} & -g_{k-2} &\dots & -e g_1 & -e\\
	 -b_{k-1}/e & 1 & 0 & \dots & 0 & 0 \\
	 -b_{k-2}/e & 0  & 1 & \dots & 0 & 0 \\
	 \vdots & \vdots & \vdots & \vdots & \vdots & \vdots \\
	 -b_1/e & 0 & 0 & \dots & 1 & 0 \\
	 \frac{1}{e} & 0 & 0 & \dots & 0 & 0
	\end{pmatrix}
	~.
\end{equation}
Write 
\begin{equation}
	u = 
	\begin{pmatrix}
	1 & u_{1,1} & u_{1,2} & \dots & u_{1,k-1} & u_{1,k}\\
	0 & 1 & u_{2,2} & \dots & u_{2,k-1} & u_{2,k} \\
	\vdots & \vdots & \vdots & \vdots & \vdots & \vdots\\
	0 & 0 & 0 &  \dots & 1 & u_{k,k} \\
	0 & 0 & \dots & 0 & 0 & 1
	\end{pmatrix}
	~,
\end{equation}
where  $u_{i,j}\in t^{-1} \ik[\![t^{-1}]\!]$. Performing the matrix multiplication we have, schematically,
\begin{equation}
	(y')^{-1} \cdot u  = 
	\begin{pmatrix}
	t-p & \underline{q} & s \\
	\underline{b}^T & Q & \underline{c}^T\\
	1/e & \underline{g} & u_{1,k}/e
	\end{pmatrix}
	~,
\end{equation}
where $\underline{b} = (-b_{k-1}/e,-b_{k-2}/e, \dots, -b_1/e)$, $\underline{g} = ( u_{1,1}/e, u_{1,2}/e,\dots,u_{1,k-1}/e)$. Moreover $\underline{q}= (q_1,q_2,\dots, q_{k-1})$ and $\underline{c} = (c_1,c_2, \dots, c_{k-1})$, with
\begin{equation}
\begin{split}
 q_i &= (t-p) u_{1,i} - \sum_{j=2}^{i} eg_{k-j+1} u_{j,i}  - eg_{k-i}~,\\
 c_i &= u_{i+1,k} -b_{k-i}u_{1,k}/e~,
\end{split}
\end{equation}
and 
\begin{equation}
 s = (t-p) u_{1,k} - \sum_{j=2}^{k}e g_{k-j+1}u_{j,k} - e~.
\end{equation}
How about the $(k-1)\times (k-1)$ matrix $Q$? Its entries are given by
\begin{equation}
	Q_{ij} = \delta_{ij}\big( 1 - b_{k-i}u_{1,j}/e\big)  + (\delta_{ij}-1)b_{k-i}/e u_{1,j} + \Theta_{ij}u_{i+1,j}~,
\end{equation}
where $\Theta_{ij}$ is the discrete Heaviside function, \ie, $\Theta_{ij} = 1$ if $j>i$ and $\Theta_{ij}=0$ otherwise. Really, all that matters is the observation that the off diagonal terms of $Q$ are homogeneous in the $u$ variables and so $Q_{ij} \sim O(t^{-1})$ and the diagonal terms satisfy, $Q_{ii} \sim 1+ O(t^{-1})$.

Now we set $(y')^{-1}\cdot u = y'' \cdot t^{\alpha}$, so 
\begin{equation}
	y''=
	\begin{pmatrix}
	1-p/t & \underline{q} & t s \\
	t^{-1}\underline{b}^T & Q & t \underline{c}^T\\
	t^{-1}/e & \underline{g} & tu_{1,k}/e
	\end{pmatrix}
	~,
\end{equation}
and recall that we wish to show that $y''$ is in the first congruent subgroup of ${\rm GL}_{k+1}[\![t^{-1}]\!]$. Note that 
\begin{equation}
	t u_{1,k}/e = t \big(et^{-1} + O(t^{-2})\big)/e = 1 + O(t^{-1}) ~,
\end{equation}
since $e = u_{1,k}^{(1)}$ from the definition of $\xi_{\alpha}$. We have also shown that $Q$ is in the appropriate form and, from their definitions, $\underline{g}$, $t^{-1}\underline{b}$ and $t^{-1}/e$ are all $O(t^{-1})$. So, now we must show that $\underline{q}, ts,t \underline{c}$ are all $O(t^{-1})$.

Let us start with $t \underline{c}$, 
\begin{equation}
	t c_i = t( u_{i+1,k} - b_{k-i} u_{1,k}/e) = t\big( b_{k-i}t^{-1} - b_{k-i} (et^{-1})/e + O(t^{-2})\big) = O(t^{-1})~,
\end{equation}
where we used the fact that $u_{i+1,k}^{(1)}  = b_{k-i}$, and $u_{1,k}^{(1)}= e$, again coming from the definition of $\xi_{\alpha}$. Now we move on to $\underline{q}$,
\begin{equation}
	q_i = (t-p) u_{1,i} -\sum_{j=2}^i e g_{k-j+1} u_{j,i} - e g_{k-i} = u_{1,i}^{(1)} - g_{k-i} +O(t^{-1})  = O(t^{-1})
\end{equation}
where in the final equality we use the fact that $u_{1,i}^{(1)} = + g_{k-i}$ by the definition of $\xi_\alpha$.  Now, all that remains is to check $t s$ or, equivalently, to show that $s\sim O(t^{-2})$.
\begin{equation}
\begin{split}
	s &=  (t-p)u_{1,k} - \sum_{j=2}^k eg_{k-j+1},u_{j,k} -e~,\\
	& =  (t-p)\big( u_{1,k}^{(1)}t^{-1} + u_{1,k}^{(2)}t^{-2} +O(t^{-3})\big) - e \sum_{j=2}^{k}g_{k-j+1}u_{j,k}^{(1)}t^{-1} -e + O(t^{-2})~,\\
	&= (u_{1,k}^{(1)}-e) + \big(u_{1,k}^{(2)} -p u_{1,k}^{(1)} + e\sum_{j=2}^{k} g_{k-j+1}u_{j,k}^{(1)}\big)t^{-1} + O(t^{-2})~.
\end{split}
\end{equation}
Unwinding the definition of $\xi_\alpha$, we see that 
\begin{equation}
	u_{1,k}^{(1)} = e~,\qquad
	\sum_{j=2}^k g_{k-j+1}u_{j,k}^{(1)} = -\sum_{j=1}^{k-1} b_j g_j~,\qquad
	u_{1,k}^{(2)} = pe + e\sum_{j=1}^{k-1} b_i g_i~,
\end{equation}
Substituting this in, we see that 
\begin{equation}
	s = (e-e) + \big(pe + e\sum_{j=1}^{k-1} b_ig_i  -p e -e\sum_{j=1}^{k-1}b_ig_i\big)t^{-1} + O(t^{-2}) = O(t^{-2})~,
\end{equation}
as desired.

Thus, $y''$ is in the first congruent subgroup, ${\rm GL}_{k+1,1}[\![t^{-1}]\!]$ and we are done.
\end{proof}

With this lemma in hand, we are able to define a putative inverse to ${\bf m}: \Gr_{-\alpha}^0 \times \Gr_{\mu+\alpha}\rightarrow \Gr_\mu$. Define  
\begin{equation}
\begin{split}
	F_\alpha:  (\overset{\circ}{\Gf^*}_\alpha)_\mu &\rightarrow \Gr^0_{-\alpha} \times \Gr_{\mu+\alpha}\\
	y & \mapsto \big( \xi_\alpha(y), \pi(\xi_\alpha(y)^{-1}y)\big)
\end{split}
\end{equation}
From Lemma \ref{lem:inverse to m is well defined} we see that $F_{\alpha}$ is indeed well-defined. 

\begin{prop}\label{prop:mult is iso}
Let $\mu$ be a coweight of $\rm PGL_n$, then $\mathbf{m}$ and $F_\alpha$ define an isomorphism
\begin{equation}
\begin{tikzcd}
	\Gr^0_{-\alpha}\times \Gr_{\mu+\alpha}\ar[r, "\mathbf{m}", bend left=40] & (\overset{\circ}{\Gf^*}_\alpha)_\mu \ar[l,"F_\alpha",bend left=40]
\end{tikzcd}
\end{equation}
\end{prop}
\begin{proof}
We proceed by directly showing that $\mathbf{m}\circ F_\alpha = {\rm id}_{(\overset{\circ}{\Gf^*}_\alpha)_\mu}$ and $F_\alpha\circ \mathbf{m} = {\rm id}_{\Gr^0_{-\alpha} \times  \Gr_{\mu+\alpha}}$. 

First, we check $\mathbf{m}\circ F_\alpha = {\rm id}_{(\overset{\circ}{\Gf^*}_\alpha)_\mu}$. Let $y \in \overset{\circ}{\Gr}_\mu$, then 
\begin{equation}
	F_\alpha(y) = \big(\xi_\alpha(y) , \pi(\xi_\alpha(y)^{-1}y)\big)~,
\end{equation}
From Lemma \ref{lem:mult antidominant slice}, we know that there exist $n\in N^\alpha[t]$ and $n_-\in N_-[t]$ such that $\pi(\xi_\alpha(y)^{-1}y) = n \xi_\alpha(y)^{-1} y n_-$. So,
\begin{equation}
	\mathbf{m}(F_{\alpha}(y) = \pi( \xi_\alpha(y), \pi(\xi_\alpha(y)^{-1}y)) = \pi(\xi_\alpha(y) n\xi_\alpha(y)^{-1} yn_-)~.
\end{equation}
But, from Proposition \ref{prop:conjugating Nalpha}\,$(iii)$ we see that $\xi_\alpha(y) n \xi_\alpha(y)^{-1} = n' \in N^\alpha[t]$. Thus,
\begin{equation}
	{\bf m}(F_\alpha(y)) = \pi ( n' y n_- ) = y~,
\end{equation}
as desired.

Now, we check $F_\alpha\circ \mathbf{m} = {\rm id}_{\Gr^0_{-\alpha} \times  \Gr_{\mu+\alpha}}$. Let $y_1\in \Gr^0_{-\alpha}$ and $y_2\in \Gr_{\mu+\alpha}$. From Lemma \ref{lem:mult antidominant slice}, we know that there exist some $n\in N^\alpha[t]$ and $n_-\in N_-[t]$ (not to be confused with the ones from the previous paragraph) such that
\begin{equation}
	\mathbf{m}(y_1,y_2) = ny_1y_2n_-~.
\end{equation}
By Lemma \ref{lem:xi only sees antidominant}, we know that $\xi_\alpha(\mathbf{m}(y_1,y_2))=y_1$, and so 
\begin{equation}
	F_\alpha(\mathbf{m}(y_1,y_2)) = (y_1, \pi(y_1^{-1} \mathbf{m}(y_1,y_2))) = (y_1, \pi(y_1^{-1} ny_1y_2n_-))~.
\end{equation}
Note that $y_1\in {\rm im}\,\tau_\alpha$; so by Proposition \ref{prop:conjugating Nalpha}\,$(iii)$, we once again have that $y_1^{-1} ny_1 = n'\in N^\alpha[t]$. Therefore, 
\begin{equation}
	F_\alpha(\mathbf{m}(y_1,y_2)) =  (y_1, \pi(n'y_2n_-)) =(y_1,y_2)~.
\end{equation}
\end{proof}

Combining Theorem \ref{thm:Ga action is equivariant} and Proposition \ref{prop:mult is iso}, we have an immediate corollary:
\begin{cor}
	The action of  $\IG_\alpha$ on $\Gr_\mu$ is Hamiltonian with moment map $\Phi_\alpha$.
\end{cor}

\begin{thm}\label{thm:general reduction by stages}
Let $\mu$ be a coweight of $\PGL_n$, then we have an isomorphism of Poisson varieties
\begin{equation}
	\Gr_{\mu+\alpha} \cong \Gr_{\mu}\red{\chi_\alpha} \IG_\alpha
\end{equation}
\end{thm}
\begin{proof}
The character $\chi_\alpha$ is contained in $\open{\Gf^*}_\alpha$. By Proposition \ref{prop:mult is iso}, we can identify its preimage under the moment map with $\IG_\alpha\times \Gr_{\mu+\alpha}$. Moreover, by Theorem \ref{thm:Ga action is equivariant} this identification is $\IG_{\alpha}$-equivariant. Thus,
\begin{equation}
	\Gr_{\mu}\red{\chi_\alpha} \IG_\alpha \cong (\IG_\alpha \times \Gr_{\mu+\alpha})/\IG_\alpha \cong \Gr_{\mu+\alpha}~.
\end{equation}
\end{proof}

\begin{thm}\label{thm:general inverse reduction}
Let $\mu$ be a coweight of $\PGL_n$ and $\alpha$ a positive coroot, then the tuple $(\open{\Gf^*}_\alpha, \mathbf{m})$ is an inverse Hamiltonian reduction for $(\Gr_\mu, \IG_\alpha, \Phi_\alpha, \chi_\alpha)$.
\end{thm}
\begin{proof}
 Looking at Definition \ref{dfn:IHR}, we see that conditions $(i)$ and $(ii)$ are immediately satisfied, and Proposition \ref{prop:mult is iso} tells us that $\mathbf{m}$ satisfies condition $(iii)$. The $\IG_\hbar$-equivariance is established by Lemma \ref{lem:multiplication is hbar equiv}.
\end{proof}

Intersecting with ${\rm Gr}^\lambda$ and noting that the $\PGL_n(\OO)$-orbits are $\IG_\alpha$-stable, we have the immediate corollary.
\begin{cor}\label{cor:general inverse reduction 2}
Let $\lambda$ be a dominant coweight of $\rm PGL_n$ and $\mu$ any coweight such that $\lambda\ge \mu$ and let $\alpha$ be any positive coroot such that $\mu+\alpha\le \lambda$.
Then $\Gr^\lambda_\mu$ has a Hamiltonian action of $\IG_\alpha$ and we have an isomorphism
\begin{equation}
	\Gr_\mu^\lambda\red{\chi_\alpha}\IG_\alpha\cong \Gr_{\mu+\alpha}^\lambda~.
\end{equation}
Moreover, this reduction admits an inverse Hamiltonian reduction given by $(\open{\Gf^*}_\alpha, \mathbf{m})$.
\end{cor}
\begin{proof}
	This follows from Theorem \ref{thm:general inverse reduction}, since $\mathbf{m}$ preserves $\Gr^\lambda$ and since $\IG_\alpha$ preserves $\PGL_n(\OO)$ orbits.
\end{proof}


\section{Inverse Hamiltonian reduction in type \texorpdfstring{$A$}{A}}

In this section we apply the results of the previous section on inverse Hamiltonian reduction for generalized slices in the affine Grassmannian, via the identifications of the latter with Slodowy slices and their equivariant analogues recalled in Section \ref{Liethysec}, to produce inverse Hamiltonian reductions of finite W-algebras and their equivariant analogues, and thereby prove Theorems \ref{IRCFintrothm}, \ref{IRQFintrothm}, and \ref{IRHCintrothm} from the introduction.

Throughout this section, we fix $\mu$ to be a partition of $N$ of length $n-1$, thought of as a dominant coweight of $\rm PGL_n$ as in Section \ref{Liethysec}, and let $\alpha$ be a positive coroot of $\rm PGL_n$ such that $\mu+\alpha$ is again a dominant coweight, in keeping with the hypotheses of the preceding Section \ref{IHRslicesec}. Similarly, we return to the notational conventions of Section \ref{Liethysec}, for example letting $G=\GL_N$ throughout.

\subsection{Inverse Hamiltonian reduction for equivariant Slodowy slices}

Recall that the equivariant Slodowy slices $S_{G,\mu}$ were defined in Section \ref{ssec:equiv slodowy} as Hamiltonian reductions of $T^*\GL_N$ with respect to the left moment map, so that the right moment map $\mu_r^\#$ descends to each $S_{G,\mu}$. Similarly, the equivariant slices retain a residual group action on the left after Hamiltonian reduction with respect to the left moment map restricted to $N_{\chi_\mu}$, and the group $\IG_\alpha$ which acts on $\Gr_{\mu}$ is identified with some subgroup of this residual action, so that in particular it commutes with the right $G$ action. To see this, note that the action of $\IG_\alpha$ is always contained within the upper $n\times n$ block of a matrix realization of $ \Gr^0_{-\alpha_\mu}$, and so the leaves the lower $n\times n$ block invariant. 

The above discussion leads us to the following key result:

\begin{lem}\label{lem:right G equiv}
The maps induced by multiplication of generalized slices in the affine Grassmannian,
$$
	\mathbf{m}:T^*\IG_\alpha^\textup{loc} \times S_{G,\mu+\alpha} \hookrightarrow S_{G,\mu}~,
$$
are strongly $G$-equivariant with respect to the right $G$-actions on $S_{G,\mu+\alpha}$ and $ S_{G,\mu}$, respectively, where we let $T^*\IG_\alpha^\textup{loc} =\IG_\alpha \times \open{\Gf^*}_\alpha\cong \Gr_{-\alpha}^0$.
\end{lem}
\begin{proof}
	Note that the moment maps for the right $G$-action are (coefficients of) matrix elements in the lower $(n-1)\times(n-1)$ block of a matrix realization $x\in \Gr^0_{-\alpha_\mu}\cong S_{G,\mu}$. The action of $\IG_\alpha$ and the multiplication map only sees the top $(\ell(\alpha)+1)\times (\ell(\alpha)+1)$ block of $x$, which does not touch the lower $(n-1)\times (n-1)$ block. So the moment map for the right $G$-action is unchanged under multiplication and the multiplication map is strongly $G$-equivariant.
\end{proof}

\begin{rem}
 The multiplication maps can be made to be $\IG_\hbar$-equivariant where $\IG_\hbar$ acts on the equivariant Slodowy slices by some Kazhdan action, \ie, when restricted to functions on $S_\mu$, the grading is positive. The induced grading on $\Gr^0_{-\alpha}$ is not the usual $\IG_\hbar$-action that scales the cotangent fibre with unit weight. This follows from Lemma \ref{lem:multiplication is hbar equiv}. We can untwist by modifying the action of the maximal torus.
\end{rem}



Applying the results of the previous section, we obtain the following corollary, which is essentially the statement of Hamiltonian reduction by stages for equivariant Slodowy slices $S_{G,\mu}$ and in turn the usual Slodowy slices $S_\mu$ in type $A$; the latter was established previously in general type in \cite{Genra2024:finitestages}.

\begin{cor}\label{cor:reduction by stages for Slodowy}
 The equivariant Slodowy slice, $S_{G,\mu}$, has a Hamiltonian action of $\IG_\alpha$ and
	\begin{equation}
		S_{G,\mu}\red{\chi_\alpha} \IG_\alpha = S_{G,\mu+\alpha}~,
	\end{equation}
	as symplectic varieties. Moreover, $S_{\mu}$ has a Hamiltonian action of $\IG_\alpha$ such that
	\begin{equation}
	S_{\mu} \red{\chi_\alpha} \IG_\alpha = S_{\mu+\alpha}~,
	\end{equation}
	as Poisson varieties.
\end{cor}
\begin{proof}
The first result is an immediate application of Corollary \ref{cor:general inverse reduction 2} and thus Theorem \ref{thm:general reduction by stages} and the second follows by passing to the quotient by the right $G$-action, which---by Lemma \ref{lem:right G equiv}--- the moment map is equivariant with respect to, and commutes with the action of $\IG_\alpha$.
\end{proof}

\begin{cor}\label{cor:ihr for Slodowy slices}
 The tuple, $(\open{\Gf^*}_\alpha,\mathbf{m}) $ is an inverse Hamiltonian reduction for $(S_{G,\mu},\IG_\alpha,\Phi_\alpha,\chi_\alpha)$.
\end{cor}
This in turn implies that we have an isomorphism of Poisson algebras
\begin{equation}
	\OO(S_{G,\mu})[E_\alpha^{-1}] \cong \OO(S_{G,\mu+\alpha}) \otimes \OO(\Gr^0_{-\alpha})
\end{equation}
where $E_\alpha^{-1}$ denotes the localization at the image of $\bf m$. By taking $G$-invariants, we obtain
\begin{equation}
	\OO(S_{\mu})[E_\alpha^{-1}] \cong \OO(S_{\mu+\alpha}) \otimes \OO(\Gr^0_{-\alpha})~,
\end{equation}
the semiclassical limit of our desired embedding of W-algebras.

\subsection{Existence of strong actions}\label{ssec:existence of strong actions}

We note that from \cite{Losev2007:quant}, the quantizations $\AA_{\hbar,G,\mu}$ are strongly $G$-equivariant for the right $G$-action on $S_{G,\mu}$. However, we must show that the Hamiltonian action of $\IG_\alpha$ quantizes to a strong action of $\IG_\alpha$ on the W-algebra.

\begin{lem}\label{lem:free Ga action}
	The action of $\IG_\alpha$ on $S_{G,\mu}$ is free.
\end{lem}
\begin{proof}
	Note that the action of $\IG_\alpha$ on $S_{G,\mu}$ is the residual action of some unipotent subgroup of $G$ acting, via the left action, on the unreduced $T^*G$. Since this left action is free, this residual action on $S_{G,\mu}$ is free too.
\end{proof}

\begin{lem}\label{lem:smooth moment map}
	The moment map $\Phi_\alpha: S_{G,\mu}\rightarrow \Gf^*_\alpha$ is smooth.
\end{lem}
\begin{proof}
Note that the moment map $\Phi_\alpha$ is the base change of the same-named moment map $\Phi_\alpha: {\rm Gr}_\mu \rightarrow \Gf^*_\alpha$ from Section \ref{ssec:action of translation group}, which is smooth by Lemma \ref{lem:moment map smooth}. Thus, $\Phi_\alpha:S_{G,\mu}\rightarrow \Gf^*_\alpha$ is smooth too.
%
\end{proof}

These lemmas ensure that we can lift the Hamiltonian action.

\begin{prop}\label{prop:equiv W algebra is strong equiv}
 Every, graded quantization of $S_{G,\mu}$ can be lifted to a strongly $\IG_\alpha$-equivariant quantization.	

 In particular, there exists a quantum comoment map $\Phi_{\hbar,\alpha}: Sym(\Gf_\alpha)\Ph \rightarrow \WW_{\hbar,G,\mu}$, which integrates to an action of $\IG_\alpha$. Moreover, every strongly $\IG_\alpha$-equivariant quantization is isomorphic to $\WW_{\hbar,G,\mu}$ and the isomorphism can be chosen to be strongly $\IG_\alpha$-equivariant.
\end{prop}
\begin{proof}
	From Lemmas \ref{lem:free Ga action} and \ref{lem:smooth moment map}, we note that the hypotheses of \ref{prop:BuN3 vanishing of obstruction} are met. Thus, we have the desired lifitng property. The uniqueness follows from the fact that the moduli space of quantizations of $S_{G,\mu}$ is rigid from Lemma \ref{lem:equiv W is unique}.
\end{proof}

As an immediate corollary, we note that $\AA_{\hbar,G,\mu}$ is a strongly $\IG_\alpha$-equivariant quantization of $\OO_{S_{G,\mu}}$. We combine both results on equivariance in the following.
\begin{thm}
	Any graded, strongly $\IG_\alpha\times G$-equivariant quantization of $S_{G,\mu}$ is isomorphic to $\WW_{\hbar,G,\mu}$ and the isomorphism can be chosen to be strongly $\IG_\alpha \times G$-equivariant.
\end{thm}
\begin{proof}
	Note that $\WW_{\hbar,G,\mu}$ is strongly $\IG_\alpha\times G$-equivariant by Proposition \ref{prop:equiv W algebra is strong equiv} and \cite{Losev2007:quant}. Uniqueness follows from \ref{prop:classifying strong equiv} and the fact that $H^2_{\IG_\alpha\times G}(S_{G,\mu}) =0$.
\end{proof}

Note that the restriction $\WW_{\hbar,G,\mu}\hookrightarrow {\bf m}^*(\WW_{\hbar,G,\mu})$ is strongly $\IG_\alpha$-equivariant, since the strong action on ${\bf m}^*(\WW_{\hbar,G,\mu})$ is defined by transport of structure.

Combining these observations, we have a version of reduction by stages for $\hbar$-adic finite W-algebras in type $A$. This result was initially conjectured (in general type) by Morgan in  \cite{Morgan2015:stages} and proven (likewise in general type) by Genra and Juillard in \cite{Genra2024:finitestages}, via homological methods. Though, we note that our hypotheses on the nilpotent orbits are different to the ones in \textit{loc.\ cit.}

\begin{thm}
We have a strongly $G$-equivariant isomorphism of strongly $G$-equivariant quantizations of $S_{G,\mu+\alpha}$:
	\begin{equation}
		\WW_{\hbar,G,\mu}\red{\chi_\alpha} \IG_\alpha \cong \WW_{\hbar,G,\mu+\alpha} 
	\end{equation}
and an isomorphism of associative $\hbar$-adic algebras,
	\begin{equation}
	\WW_{\hbar,\mu}\red{\chi_\alpha}\IG_\alpha \cong \WW_{\hbar,\mu+\alpha} 
	\end{equation}
\end{thm}
\begin{proof}
From Proposition \ref{prop:equiv W algebra is strong equiv}, we note that $\WW_{\hbar,G,\mu}$ has a quantum comoment map and thus we can make sense of the reduction. By Proposition \ref{prop:IHR implies quantisation commutes with reduction}, since $S_{G,\mu}$ admits an IHR, the Hamiltonian reduction of $\WW_{\hbar,G,\mu}$ is a quantization of $S_{G,\mu+\alpha}$. However, from Lemma \ref{lem:equiv W is unique}, there is a unique quantization of $S_{G,\mu+\alpha}$, up to isomorphism, and thus $\WW_{\hbar,G,\mu}\red{\chi_\alpha}\cong \WW_{\hbar,G,\mu+\alpha}$. Taking invariants with respect to the right $G$-action, we have the desired result for the ordinary (non-equivariant) finite W-algebras.
\end{proof}



\subsection{Inverse Hamiltonian reduction for finite W-algebras}

We already have IHR data for the equivariant Slodowy slices, all that remains is to find an isomorphism between $\mathbf{m}^*(\AA_{G,\mu})$ and a split quantization of $\Gr^0_{-\alpha}$ and the reduced slice.

Instead of explicitly constructing such a splitting, we will appeal to the rigidity of quantizations of $\Gr^0_{-\alpha}\times S_{G,\mu+\alpha}$. 

\begin{prop}\label{prop:splitting iso exists}
There exists a strongly $G\times\IG_\alpha$-equivariant isomorphism of quantizations
\[\psi_{\mu,\alpha}: \mathbf{m}^*(\AA_{G,\mu}) \xrightarrow{\sim} \DD_{\hbar,\alpha}^{\rm loc} \boxtimes \AA_{G,\mu+\alpha} \ . \]
\end{prop}
\begin{proof}
Note $H^2_{G\times \IG_\alpha}(\Gr^0_{-\alpha}\times S_{G,\mu+\alpha})=0$, as $G\times \IG_\alpha$ acts freely on $\Gr^0_{-\alpha}\times S_{G,\mu+\alpha}$ with quotient $\open{\Gf^*}_\alpha\times S_\mu$, the latter of which is homotopy equivalent to $S^1$, so that there is a unique quantization up to strongly $G\times \IG_\alpha$-equivariant isomorphism.

\end{proof}

\begin{thm}\label{thm:quantum IHR}
The reduction $\AA_{G,\mu}\red{\chi_\alpha}\IG_\alpha$ admits a quantum inverse Hamiltonian reduction given by $(\open{\Gf^*}_\alpha, {\bf m},\DD^{\rm loc}_{\hbar,\alpha}, \psi_{\mu,\alpha})$.
\end{thm}
\begin{proof}
By Corollary \ref{cor:ihr for Slodowy slices} we see that $((\open{\Gf^*}_\alpha, {\bf m})$ is inverse Hamiltonian reduction data for the equivariant Slodowy slices and by Proposition \ref{prop:splitting iso exists}, $(\DD^{\rm loc}_{\hbar,\alpha}, \psi_{\mu,\alpha})$ completes this to a quantum inverse Hamiltonian reduction.
\end{proof}

\begin{cor}\label{cor:quantum IHR embedding}
There is a strongly $G\times\IG_\alpha$-equivariant embedding of $\hbar$-adic associative algebras
	\begin{equation}
		\psi_{\mu,\alpha}: \WW_{\hbar,G,\mu}\hookrightarrow D^{\rm loc}_{\hbar,\alpha}\otimes W_{\hbar,G,\mu+\alpha} \ .
	\end{equation}
\end{cor}
\begin{proof}
	Follows from Theorem \ref{thm:quantum IHR} by taking global sections. The embedding of ordinary W-algebras follows from taking $G$-invariants.
\end{proof}

\begin{cor}There is a strongly $\IG_\alpha$-equivariant embedding of $\hbar$-adic associative algebras
	\begin{equation}
	\psi_{\mu,\alpha}: \WW_{\hbar,\mu}\hookrightarrow D^{\rm loc}_{\hbar,\alpha}\otimes W_{\hbar,\mu+\alpha} \ .
	\end{equation}
\end{cor}
\begin{proof}
	This follows from the preceding corollary by taking invariants with respect to the right $G$ action.
\end{proof}

 


To recover the usual non $\hbar$-adic W-algebras, we use a trick of Losev's from \cite{Losev2007:quant}.
\begin{dfn}
Given a  graded quantization $A$ that is weakly equivariant for some reductive group $H$, we write 
\begin{equation}
	A_{H\times \IG_\hbar-{\rm fin}} \coloneqq \{a\in A\,|\, a \text{ is in a finite dimensional } H\times \IG_\hbar \text{ representation }\}~,
\end{equation}
\end{dfn}
Recalling Section 3.1 of \cite{Losev2007:quant}, we note that we can recover the filtered equivariant W-algebra as
\begin{equation}
 \WW_{G,\mu} = (\WW_{\hbar,G,\mu})_{G\times \IG_\hbar-{\rm fin}} / (\hbar-1)(\WW_{\hbar,G,\mu})_{G\times \IG_\hbar-{\rm fin}} 
\end{equation}
and, moreover, we can recover the usual finite W-algebras by taking $G$-invariants:
\begin{equation}
	\WW_{\mu} \cong (\WW_{G,\mu})^G~.
\end{equation}

Note that the modified $\IG_\hbar$-grading on $D^{\rm loc}_{\hbar,\alpha}$ is not bounded below and the $G$-action on it is trivial, thus $(D^{\rm loc}_{\hbar,\alpha})_{G\times \IG_\hbar-{\rm fin}}/(\hbar-1)(D^{\rm loc}_{\hbar,\alpha})_{G\times \IG_\hbar-{\rm fin}}$ will contain infinite linear combinations. More specifically, we write $D^{\rm loc}_\alpha$ for the associative $\ik$-algebra generated by $(p,e^\pm,b_i,g_i)$, subject to the relations
\begin{equation}
\begin{split}
	e^+ e^- &=  e^- e^+ = 1\\
	[p,e^\pm] = \pm e^\pm ~&,  \qquad [b_i,g_i] = 1~.
\end{split}
\end{equation}
Then, $(D^{\rm loc}_{\hbar,\alpha})_{G\times \IG_\hbar-{\rm fin}}/(\hbar-1)(D^{\rm loc}_{\hbar,\alpha})_{G\times \IG_\hbar-{\rm fin}}$ is isomorphic to some completion of $D_\alpha^{\rm loc}$, call it $\hat{D}_\alpha^{\rm loc}$.
\begin{thm}
	We have a strongly $G\times \IG_\alpha$ equivariant embedding of associative algebras over $\ik$,
\begin{equation}
	\WW_{G,\mu}\hookrightarrow \WW_{G,\mu+\alpha} \otimes \hat{D}^{\rm loc}_{\alpha}~.
\end{equation}
\end{thm}
\begin{proof}
Since the embedding of Corollary \ref{cor:quantum IHR embedding} is strongly $G\times \IG_\alpha$-equivariant and graded, we may restrict to the $G\times\IG_\hbar$-finite subalgebras on both sides. We obtain the desired result by setting $\hbar=1$.
\end{proof}
\subsection{Inverse Hamiltonian reduction for Harish Chandra bimodules}

Let $\rm HC_\hbar$ denote the category of $\hbar$-adic Harish Chandra bimodules for $G=\GL_N$. An object $M\in {\rm HC}_\hbar$ is a graded $U_\hbar(\gf)$-module, equipped with an algebraic action $\rho$ of $G$, such that the action map $U_\hbar(\gf)\otimes M \rightarrow M$ is $G$-equivariant. The category $\rm HC_\hbar$ is symmetric monoidal with the tensor product over $\ik\Ph$ and the diagonal coproduct on $G$ and $U_\hbar(\gf)$. 

We have a monoidal functor $\rm HC_{\hbar} \rightarrow HC_0= \OO(\gf^*)-{\rm Mod}\big( {\rm Rep}\,G\big)\cong {\rm QCoh}_G(\gf^*)$ by setting $\hbar=0$, called the classical limit.

Note that the equivariant W-algebras $\WW_{\hbar,G,\mu}$ are objects in $\rm HC_\hbar$ for any partition $\mu$, with the right $G$-action and the action of $U_\hbar(\gf)$ coming from the right comoment map.

\begin{rem}\label{rem:left right swap on diff ops}
The $\hbar$-adic differential operators, $D_\hbar(G)$, has two commuting actions of $G$ and two commuting comoment maps $\Phi^\#_{L,R}: U_\hbar(\gf)\rightarrow D_\hbar(G)$ and thus is an object in $\rm HC_\hbar$ in two \textit{a priori} non-equivalent ways. However both actions are equivalent and so both ${\rm HC_\hbar}$ structures on $D_\hbar(G)$ are isomorphic. Thus for any $M\in \rm HC_\hbar$,  $M\otimes D_\hbar(G)$ has two equivalent $\rm HC_\hbar$ structures, where $U(\gf)$ and $G$ act diagonally on $M$ and $D_\alpha(G)$ via either the left or right comoment maps on $D_\alpha(G)$.
\end{rem}




Let $\chi\in \IO_\mu$ be some representative in the nilpotent orbit defined by $\mu$. Recall the grading on $\gf$ from Section \ref{ssec:nilpotent orbits} and write $\gf_{\chi,<0}$ for the subalgebra of $\gf$ that is strictly negatively graded. 

If $M\in \rm HC_\hbar$, then we have a natural action of $U_\hbar(\gf_{\chi_\mu,<0})$ coming from restriction. We write $M_{\chi_\mu}$ for the maximal $U_\hbar(\gf_{\chi,<0})$ submodule generated by  $\{ (x- \chi(x)\hbar)m \in M\, |\, x\in \gf_{\chi_\mu,<0}, m \in M\}$. We define the reduction of $M$ at $\chi$ as 
\begin{equation}
	M\red{\chi_\mu}N_\chi \coloneqq \big( M \big/ M_{\chi_\mu}\big)^G\cong (M\tensor{U_\hbar(\gf_{\chi_\mu,<0})}\ik\Ph_\chi)^{N_{\chi_\mu}}~,
\end{equation}
where $\gf_{\chi,<0}$ acts on $\ik_\chi$ via $\chi_\mu$. The reduction $M\red{\chi_\mu} N_{\chi_\mu}$ is naturally a module for the W-algebra $\WW_{\hbar,\mu}$. 
Thus for every partition $\mu$, we have a functor.
\begin{equation}
\begin{split}
	{\rm KW}_\mu:{\rm HC}_\hbar &\rightarrow \WW_{\hbar,\mu}-{\rm Mod}\\
	M &\mapsto M\red{\chi_\mu} N_{\chi_\mu}
\end{split}
\end{equation}
In the classical limit, we note that 
\begin{equation}
	M\red{\chi_\mu}N_{\chi_\mu}|_{\hbar=0} \cong M|_{\hbar=0} \tensor{\OO(\gf^*)} \OO(S_\mu)
\end{equation}

The associated restriction functor induced by the embedding $\WW_{\hbar,\mu}\hookrightarrow \WW_{\hbar,\mu+\alpha}\otimes D_{\hbar,\alpha}^{\rm loc}$ of Corollary \ref{cor:quantum IHR embedding} is the following
\begin{cor}
There exists a natural functor
\begin{equation}
\begin{split}
	\rm{IHR_\alpha}:\WW_{\hbar,\mu+\alpha} - {\rm Mod} &\rightarrow \WW_{\hbar,\mu}-{\rm Mod}\\
	M &\mapsto  {\rm Res}^{\WW_{\hbar,\mu}}_{\WW_{\hbar,\mu+\alpha}\otimes D^{\rm loc}_{\hbar,\alpha}}(M\otimes D^{\rm loc}_{\hbar,\alpha})
\end{split}
\end{equation}
\end{cor}

If $M\in {\rm HC}_\hbar$, then we may also compute its reduction with respect to the full $\gf$ action:
\begin{equation}
M\red{0}G \coloneqq \big(  M \big/ U_\hbar (\gf)M \big)^G\cong \bigg( M\tensor{U_\hbar(\gf)} \ik\Ph\bigg)^G~,
\end{equation}
In the classical limit $M\red{0} G |_{\hbar=0} \cong (M|_{\hbar=0} \tensor{\OO(\gf^*)}\ik)^G$.

\begin{prop}\label{prop:diagonal BRST is DS}
For any $M\in {\rm HC}$, 
\begin{equation}
	\big(M\otimes \WW_{\hbar,G,\mu} \big)\red{0} G \cong M\red{\chi_\mu} N_{\chi_\mu}~.
\end{equation}
as $\WW_{\hbar,\mu}$-modules, where on the left $\WW_{\hbar,\mu}$ acts solely on $\WW_{\hbar,G,\mu}$.
\end{prop}
\begin{proof}
This Proposition is essentially an associative algebra version of Theorem 6.5 in \cite{Arakawa:2018egx} and in some sense follows from it. Therefore, we paraphrase the proof of \textit{loc. cit.}, suppressing some homological details for brevity.

By taking the classical limit, we have the following identifications
\begin{equation}\label{eq:classical limit of diagonal gluing}
	\big(M\otimes \WW_{\hbar,G,\mu} \big)\red{0} G|_{\hbar=0} \cong \big((M|_{\hbar=0} \otimes \OO(S_{G,\mu}))\tensor{ \OO(\gf^*)} \ik\big)^G\cong (M|_{\hbar=0} \tensor{\OO(\gf^*)} \OO(S_\mu)) \cong (M\red{\chi_\mu} N_{\chi_\mu})|_{\hbar=0}~,
\end{equation}
where we have used the fact that $\OO(S_{G,\mu}$ is free as a $\OO(\gf^*)$ module.

The equivariant W-algebras are defined by reducing $D_\hbar(G)$ and so $M\otimes \WW_{\hbar,G,\mu}\cong (M\otimes D_{\hbar}(G))\red{\chi_\mu}N_{\chi_\mu}$ as objects in $\rm HC_\hbar$, where the tensor product is using the right action on $D_\hbar(G)$ and the reduction is performed with respect to the left action. 

Therefore, $(M\otimes \WW_{\hbar,G,\mu})\red{0} G\cong \big((M\otimes D_\hbar(G))\red{\chi_\mu}N_{\chi_\mu}\big) \red{0}G$, where we remark again that the reduction along the diagonal $G$ action uses the right comoment map of $D_\hbar(G)$. By a standard spectral sequence argument we note that
\begin{equation}
	\big((M\otimes D_\hbar(G))\red{\chi_\mu}N_{\chi_\mu}\big) \red{0}G \cong \big(M\otimes D_{\hbar}(G)\big) \red{(\chi_\mu,0)}( N_{\chi_\mu}\times G)~.
\end{equation}
In other words, the two step reduction is isomorphic to reducing at once. But, by Remark \ref{rem:left right swap on diff ops} we can swap to the other $\rm HC_\hbar$ structure on $M\otimes D_{\hbar}(G)$, \ie, the diagonal action uses the left comoment map on $D_\hbar(G)$.

Consider the morphism, $M\red{\chi_\mu}N_{\chi_\mu}\rightarrow \big(M\otimes D_{\hbar(G)}\big) \red{(\chi_\mu,0)}( N_{\chi_\mu}\times G)$, that sends a class $[m]\in (M/M_{\chi_\mu})^{N_\chi}$ to $[m\otimes 1] \in ((M\otimes D_\hbar(G))/U_\hbar(\gf)(M\otimes D_\hbar(G))_{\chi_\mu})^{N_{\chi_\mu}\times G}$.


In the classical limit, this induces a morphism 
$	(M\red{\chi_\mu}N_{\chi_\mu})|_{\hbar=0} \rightarrow \big(M\otimes D_{\hbar}(G)\big) \red{(\chi_\mu,0)}( N_{\chi_\mu}\times G)|_{\hbar=0}$. However, by \eqref{eq:classical limit of diagonal gluing}, this is an isomorphism and so the morphism $M\red{\chi_\mu}N_{\chi_\mu}\rightarrow \big(M\otimes D_{\hbar(G)}\big) \red{(\chi_\mu,0)}( N_{\chi_\mu}\times G)$ must be too. Thus, we may conclude that 
$$
	M\red{\chi_\mu} N_{\chi_\mu} \cong (M\otimes \WW_{\hbar,G,\mu})\red{G}~.
$$
\end{proof}

\begin{thm}
We have natural embeddings of $\WW_{\hbar,\mu}$ modules
\begin{equation}
M\red{\chi_\mu}N_{\chi_\mu} \hookrightarrow	{\rm IHR}_\alpha(M\red{\chi_{\mu+\alpha}}N_{\chi_{\mu+\alpha}})= M\red{\chi_{\mu+\alpha}}N_{\chi_{\mu+\alpha}}\otimes D^{\rm loc}_{\hbar,\alpha}
\end{equation}
\end{thm}
\begin{proof}
	By Proposition \ref{prop:diagonal BRST is DS}, $M\red{\chi_\mu}N_{\chi_\mu}$ is isomorphic to $\big(M\otimes \WW_{\hbar,G,\mu}\big)\red{0}G$. From Corollary \ref{cor:quantum IHR embedding}, we have a strongly $G$-equivariant embedding $\WW_{\hbar,G,\mu}\hookrightarrow \WW_{\hbar,G,\mu+\alpha}\otimes D^{\rm loc}_{\hbar,\alpha}$ which induces a map between the reductions
	\begin{equation}
		\big(M\otimes \WW_{\hbar,G,\mu} \big)\red{0} G \hookrightarrow \big(M\otimes \WW_{\hbar,G,\mu+\alpha}\big)\red{0}G \otimes D_{\hbar,\alpha}^{\rm loc}~,
	\end{equation}
	which by Proposition \ref{prop:diagonal BRST is DS} gives the desired result.
\end{proof}

Suppose now that $A$ is an associative algebra object in $\rm HC_\hbar$. By unraveling definitions, we note that this is essentially the same as an associative $\ik\Ph$-algebra, $A$, flat over $\ik\Ph$ and complete in the $\hbar$-adic topology, and equipped with a comoment map, \ie, a map of associative algebras $\Phi^\#:U_\hbar(\gf) \rightarrow A$ and an action of $G$ by algebra automorphisms such that the differential of the $G$ action agrees with the adjoint action from the comoment map. The preceding theorem also applies in this setting:

\begin{cor}
Let $A$ be an associative algebra in $\rm HC_\hbar$. There exist embeddings of $\ik\Ph$-algebras
\begin{equation}
	A\red{\chi_\mu} N_{\chi_\mu} \hookrightarrow A\red{\chi_{\mu+\alpha}}N_{\chi_{\mu+\alpha}} \otimes D^{\rm loc}_{\hbar,\alpha} 
\end{equation}
for each compatible $\mu,\alpha$ as above.
\end{cor}

In particular, we can apply this in the example of quantizations of the Moore--Tachikawa varieties, recalled in Definition \ref{MTdefn}. In particular, by the quantum analogue of Proposition \ref{MTglueprop}, which is stated in Remark 5.22 of \cite{Braverman:2017ofm}, we have:

\begin{cor} There is an embedding of associative algebras over $\ik \Ph$
\[  \ik_\hbar[ W_{\mu^1,...,\mu^{b-1}, \mu} (\SL_N)] \into  \ik_\hbar[ W_{\mu^1,...,\mu^{b-1}, \mu+\alpha} (\SL_N)]  \otimes D^{\rm loc}_{\hbar,\alpha}   \]
for each compatible $\mu,\alpha$ as above.
\end{cor}

In future work, we hope to establish the chiral analogue of this result, which was conjectured in \cite{Beem:2023uni} for $\SL_2$. This was, in fact, the original motivation for the project which eventually became the present work, \cite{BuN2} and \cite{BuN3}.


\bibliographystyle{amsalpha}
\bibliography{refs}

\newcommand{\etalchar}[1]{$^{#1}$}
\providecommand{\bysame}{\leavevmode\hbox to3em{\hrulefill}\thinspace}
\providecommand{\MR}{\relax\ifhmode\unskip\space\fi MR }
\providecommand{\MRhref}[2]{%
  \href{http://www.ams.org/mathscinet-getitem?mr=#1}{#2}
}
\providecommand{\href}[2]{#2}
\begin{thebibliography}{FKMM99}

\bibitem[ACG24]{ACG}
Dra{\v{z}}en Adamovi{\'{c}}, Thomas Creutzig, and Naoki Genra, \emph{{Relaxed and logarithmic modules of $\widehat{\mathfrak{sl}_3}$}}, Mathematische Annalen \textbf{389} (2024), no.~1, 281--324.

\bibitem[ACM11]{Arakawa2011:twist}
Tomoyuki Arakawa, Dmytro Chebotarov, and Fyodor Malikov, \emph{{Algebras of twisted chiral differential operators and affine localization of $\mathfrak{g}$-modules}}, Selecta Mathematica \textbf{17} (2011), no.~1, 1--46.

\bibitem[Ada05]{Adamovic:2004zi}
Dražen Adamović, \emph{A construction of admissible a1(1)-modules of level -43.}, Journal of Pure and Applied Algebra \textbf{196} (2005), no.~2, 119--134.

\bibitem[Ada19]{Adamović2019ER}
Dra{\v{z}}en Adamovi{\'{c}}, \emph{Realizations of simple affine vertex algebras and their modules: the cases $\widehat{sl(2)}$ and $\widehat{osp(1,2)}$}, Communications in Mathematical Physics \textbf{366} (2019), no.~3, 1025--1067.

\bibitem[AKM23]{Arakawa:2023cki}
Tomoyuki Arakawa, Toshiro Kuwabara, and Sven M\"oller, \emph{{Hilbert Schemes of Points in the Plane and Quasi-Lisse Vertex Algebras with $\mathcal{N}=4$ Symmetry}}.

\bibitem[Ara18]{Arakawa:2018egx}
Tomoyuki Arakawa, \emph{{Chiral algebras of class $\mathcal{S}$ and Moore-Tachikawa symplectic varieties}}.

\bibitem[BDG17]{Bullimore:2015Coul}
Mathew Bullimore, Tudor Dimofte, and Davide Gaiotto, \emph{{The Coulomb Branch of 3d ${\mathcal{N}= 4}$ Theories}}, Commun. Math. Phys. \textbf{354} (2017), no.~2, 671--751.

\bibitem[BDM{\etalchar{+}}24]{Beem:2024fom}
Christopher Beem, Anirudh Deb, Mario Martone, Carlo Meneghelli, and Leonardo Rastelli, \emph{{Free field realizations for rank-one SCFTs}}, JHEP \textbf{12} (2024), 004.

\bibitem[BF08]{BezF}
R.~Bezrukavnikov and M.~Finkelberg, \emph{{Equivariant Satake category and Kostant–Whittaker Reduction}}, Moscow Mathematical Journal \textbf{8} (2008), no.~1, 39–72.

\bibitem[BF23]{BeemF}
Chirstopher Beem and Andrea Ferrari, \emph{{Free field realisation of boundary vertex algebras for Abelian gauge theories in three dimensions }}, arXiv:arXiv:2304.11055, 2023.

\bibitem[BFGM02]{Braverman2002:zastava}
A.~Braverman, M.~Finkelberg, D.~Gaitsgory, and I.~Mirkovi{\'{c}}, \emph{{Intersection cohomology of Drinfeld's compactifications}}, Selecta Mathematica \textbf{8} (2002), no.~3, 381--418.

\bibitem[BFN18]{BFN1}
Alexander Braverman, Michael Finkelberg, and Hiraku Nakajima, \emph{{Towards a mathematical definition of coulomb branches of $3$-dimensional $\mathcal{N} = 4$ gauge theories, II}}, Advances in Theoretical and Mathematical Physics \textbf{22} (2018), no.~5, 1071–1147.

\bibitem[BFN19a]{Braverman:2016pwk}
\bysame, \emph{{Coulomb branches of $3d$ $\mathcal{N}=4$ quiver gauge theories and slices in the affine Grassmannian}}, Adv. Theor. Math. Phys. \textbf{23} (2019), 75--166.

\bibitem[BFN19b]{Braverman:2017ofm}
\bysame, \emph{{Ring objects in the equivariant derived Satake category arising from Coulomb branches (with an appendix by Gus Lonergan)}}, Physics \textbf{23} (2019), 253--344.

\bibitem[BK04a]{Bezrukavnikov2004:Fed}
R.~V. Bezrukavnikov and D.~Kaledin, \emph{{Fedosov quantization in algebraic context}}, Mosc. Math. \textbf{4} (2004), 559--592.

\bibitem[BK04b]{BrKl}
Jonathan Brundan and Alexander~S. Kleshchev, \emph{Shifted {Y}angians and finite {W}-algebras}, Advances in Mathematics \textbf{200} (2004), 136--195.

\bibitem[BL95]{BeauL}
Arnaud Beauville and Yves Laszlo, \emph{Un lemme de descente}, C. R. Acad. Sci. Paris S\'er. I Math. \textbf{320} (1995), no.~3, 335--340.

\bibitem[BMR19]{Beem:2019tfp}
Christopher Beem, Carlo Meneghelli, and Leonardo Rastelli, \emph{{Free Field Realizations from the Higgs Branch}}, JHEP \textbf{09} (2019), 058.

\bibitem[BN23]{Beem:2023uni}
Christopher Beem and Sujay Nair, \emph{Free field realisation of the chiral universal centraliser}, Annales Henri Poincar{\'e} \textbf{24} (2023), no.~12, 4343--4404.

\bibitem[BN25a]{BuN2}
Dylan Butson and Sujay Nair, \emph{{Inverse {H}amiltonian reduction for affine W-algebras in type A}}, \emph{in preparation}, 2025.

\bibitem[BN25b]{BuN3}
\bysame, \emph{{On the deformation theory of chiral quantizations}}, \emph{in preparation}, 2025.

\bibitem[BP08]{BielP}
Roger Bielawski and Victor Pidstrygach, \emph{{Gelfand–Zeitlin actions and rational maps}}, Mathematische Zeitschrift \textbf{260} (2008), no.~4, 779–803.

\bibitem[BR23]{BR1}
D.~Butson and M.~Rap\v{c}\'{a}k, \emph{Perverse coherent extensions on {C}alabi-{Y}au threefolds and representations of cohomological {H}all algebras}, arXiv:2309.16582, 2023.

\bibitem[But23]{Butson:2023fcv}
Dylan Butson, \emph{{Vertex algebras from divisors on Calabi-Yau threefolds}}, arXiv:2312.03648, 2023.

\bibitem[CFLN24]{CFLN}
Thomas Creutzig, Justine Fasquel, Andrew~R. Linshaw, and Shigenori Nakatsuka, \emph{{On the structure of W-algebras in type A}}, Dec 2024.

\bibitem[CG10]{Chriss2010}
Neil Chriss and Victor Ginzburg, \emph{Representation theory and complex geometry}, Birkh{\"a}user Boston, Boston, 2010.

\bibitem[dWL83]{deWilde1983}
Marc de~Wilde and Pierre B.~A. Lecomte, \emph{{Existence of star-products and of formal deformations of the Poisson Lie algebra of arbitrary symplectic manifolds}}, Letters in Mathematical Physics \textbf{7} (1983), no.~6, 487--496.

\bibitem[ENST25]{Esposito2025:Equiv}
Chiara Esposito, Ryszard Nest, Jonas Schnitzer, and Boris Tsygan, \emph{{Quantization of the Momentum Map via $\frak{g}$-adapted Formalities}}.

\bibitem[Fed94]{Fedosov1994:quant}
Boris~V. Fedosov, \emph{{A simple geometrical construction of deformation quantization}}, Journal of Differential Geometry \textbf{40} (1994), no.~2, 213 -- 238.

\bibitem[Feh23a]{Feh2}
Zachary Fehily, \emph{{Inverse reduction for hook-type W-algebras}}.

\bibitem[Feh23b]{Feh1}
\bysame, \emph{{Subregular W-algebras of type A}}, Commun. Contemp. Math. \textbf{25} (2023), no.~09, 2250049.

\bibitem[FF96]{FF1}
Boris Feigin and Edward Frenkel, \emph{Integrals of motion and quantum groups}, Lecture Notes in Mathematics (1996), 349–418.

\bibitem[FFFN24]{FFFN}
Justine Fasquel, Zachary Fehily, Ethan Fursman, and Shigenori Nakatsuka, \emph{{Connecting affine $\mathcal{W}$-algebras: A case study on $\mathfrak{sl}_4$}}, Sep 2024.

\bibitem[FKMM99]{Finkelberg1999:MonopoleSymplectic}
Michael Finkelberg, Alexander Kuznetsov, Nikita Markarian, and Ivan Mirkovi{\'{c}}, \emph{{A Note on a Symplectic Structure on the Space of G-Monopoles}}, Communications in Mathematical Physics \textbf{201} (1999), no.~2, 411--421.

\bibitem[FKN25]{FKN}
Justine Fasquel, Vladimir Kovalchuk, and Shigenori Nakatsuka, \emph{{On Virasoro-type reductions and inverse hamiltonian reductions for $W$-algebras and $W_\infty$-algebras}}, Feb 2025.

\bibitem[FKP{\etalchar{+}}18]{FKPRW}
Michael Finkelberg, Joel Kamnitzer, Khoa Pham, Leonid Rybnikov, and Alex Weekes, \emph{{Comultiplication for shifted Yangians and Quantum Open Toda lattice}}, Advances in Mathematics \textbf{327} (2018), 349–389.

\bibitem[FMS86]{FMS}
Daniel Friedan, Emil~J. Martinec, and Stephen~H. Shenker, \emph{{Conformal Invariance, Supersymmetry and String Theory}}, Nucl. Phys. B \textbf{271} (1986), 93--165.

\bibitem[Fur23]{Furihata:2023qzp}
Shun Furihata, \emph{{On the Beem-Nair Conjecture}}.

\bibitem[GG02]{GG}
Wee Gan and Victor Ginzburg, \emph{{Quantization of Slodowy slices}}, International Mathematics Research Notices \textbf{2002} (2002), no.~5, 243–255.

\bibitem[GJ24]{Genra2024:finitestages}
Naoki Genra and Thibault Juillard, \emph{{Reduction by stages for finite W-algebras}}, Mathematische Zeitschrift \textbf{308} (2024), no.~1, 15.

\bibitem[KP21]{Krylov:2021conv}
Vasily Krylov and Ivan Perunov, \emph{{Almost dominant generalized slices and convolution diagrams over them}}, Advances in Mathematics \textbf{392} (2021), 108034.

\bibitem[KPW22]{Kamnitzer:2022ham}
Joel Kamnitzer, Khoa Pham, and Alex Weekes, \emph{{Hamiltonian reduction for affine Grassmannian slices and truncated shifted Yangians}}, Advances in Mathematics \textbf{399} (2022), 108281.

\bibitem[Los07]{Losev2007:quant}
Ivan~V. Losev, \emph{{Quantized symplectic actions and W -algebras}}, Journal of the American Mathematical Society \textbf{23} (2007), 35--59.

\bibitem[Los12]{Losev2012:quantresolv}
Ivan Losev, \emph{{Isomorphisms of quantizations via quantization of resolutions}}, Advances in Mathematics \textbf{231} (2012), no.~3, 1216--1270.

\bibitem[Mor15]{Morgan2015:stages}
Stephen Morgan, \emph{{Quantum Hamiltonian reduction of W-algebras and category O}}, Ph.D. thesis, 2015.

\bibitem[MVK22]{MVK}
Ivan Mirković, Maxim Vybornov, and Vasily Krylov, \emph{{Comparison of quiver varieties, loop grassmannians and nilpotent cones in type A}}, Advances in Mathematics \textbf{407} (2022), 108397.

\bibitem[RW16]{Reichert2016:strongequiv}
Thorsten Reichert and Stefan Waldmann, \emph{{Classification of Equivariant Star Products on Symplectic Manifolds}}, Letters in Mathematical Physics \textbf{106} (2016), no.~5, 675--692.

\bibitem[Sem94]{Sem}
A.~M. Semikhatov, \emph{{Inverting the Hamiltonian reduction in string theory}}, {28th International Symposium on Particle Theory}, 8 1994.

\bibitem[Wak86]{Wak}
Minoru Wakimoto, \emph{{Fock representations of the affine lie algebra A1(1)}}, Commun. Math. Phys. \textbf{104} (1986), 605--609.

\bibitem[WWY20]{WWY}
Ben Webster, Alex Weekes, and Oded Yacobi, \emph{{A Quantum Mirković-Vybornov isomorphism}}, Representation Theory of the American Mathematical Society \textbf{24} (2020), no.~2, 38–84.

\end{thebibliography}

\end{document}